\documentclass[a4paper,10pt]{amsart}
 \usepackage{amsfonts,mathrsfs}
 \usepackage{amsthm}
 \usepackage{amssymb}
 \usepackage{enumerate}
 \usepackage{enumitem}
 \usepackage{tikz-cd}
 \usepackage{ulem}
 \usepackage{stmaryrd}	 
 \usepackage{subcaption}
 \usepackage{booktabs}
 \usepackage{hyperref}
 \usepackage{cleveref}
 \usepackage{nicematrix}
 
\theoremstyle{definition}
\newtheorem{definition}{Definition}[section]
\newtheorem{example}[definition]{Example}
\newtheorem{remark}[definition]{Remark}

\theoremstyle{plain}
\newtheorem{proposition}[definition]{Proposition}
\newtheorem{theorem}[definition]{Theorem}
\newtheorem*{thm*}{Theorem}
\newtheorem{lemma}[definition]{Lemma}

\newtheorem{corollary}[definition]{Corollary}
\newtheorem*{cor*}{Corollary}

\newtheorem*{con*}{Conjecture}

\newtheorem*{verm*}{Vermutung}
\newtheorem{notation}[definition]{Notation}
\usepackage{mathrsfs}
\newcommand{\rank}{\operatorname{rank}} %rank of a free module
\newcommand{\T}{\operatorname{T}}

\newcommand{\GL}{\operatorname{{\mathbf GL}}}

\newcommand{\cF}{{\mathcal F}}

\newcommand{\C}{{\mathbb C}}
\newcommand{\R}{{\mathbb R}}

\newcommand{\Q}{{\mathbb Q}}

\newcommand{\Z}{{\mathbb Z}}

\renewcommand{\T}{{\mathbb T}}
\title[Topology of real algebraic curves near the tropical limit]{Topology of real algebraic curves near the non-singular tropical limit}

\DeclareMathOperator{\Conv}{Conv}

\DeclareMathOperator{\Col}{Col}
\DeclareMathOperator{\Area}{Area}
\DeclareMathOperator{\Circuit}{Circuit}

\DeclareMathOperator{\EvCirc}{EvCirc}
\DeclareMathOperator{\EvFreeCirc}{EvFreeCirc}

\newcommand{\pr}{\mathbb{P}}

\newcommand{\E}{\mathcal{E}}

\renewcommand{\qedsymbol}{\(\blacksquare\)}

\DeclareMathOperator{\Exp}{Exp}
\DeclareMathOperator{\Edge}{Edge}
\DeclareMathOperator{\Vertex}{Vertex}
\DeclareMathOperator{\Adm}{Adm}
\DeclareMathOperator{\Div}{Div}

\DeclareMathOperator{\Max}{Max}
\DeclareMathOperator{\EvFree}{EvFree}

\DeclareMathOperator{\Log}{Log}

\NiceMatrixOptions{nullify-dots}

\author{C\'edric Le Texier}
\address{Universitetet i Oslo, Norway}
\email{cedricle@math.uio.no}
\begin{document}
	
 \subjclass[2010]{Primary: 14T05, 14P25}

\begin{abstract}
In the 1990's, Itenberg and Haas studied the relations between combinatorial data in Viro's patchworking and the topology of the resulting non-singular real algebraic curves in the projective plane.
Using recent results from Renaudineau and Shaw on real algebraic curves near the non-singular tropical limit, we continue the study of Itenberg and Haas inside any non-singular projective toric surface. 
We give new Haas' like criteria for patchworking non-singular real algebraic curves with prescribed number of connected components, in terms of twisted edges on a non-singular tropical curve.
We then obtain several sufficient conditions for patchworking non-singular real algebraic curves with only ovals in their real part.
One of these sufficient conditions does not depend on the ambient toric surface. 
In that case, we count the number of even and odd ovals of those curves in terms of the dual subdivision, and construct some new counter-examples to Ragsdale's conjecture.
\end{abstract}
\maketitle
\tableofcontents

\section{Introduction}

The aim of this article is to study the topology of real algebraic curves near the non-singular tropical limit, in terms of the induced combinatorial data on the tropical limit, continuing the work started by Itenberg \cite{itenberg1995counter} and Haas \cite{haas1997real}.
A \emph{non-singular real algebraic curve} $\mathcal{C}$ in a non-singular projective toric surface $\pr_\Sigma$ is a non-singular algebraic curve with set of real points $\mathcal{C}(\R ) = \mathcal{C} \cap \pr_\Sigma (\R)$, for $\pr_\Sigma (\R )$ non-empty defined by the construction from Gelfand, Kapranov and Zelevinsky \cite[Theorem 11.5.4]{gelfand2014discriminants}.
%Given such a curve $\mathcal{C}$, we denote by $\mathcal{C}(\R )$ the set of real points of $\mathcal{C}$, ie. $\mathcal{C}(\R ) = \mathcal{C} \cap (\R^\times )^2$ (resp. $\mathcal{C}(\R ) = \mathcal{C} \cap \pr_\Sigma (\R)$).
%Here, the real part $\pr_\Sigma (\R )$ is always defined and non-empty, using the construction from .

One of the first results motivating the study of the topology of non-singular real algebraic curves is the \emph{Harnack-Klein inequality} \cite{harnack1876ueber}, \cite{Klei73}, which states that the number of connected components of the real part $\mathcal{C}(\R )$ of a non-singular real algebraic curve $\mathcal{C}$ of genus $g$ is less than or equal to $g+1$.
Moreover, Harnack gave a construction of non-singular real algebraic curves in $\pr^2$ reaching this bound for any degree $d$ \cite{harnack1876ueber}.

\begin{definition}
Let $\mathcal{C}$ be a non-singular real algebraic curve of genus $g$.
We say that $\mathcal{C}$ is a \emph{maximal} curve (or \emph{$M$-curve}) if its real part $\mathcal{C} (\R )$ has $g+1$ connected components.
More generally, we say that $\mathcal{C}$ is a \emph{$(M-r)$-curve} if its real part $\mathcal{C} (\R )$ has $g+1-r$ connected components.
\end{definition}
%The study of real algebraic curves and their topology dates back to the 19th century with the work of Klein \cite{Klei73} and \cite{harnack1876ueber}. 
%One of the first known fact is that a non-singular real algebraic curve has real part homeomorphic to a disjoint union of circles. 
%Then in the case where such a curve is contained in a non-singular real projective surface, this fact gives rise to the notion of ovals. 

Haas (\cite{haas1997real}) gave a necessary and sufficient criterion for a real algebraic curve near the non-singular tropical limit to be maximal, using Viro's combinatorial patchworking \cite{viro1980curves}, \cite{viro1984gluing}, \cite{viro2001dequantization}.
This criterion can be stated in terms of \emph{twisted edges} on a tropical curve (see \Cref{DefTwist+} and \Cref{amoebaedge2+}). 
%We say that a non-singular real tropical curve $C$ of genus $g$ is $(M-r)$ if the real part of $C$ with respect to the chosen real structure has $g+1-r$ connected components. 
%Using Viro's Patchworking Theorem \cite{viro2001dequantization}, this is equivalent to say that the real algebraic curves close to this tropical limit are $(M-r)$. 
%Haas (\cite{haas1997real}) gave a necessary and sufficient criterion  for a real algebraic curve constructed by combinatorial patchworking to be maximal.
Renaudineau and Shaw (\cite{renaudineau2018bounding}) introduced a tool allowing to compute the number of connected components of a real algebraic curve near the non-singular tropical limit in terms of twisted edges on the tropical limit, and recovered Haas result.
Using these results, we introduce a new criterion for a real algebraic curve near the non-singular tropical limit to be a $(M-1)$-curve.
In the following, a \emph{primitive cycle} is a cycle on a tropical curve bounding a connected component of the complement, and an \emph{exposed edge} is an edge of a tropical curve lying in the boundary of an unbounded connected component of the complement.

\begin{theorem}[Theorem \ref{ThM-1}]
\label{IntroThM-1}
Let $C$ be a non-singular tropical curve in a non-singular projective tropical toric surface $\T \pr_\Sigma$, and let $T$ be an admissible set of twisted edges on $C$.
A real algebraic curve $\mathcal{C} \subset \pr_\Sigma$ near the non-singular tropical limit $C$ and inducing the set of twisted edges $T$ is a $(M-1)$-curve if and only if the dual subdivision $\Delta_C$ of $C$ contains a complete subgraph $K_n$ on $1\leq n\leq 4$ vertices satisfying the following conditions.
\begin{enumerate}
\item \label{IntroItem1M-1} The non-exposed twisted edges are those dual to the edges of $K_n$.
\item \label{IntroItem2M-1} Every primitive cycle $\gamma$ in $C$ dual to a vertex of $K_n$ has an odd number of twisted edges.
\item \label{IntroItem3M-1} Every primitive cycle $\gamma$ in $C$ not dual to a vertex of $K_n$ has an even number of twisted edges.
\end{enumerate}
\end{theorem}

We can also use \Cref{IntroThM-1} in order to construct many examples of $(M-r)$ curves for any $r\leq g$, see \Cref{CorM-rNonDiv}.
Indeed, let $C$ be a non-singular tropical curve given as the gluing of $r$ non-singular tropical curves $C_1 ,\ldots , C_r$ at some unbounded edges, such that the set of twisted edges $T_i$ on each $C_i$ satisfy \Cref{IntroThM-1}.
Then a real algebraic curve $\mathcal{C}$ near the non-singular tropical limit $C$ with set of twisted edges $T_1 \sqcup \ldots \sqcup T_r$ will be a $(M-r)$ curve. 
However, we would not be able to construct \emph{dividing curves} by this process.

\begin{definition}
A non-singular real algebraic curve $\mathcal{C}$ is said to be \emph{dividing} if $\mathcal{C}(\C ) \backslash \mathcal{C}(\R )$ is disconnected.
\end{definition}

Since the genus of a Riemann surface $S$ is given as the maximal number of non-intersecting loops on $S$ which does not disconnect $S$, every maximal curve is dividing.
Furthermore, every dividing curve of genus $g$ is a $(M-2s)$-curve, for $0 \leq s \leq \left\lfloor \frac{g}{2} \right\rfloor$.
%The converse does not hold, hence we will say that a non-singular real algebraic curve $\mathcal{C}$ is a dividing $(M-2s)$-dividing if $\mathcal{C}$ is $(M-2s)$ and dividing.
In the process of proving the criterion for maximal curves, Haas obtained a necessary and sufficient condition to construct dividing curves near the non-singular tropical limit \cite{haas1997real}, again via Viro's combinatorial patchworking.
%Translating in terms of twisted edges, we will say that a set of twisted edges $T$ on a non-singular tropical curve $C$ is \emph{dividing} if the real part of $C$ with respect to $T$ is dividing. 
Using Haas dividing criterion in addition to the results of \cite{renaudineau2018bounding}, we give a new criterion for a real algebraic curve near the non-singular tropical limit to be a $(M-2)$-dividing curve, in terms of twisted edges on the tropical limit.

\begin{theorem}[Theorem \ref{ThM-2}]
\label{IntroThM-2}
Let $C$ be a non-singular tropical curve in a non-singular projective tropical toric surface $\T \pr_\Sigma$, and let $T$ be an admissible dividing set of twisted edges on $C$.
A real algebraic curve $\mathcal{C} \subset \pr_\Sigma$ near the non-singular tropical limit $C$ and inducing the set of twisted edges $T$ is a dividing $(M-2)$ curve if and only if the subgraph of the dual subdivision $\Delta_C$ dual to the set of non-exposed twisted edges is either a complete planar bipartite or a complete planar tripartite graph.   
\end{theorem} 

We can also use \Cref{IntroThM-2} in order to construct many examples of dividing $(M-2s)$ curves for any $s\leq \left\lfloor \frac{g}{2} \right\rfloor$, see \Cref{CorM-2sDiv}, in a similar way as we construct $(M-r)$ curves via \Cref{IntroThM-1}.
The next natural question to ask would be about the possible arrangements of connected components in the real part of real algebraic curves near the non-singular tropical limit. 

In its 16th problem, Hilbert \cite{hilbert1900mathematische} asked what are the realisable \emph{isotopy types} for a non-singular real algebraic curve $\mathcal{C}$ of degree 6 in the projective plane $\pr^2$, or said otherwise the realisable homeomorphism types for the pair $(\pr^2 (\R ) , \mathcal{C} (\R ))$.
This problem and its generalisation to higher degrees has been extensively studied during the 20th century \cite{hilbert1891ueber}, \cite{ragsdale1906arrangement}, \cite{petrowsky1938topology}, \cite{gudkov1969arrangement}, \cite{rokhlin1978complex}, \cite{viro1980curves}, \cite{itenberg1993contre}, \cite{orevkov2002classification}. 
In the projective plane case, the classification problem reduces to the realisable arrangements of \emph{ovals} and \emph{pseudo-lines}. 

\begin{definition}
\label{DefOvalGeneral}
Let $\mathcal{C}$ be a non-singular real algebraic curve in a non-singular projective toric surface $\pr_\Sigma$.
A connected component $C_0$ of $\mathcal{C} (\R)$ is called an \emph{oval} if $C_0$ divides $\pr_\Sigma (\R )$ into two connected components.
We say that an oval $C_0$ of $\mathcal{C}(\R )$ is \emph{even} if $C_0$ lies inside an even number of ovals of $\mathcal{C}(\R)$, and we say that $C_0$ is \emph{odd} otherwise. 
\end{definition} 

For $\mathcal{C}$ a non-singular real algebraic curve of even degree in $\pr^2$, the connected components of the real part $\mathcal{C} (\R )$ are all ovals.
In general, for $\pr_\Sigma$ a non-singular projective toric surface distinct from $\pr^2$, we need stronger restrictions in order to have only ovals in the real part of a non-singular real algebraic curve.
For $\mathcal{C} \subset \pr_\Sigma$ a real algebraic curve near the non-singular tropical limit $C$ with set of twisted edges $T$, we obtain that if every twisted edge in $T$ is non-exposed, the real part $\mathcal{C}(\R)$ consists only of ovals, see \Cref{PropTwistBoundOval}.
In particular, that condition does not depend on the ambient non-singular projective toric surface $\pr_\Sigma$.

In the dividing case, we can partition a set of twisted edges into \emph{multi-bridges}, see \Cref{DefMulti+Circuit} and \Cref{PropPartition}.
A multi-bridge $B$ on a non-singular tropical curve $C$ is a twist-admissible set of bounded edges of $C$ which are parallel modulo 2, such that the graph $C\backslash B$ is disconnected and for every proper subset $B' \subset B$, the graph $C\backslash B$ is connected.
These multi-bridges form a spanning family of the $\Z_2$-vector space $\Div (C)$ of dividing configuration of twists, \Cref{PropPartition}.
If we restrict to configuration of twists with only non-exposed twisted edges, we obtain a $\Z_2$-vector subspace $\Circuit (C) \subset \Div (C)$ generated by \emph{circuits} (see \Cref{DefMulti+Circuit}), which are multi-bridges with only non-exposed edges.
An element $T$ of $\Circuit (C)$ satisfies \Cref{PropTwistBoundOval}, hence we want to count the number of even and odd ovals of real algebraic curves near the non-singular tropical limit $C$ with set of twisted edges belonging to $\Circuit (C)$. 

%For $\mathcal{C}$ a non-singular real algebraic curve of even degree in $\pr^2$, denote by $p,n$ the number of even, odd ovals of $\mathcal{C}(\R)$.

Whenever a partition of a set of twisted edges $T$ contains a circuit $T'$ such that its dual $(T')^\vee \subset \Delta_C$ contains an integer point of the dual subdivision $\Delta_C$ of $C$ with both coordinates even, then untwisting the edges of $T'$ cannot decrease the number of even and odd ovals of the real part, see \Cref{PropEvenTwists}.
A similar statement in the maximal case was shown by Haas \cite[Theorem 10.6.0.5]{haas1997real}.
In order to construct non-singular real algebraic curves with many even or many odd ovals, we can then restrict to even-free configurations of twists.
A configuration of twists $T\in \Circuit (C)$ is \emph{even-free} if its the dual set of edges $T^\vee$ does not contain any integer point with both coordinates even in $\Delta_C \cap \Z^2$, for $\Delta$ the Newton polygon of $C$.
Note that the even-free circuits generate again a $\Z_2$-vector subspace $\EvFreeCirc (C) \subset \Circuit (C)$.

Itenberg \cite{itenberg1995counter} introduced a way to count the number of even and odd ovals of maximal curves of even degree in $\pr^2$ near the non-singular tropical limit, whenever the induced set of twisted edges on the tropical limit $C$ can be partitioned into even-free \emph{bridges}, that is multi-bridges with a single edge of $C$.
Haas \cite{haas1997real} extended this result to all maximal curves near the non-singular tropical limit, hence whenever the induced set of twisted edges can be partitioned into even-free bridges and even-free \emph{double-bridges}, that is multi-bridges with exactly two edges of $C$. 
Note that Itenberg and Haas' counts are expressed in terms of \emph{zone decomposition} of a Newton polygon (see \Cref{DefZone}), that is a decomposition of the Newton polygon induced by the set of edges dual to the set of twisted edges.
In the maximal case, these zone decompositions are induced by set of edges dual to bridges and double-bridges.
We extend those counts first to a partial count of even and odd empty ovals in the case when the tropical limit has set of twisted edges $T$ belonging to $\EvFreeCirc (C)$, see \Cref{LemEvenFreeCount}.

In order to control easily the number of connected components, we want to use a configuration of twists in $\EvFreeCirc (C)$ satisfying \Cref{CorM-2sDiv}.
Such a configuration of twists $T$ is dual to a disjoint union of complete bipartite graphs of the form $K_{2,2l}$ and complete tripartite graphs $K_{2,2,2}$, see \Cref{LemCircGraph}.    
We restrict to disjoint unions of complete bipartite graphs of the form $K_{2,2l}$ and obtain the following count of even and odd ovals.

\begin{theorem}[\Cref{ThEvenOddCountBipartite}]
\label{IntroThEvenOddCountBipartite}
Let $C$ be a non-singular tropical curve in $\T \pr_\Sigma$ such that its Newton polygon $\Delta$ (dual to $\Sigma$) has every edge of even lattice length.
Let $T := T_1 \sqcup \ldots \sqcup T_s \in \EvFreeCirc (C)$ be such that all $(T_i)^\vee \subset \Delta_C$ are disjoints and each $(T_i)^\vee$ is a complete bipartite graph of the form $K_{2,2l_i}$.
Let $Y_T := Y_1^T \sqcup Y_0^T$ be the zone partition (see \Cref{DefZone}) of $\Delta$ induced by $T$ such that $Y_1^T$ contains the unique zone $Z^T$ meeting the boundary of $\Delta$.
Let $p_j$ and $n_j$ be the number of even and odd interior integer points in $Y_j^T$.
Then for $\mathcal{C} \subset \pr_\Sigma$ a real algebraic curve near the non-singular tropical limit $C$ with set of twisted edges $T$, the numbers $p$ of even ovals and $n$ of odd ovals of $\mathcal{C}(\R )$ are 
\begin{align*}
p & = n_1 + p_0 + 1 + \sum\limits_{i=1}^s (l_i+1), \\
n & = p_1 + n_0 + \sum\limits_{i=1}^s (l_i-1) .
\end{align*}
\end{theorem}

Let $R(k) = \frac{3k^2-3k}{2}$.
Ragsdale \cite{ragsdale1906arrangement} made the conjecture that a non-singular real algebraic curve $\mathcal{C}$ of degree $2k$ in $\pr^2$ has $p \leq R(k) + 1$ even ovals and $n \leq R(k)$ odd ovals.
Petrowsky \cite{petrowsky1938topology} proved that
\[ |p-n| \leq R(k) + 1    \]  
and that there exist some curves reaching this bound for any degree $2k$.
He then conjectured that $p \leq R(k) + 1$ and $n \leq R(k) +1$. 
Viro disproved Ragsdale's conjecture (but not Petrowsky's one) by constructing maximal curves of degree 8 with $n = R(k)+1$ \cite{viro1980curves}, and Itenberg disproved both Ragsdale and Petrowsky's conjectures by constructing via combinatorial patchworking curves with 
\[ p = R(k) + 1 + \left\lfloor \frac{(k-3)^2+4}{8} \right\rfloor \text{ and } n = R(k) + \left\lfloor \frac{(k-3)^2+4}{8} \right\rfloor . \]
Haas \cite{haas1995multilucarnes} gave another construction of counter-examples, again via combinatorial patchworking, giving rise to 
\[ p = R(k) + 1 + \left\lfloor \frac{k^2 -7k+16}{6} \right\rfloor  \text{ and } n = R(k) + \left\lfloor \frac{k^2 -7k+16}{6} \right\rfloor . \]
Brugall\'{e} \cite{brugalle2006real} constructed a family of non-singular real algebraic curves, with increasing degree $2k$, such that the family has asymptotically maximal number of even ovals with respect to the Harnack-Klein and Petrowsky inequalities, that is
\[  \lim\limits_{k\rightarrow +\infty } \frac{p}{k^2} = \frac{7}{4} . \]
Moreover, Renaudineau gave a tropical construction of a family with asymptotically maximal number of even ovals \cite{renaudineau2017tropical}.
Hence one of the remaining open problem is to determine the maximal number of even ovals constructible in low degree.
Using \Cref{IntroThEvenOddCountBipartite}, we construct some new counter-examples to Ragsdale conjecture.

\begin{theorem}[Theorem \ref{ThCounterRagsdale}]
\label{IntroThRagsdale}
There exists a dividing $(M-2 \left\lfloor \frac{k-3}{2} \right\rfloor )$  non-singular real algebraic curve of degree $2k \geq 10$ in $\pr^2$ with
\[ R (k) + 1 + \frac{k^2 - 5k + s(k)}{6}  \]
even ovals, with $s(k)$ determined by the value of $k$ modulo 6 as follows.
\begin{itemize}
\item If $k = 0 \mod 6$, then $s(k) = 0$.
\item If $k = 1 \mod 6$, then $s(k) = 10$.
\item If $k = 2 \mod 6$, then $s(k) = 8$.
\item If $k = 3,5 \mod 6$, then $s(k) = 6$.
\item If $k = 4 \mod 6$, then $s(k) = 4$.
\end{itemize} 
%For $k\geq 5$, let $x_t = 2k - (4t+3)$ for $t = 0 , \ldots , \left\lfloor \frac{k-5}{2} \right\rfloor $ and let 
%\[ N_i := \# \{ t\in \{ 0 , \ldots , \left\lfloor \frac{k-5}{2} \right\rfloor  \}  ~ | ~ x_t \equiv i \mod 3 \}  . \]
%There exists a non-singular real algebraic curve of degree $2k$ in $\pr^2$ with 
% \[ R(k) + 1 + \frac{k^2 - 7k + 12}{6} + \frac{2}{3} N_1 + \frac{4}{3} N_2  \] 
%even ovals. 
\end{theorem}

For instance, we construct in \Cref{ExDeg14} a non-singular real algebraic curve of degree 14 in $\pr^2$ with 68 even ovals, hence 4 even ovals more than Ragsdale's conjecture (see also \Cref{FigRagsdale14} for the construction and \Cref{RealPart142Article3} for the isotopy type of the real part).
As an intermediate step for the general construction, we construct in \Cref{LemCounterEx} a dividing $(M-2)$ non-singular real algebraic curve of degree $2k\geq 10$ in $\pr^2$ with more even ovals than Ragsdale's conjecture.
As we see in \Cref{ExRagsdale14}, we obtain in degree 14 a dividing $(M-2)$ non-singular real algebraic curve with 67 even ovals, hence 3 even ovals more than Ragsdale's conjecture (see also \Cref{FigCounter14} for the construction).  

The article will be organised as follows.
We begin by recalling standard notions about \emph{tropical curves} in \Cref{Section2}.
We define then in \Cref{Section3} the real part of a tropical curve in terms of \emph{Viro's combinatorial patchworking}, and introduce the notion of \emph{twisted cycle} on a non-singular tropical curve with twisted edges, which allows us to recover the data of a connected component of the real part directly on the tropical curve.
In \Cref{Sec4}, we prove \Cref{IntroThM-1} and \Cref{IntroThM-2} using a reformulation of the results of Renaudineau and Shaw \cite{renaudineau2018bounding} in the case of curves (see \Cref{ThExactSeq} and \Cref{ThConnectHom+}).
In \Cref{SectionAllOvals}, we give sufficient conditions for a non-singular real algebraic curve near the tropical limit to have only ovals in its real part (\Cref{PropNewtonOval} , \Cref{PropTwistBoundOval}).
Finally, in \Cref{SectionEvenOddOval}, we prove the count of even and odd ovals from \Cref{IntroThEvenOddCountBipartite}, justifying first the choice of restrictions to the configurations of twists we consider (\Cref{PropEvenTwists}, \Cref{LemCircGraph}), and then we use the count to construct new counter-examples to Ragsdale's conjecture (\Cref{IntroThRagsdale}).

\noindent \textbf{Acknowledgements.}
The author would like to thank Kris Shaw and Fr\'{e}d\'{e}ric Bihan for their interest, support and useful discussions in this project.
The research of the author is supported by the Trond Mohn Stiftelse (TMS) project ``Algebraic and topological cycles in complex and tropical geometry".

\section{Tropical curves}
\label{Section2}

We recall in this section several definitions in tropical geometry.
For more details and examples, one can read \cite[Section 2]{brugalle2015brief}, \cite[Chapter 1]{itenberg2009tropical}, \cite{maclagan2009introduction}.

\subsection{Tropical curves in $\R^2$}

A \emph{tropical polynomial in two variables} is 
\[ P(x,y) = \max\limits_{(i,j) \in A} (a_{i,j} + ix +jy) ,    \]
where $A$ is a finite subset of $(\Z_{\geq 0})^2$ and the coefficients $a_{i,j}$ are in the tropical semi-ring $\T = (\R \cup \{ -\infty \} , \max , +)$.
Note that we use the $``\max "$ operation, while some references in the literature use the $``\min "$ operation instead.
Thus, a tropical polynomial is a convex piecewise affine function, with corner locus 
\[ V_P = \{ (x_0 , y_0)\in \R^2 ~ | ~ \exists (i,j)\neq (k,l), ~  a_{i,j} + ix_0 + jy_0  =  a_{k,l} +  kx_0 + ly_0  \} .  \]
The set $V_P$ is a 1-dimensional polyhedral complex in $\R^2$.
%, ie. a finite union of possibly infinite straight edges in $\R^2$.

\begin{definition}
%If every vertex $v$ of $V_P$ is $3$-valent and satisfies the \emph{balancing condition} (ie. the outward direction of the edges incident to $v$ sum up to 0), we say that the graph $V_P$ is the \emph{non-singular tropical curve} defined by $P(x,y)$. 
The \emph{weight function} is defined on the edges $e$ of $V_P$ as 
\[ w(e) := \max\limits_{(i,j) , (k,l)} ( \gcd (|i-k|,|j-l|))     \]
for all pairs $(i,j)$ and $(k,l)$ such that the value of $P(x,y)$ on $e$ is given by the corresponding monomials.
The \emph{tropical curve} $C \subset \R^2$ defined by $P(x,y)$ is the 1-dimensional polyhedral complex $V_P$ equipped with the weight function $w$ on the edges.
\end{definition} 

%\begin{definition}
%If every vertex $v$ of $V_P$ is $3$-valent and satisfies the \emph{balancing condition} (ie. the outward direction of the edges incident to $v$ sum up to 0), we say that the polyhedral complex $V_P$ is the \emph{non-singular tropical curve} defined by $P(x,y)$. 
%%The \emph{weight} of an edge of $\widetilde{V}_P$ is defined as 
%%\[  \max\limits_{(i,j) , (k,l)} ( \gcd (|i-k|,|j-l|))     \]
%%for all pairs $(i,j)$ and $(k,l)$ such that the value of $P(x,y)$ on this edge is given by the corresponding monomials.
%%The \emph{tropical curve} defined by $P(x,y)$ is the graph $\widetilde{V}_P$ equipped with this weight function on the edges.
%\end{definition} 

\subsection{Dual Subdivision}

Let $P(x,y) = \max\limits_{i,j \in A} (a_{i,j} +  ix + jy)$ be a tropical polynomial defining a tropical curve $C$.
The \emph{Newton polygon} of $P(x,y)$, denoted by $\Delta (P)$, is the convex lattice polygon given by
\[ \Delta (P) := \Conv \{  (i,j) \in (\Z_{\geq 0})^2 ~ | ~ a_{i,j} \neq -\infty \} \subset \R^2 .   \]
The tropical polynomial $P$ also determines a subdivision of $\Delta (P)$, as follows.
Given $(x_0 , y_0) \in \R^2$, let
\[ \Delta (x_0 ,y_0 ) := \Conv \{ (i,j) \in (\Z_{\geq 0})^2 ~ | ~ P(x_0 , y_0) =  a_{i,j} + i x_0 +j y_0 \} \subset \Delta (P) .  \]
The tropical curve $C$ induces a polyhedral decomposition of $\R^2$, and the polytope $\Delta (x_0 ,y_0 )$ only depends on the cell $F \ni (x_0 ,y_0)$ of the decomposition given by $C$.
Thus, we define the face $F^\vee := \Delta (x_0 ,y_0 )$ dual to the face $F$ of the polyhedral subdivision of $\R^2$ induced by $C$, for $(x_0 , y_0) \in F$.
The \emph{dual subdivision} $\Delta_C$ of a non-singular tropical curve $C$ is the union
\[  \bigcup_F  F^\vee  ,  \]
for $F$ running through the faces of the decomposition of $\R^2$ induced by $C$.

\begin{definition}
We say that a tropical curve $C$ is \emph{non-singular} if its dual subdivision $\Delta_C$ is a triangulation with each 2-dimensional cell of Euclidean area $1/2$. 
\end{definition}

Equivalently, a tropical curve $C \subset \R^2$ with Newton polytope $\Delta$ is non-singular if and only if $C$ has $2 \Area (\Delta )$ vertices.
In particular, if a tropical curve $C \subset \R^2$ is non-singular, every vertex of $C$ is 3-valent and every edge $e$ of $C$ has weight $w(e) = 1$, and all integer points of $\Delta_C$ are vertices in the subdivision.
%\begin{proof}
%Each vertex $v$ of $C$ is dual to a 2-dimensional simplex $T \subset \Delta_C$ if and only if $v$ has 3 incident edges.
%As the vertices of the triangulation are lattice points, each 2-dimensional simplex $T$ (dual to a vertex $v$ of $C$) has Euclidean area $1/2$ if and only if the edges of $T$ have length 1, if and only if the edges incident to $v$ in $C$ have weight 1. 
%\end{proof}

\subsection{Tropical curves as limit of amoebas}

We want to see that any tropical curve can be approximated by algebraic curves in $(\C^\times)^2$, via the logarithm map. 
Consider the following map, where $t\in \R_{>0}$:
\begin{align*}
\Log_t :  (\C^\times )^2 & \rightarrow \R^2 \\
 (z,w) & \mapsto (\log_t |z| , \log_t |w| ) .
\end{align*}

We call the image of a subset $V\subset (\C^\times )^2$ via $\Log_t$ the \emph{amoeba} of $V$ in base $t$ (see for instance \cite{gelfand2014discriminants}).

%\begin{definition}[\cite{gelfand2014discriminants}]
%The \emph{amoeba} (in base $t$) of $V\subset (\C^\times )^2$ is $\Log_t (V)$.
%\end{definition}
\begin{theorem}[\cite{kapranov2000amoebas}]
\label{ThKapMik2}
Let $\mathcal{P}_t (z,w) = \sum_{i,j} \alpha_{i,j}(t) z^i w^j$ be a polynomial whose coefficients are functions $\alpha_{i,j} : \R \rightarrow \C$, and suppose that $\alpha_{i,j}(t) \sim \beta_{i,j} t^{-a_{i,j}}$ when $t \in \R_{> 0}$ tends to $0$, with $\beta_{i,j} \in \C^\times$ and $a_{i,j} \in \T$.
If $\mathcal{C}_t$ denotes the curve in $(\C^\times )^2$ defined by the polynomial $\mathcal{P}_t$, then the amoebas $\Log_t (\mathcal{C}_t )$ converge to the tropical curve defined by the tropical polynomial $P (x,y) = \max_{i,j} (a_{i,j} +ix + jy)$ when $t \in \R_{> 0}$ tends to $0$.
\end{theorem}

%Whenever a family of curves $(\mathcal{C}_t)_t \subset (\C^\times )^2$ converges to a non-singular tropical curve $C$ in the sense of Theorem \ref{ThKapMik2}, we will say that $C$ is the \emph{tropical limit} of $(\mathcal{C}_t)_t$. 

\subsection{Tropical curve in a tropical toric surface}

%Let $\Delta \subset \R^2$ be the Newton polytope of a non-singular tropical curve $C$ in $\R^2 = (\T^\times)^2$, and let $\Sigma$ be the fan dual to $\Delta$, such that $\Sigma$ is simplicial, unimodular, rational and polyhedral.
%A \emph{tropical toric surface} is constructed from a rational polyhedral fan $\Sigma \subset \R^2$ \cite[\S 6.2]{maclagan2009introduction}.
We refer to \cite{maclagan2009introduction} and \cite{mikhalkin2009tropical} for more details on tropical toric varieties.
Let $\Sigma$ be a fan in $\R^2$, and let $\sigma$ be a cone of $\Sigma$.
We denote $U_\sigma := \T^{\dim \sigma} \times \R^{2-\dim \sigma}$.
For $\tau$ a face of $\sigma$, we denote by $\phi_{\sigma ,\tau }$ the embedding $U_\tau \rightarrow U_\sigma$ described in \cite{mikhalkin2009tropical}, and let $U_\sigma^\tau := \phi_{\sigma ,\tau} (U_\tau )$.

\begin{definition}[\cite{mikhalkin2009tropical}]
Let $\Sigma$ be a complete unimodular pointed fan in $\R^2$.
The \emph{non-singular projective tropical toric surface} associated to $\Sigma$ is the topological space 
\[ \T \pr_\Sigma := \bigcup\limits_{\sigma \in \Sigma} U_\sigma / \sim   \]
stratified by the fan $\Sigma$, where $\sim$ is given by $x\sim \phi_{\sigma , \sigma '} (x)$ for all $\sigma , \sigma ' \in \Sigma$ and all $x\in U_{\sigma '}^\tau$. 
\end{definition}

%For $\pr_\Sigma$ the complex toric variety defined by a fan $\Sigma$, we also write $\T \pr_\Sigma =: \T $\pr_\Sigma$_\Sigma$.
The stratification on $\T \pr_\Sigma$ is compatible with the stratification on the usual non-singular projective toric variety $\pr_\Sigma$ defined by the fan $\Sigma$.
In particular, if $\Delta$ is the 2-dimensional polytope dual to the complete unimodular pointed fan $\Sigma$, there is a one-to-one correspondence between the $n$-dimensional strata of $\T \pr_\Sigma$ and the $n$-dimensional strata of $\Delta$, for $0\leq n \leq 2$.
The tropical toric surface $\T \pr_\Sigma$ can be seen as a compactification of the tropical algebraic torus $(\T^\times)^2 = \R^2$, and has maps to affine charts $\T^2$.
 
\begin{definition}
\label{DefCompTropToric}
Let $C'$ be a non-singular tropical curve in $\R^2$ with Newton polygon $\Delta$.
Let $\T \pr_\Sigma$ be the non-singular projective tropical toric surface such that $\Sigma$ is the dual fan of $\Delta$.
We extend $C'$ to a \emph{non-singular tropical curve} $C$ in $\T \pr_\Sigma$ by compactification of the unbounded edges of $C'$ (in particular, the 1-valent vertices created by the compactification lie on the 1-dimensional strata of $\T \pr_\Sigma$).  
\end{definition}

The choice of compactification in \Cref{DefCompTropToric} allows us to work with non-singular tropical curves $C$ \emph{combinatorially ample} in a non-singular projective tropical toric surface $\T \pr_\Sigma$, see \cite[Definition 2.3]{arnal2019lefschetz}.

\begin{figure}
\centering
\begin{subfigure}[t]{0.45\linewidth}
	\centering
	\includegraphics[width=0.7\textwidth]{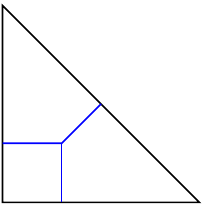}
	\caption{Tropical line in $\T \pr^2$.}
	\label{tropline+}
\end{subfigure}\hfill
\begin{subfigure}[t]{0.45\linewidth}
	\centering
	\includegraphics[width=0.7\textwidth]{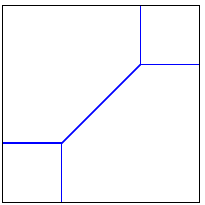}
	\caption{Tropical curve of bidegree $(1,1)$ in $\T \pr^1 \times \T \pr^1$.}
	\label{tropbideg11}
\end{subfigure}
\caption{Example \ref{ExTropLine+}.}
\end{figure}

\begin{example}
\label{ExTropLine+}
Figure \ref{tropline+} represents a tropical line $L$ inside the tropical projective plane $\T \pr^2$. 
%Since the defining polytope $\Delta$ of $\T \pr^2$ is the standard 2-dimensional simplex, we can represent $L$ inside $\Delta$.
Similarly, we have on Figure \ref{tropbideg11} a non-singular tropical curve of bidegree $(1,1)$ on $\T \pr^1 \times \T \pr^1$.
%represented inside the polytope $\Delta$ defining $\T \pr^1 \times \T \pr^1$.
\end{example}

\begin{definition}
Let $C'$ be a non-singular tropical curve in $\R^2$, and let $C$ be its compactification inside the non-singular projective tropical toric surface $\T \pr_\Sigma$, for $\Sigma$ the fan dual to the Newton polygon of $C$.
We say that an edge $e$ of $C$ is \emph{bounded} if $e$ is bounded in $C'$, and \emph{unbounded} if $e$ is unbounded in $C'$.
The set of all edges of $C$ is denoted $\Edge (C)$, and the subset of bounded edges is denoted $\Edge^0 (C)$.
\end{definition}

%For a non-singular tropical curve $C$ in a non-singular projective tropical toric surface $\T \pr_\Sigma$, we say that an edge $e$ of $C$ is \emph{bounded} if its vertices are both 3-valent, and \emph{unbounded} if one of its vertices $v$ is contained in the boundary of the defining polytope $\Delta$ of $\T \pr_\Sigma$ (in particular, $v$ is 1-valent in $C$).

\section{Real part of a non-singular tropical curve}
\label{Section3}

Given a non-singular tropical curve $C$ and some additional data, we want to construct the \emph{real part} of $C$ with respect to the additional data.
One way to do this is via the combinatorial version of \emph{Viro's patchworking} \cite{viro1980curves}.
%, which is a method to construct topological spaces isotopic to the real part of some hypersurface of a toric variety.
We will see that in the case of tropical curves, an equivalent way is given in terms of \emph{amoeba} and \emph{twisted edges}, as seen in \cite{brugalle2015brief}.
We will then see that a set of twisted edges on a tropical curve is equivalent, up to an element of $\Z_2^2$, to the necessary data for combinatorial patchworking.

\subsection{Combinatorial patchworking}

Let $\mathcal{R}^2$ be the group of symmetries in $\R^2$ generated by the reflections with respect to the coordinate hyperplanes.
For $\Z_2:= \Z / 2\Z$, we identify $\mathcal{R}^2$ with $\Z_2^2$ as follows.
If $\varepsilon = (\varepsilon_1 , \varepsilon_2 ) \in \Z_2^2$, then $\varepsilon \in \mathcal{R}^2$ is the symmetry defined by 
\[ \varepsilon (x,y) = ((-1)^{\varepsilon_1} x , (-1)^{\varepsilon_2} y) , \]
for $\varepsilon_1$ and $\varepsilon_2$ seen as integers.
Throughout the article, for $A$ a topological set, we will denote by $A^*$ the disjoint union of symmetric copies
\[ A^* := \bigsqcup\limits_{\varepsilon \in \mathcal{R}^2} \varepsilon (A). \]
%Denote by $\Delta^*$ the union $\bigcup_{\varepsilon \in \mathcal{R}^2} \varepsilon (\Delta )$ of the 4 symmetric copies of $\Delta$.
%If $a\in \Z$, we denote by $\bar{a}$ its image in $\Z_2$. 

\begin{theorem}[\cite{gelfand2014discriminants}]
\label{ThTropToric}
Let $\pr_\Sigma$ be a non-singular projective toric surface, and let $\Delta$ be the 2-dimensional polytope dual to the fan $\Sigma$.
The real part $\pr_\Sigma (\R )$ of $\pr_\Sigma$ is homeomorphic to the quotient of $\Delta^*$ under the following identifications:
if $\Gamma$ is a face of $\Delta^*$, then for all integer vectors $(\alpha_1 , \alpha_2 )$ orthogonal to the primitive integer directions in $\Gamma$, the face $\Gamma$ is identified with $\varepsilon (\Gamma )$, where $\varepsilon = (\bar{\alpha_1},\bar{\alpha_2})$ is the class of $(\alpha_1 , \alpha_2 )$ in $\mathcal{R}^2 \equiv \Z_2^2$.
\end{theorem}

Let $C$ be a non-singular tropical curve in a non-singular projective tropical toric surface $\T \pr_\Sigma$, such that $C$ has Newton polygon $\Delta$ dual to the fan $\Sigma$.
Recall that we have a one-to-one correspondence between the $n$-dimensional strata of $\T \pr_\Sigma$ and the $n$-dimensional strata of $\Delta$.
We denote by $\R \pr_\Sigma$ the quotient of $\T \pr_\Sigma^*$ under the identifications of Theorem \ref{ThTropToric}.
Let $\Delta_C$ be the dual subdivision of $\Delta$ associated to $C$.
%We extend $\Delta_C$ to a subdivision $\Delta_C^*$ of $\Delta^*$ such that $\Delta_C^*$ is invariant with respect to the action of $\mathcal{R}^2$.
Let 
\[ \delta : \Delta_C \cap \Z^2 \rightarrow \{ +1,-1 \} \]
be a distribution of signs on the integer points of the dual subdivision $\Delta_C$ (see for example Figure \ref{SignDistrib+}).
We extend $\delta$ to a distribution of signs on the integer points of the disjoint union of symmetric copies $\Delta_C^*$ following the rule:
for $v = (v_1, v_2 )$ an integer point in $\Delta_C^* \cap \Z^2$ and $\varepsilon = (\varepsilon_1 , \varepsilon_2) \in \mathcal{R}^2$ a symmetry, we have
\[
\delta (\varepsilon (v)) = (-1)^{\varepsilon_1 v_1 + \varepsilon_2 v_2} \delta (v) ,
\]
again with $\varepsilon_1$ and $\varepsilon_2$ seen as integers (see for example the signs on Figure \ref{RealPart+}).
A face of $\Delta_C^*$ is called \emph{non-empty} if at least two of its vertices have opposite signs.
We can now construct the real part of $C$ with respect to $\delta$ (denoted $\R C_\delta$) as follows.
For every edge $e$ of $C$, let $e^\vee$ be its dual edge in $\Delta_C$.
The real part $\R e_\delta$ of $e$ is defined as 
\[ \R e_\delta := \bigcup\limits_{\varepsilon } \varepsilon (e) \subset (\T \pr_\Sigma )^* , \]
for $\varepsilon$ running through the symmetries of $\mathcal{R}^2$ such that $\varepsilon (e^\vee )$ is non-empty.
%Similarly, for every vertex $v= S^\vee$ of $C$ dual to a 2-dimensional cell $S$ of $\Delta_C$, the real part $\R v_\delta$ of $v$ is defined as
%\[ \R v_\delta := \bigcup\limits_{\varepsilon } \varepsilon (v) \] for $\varepsilon$ running through the symmetries of $\mathcal{R}^2$ such that $\varepsilon (S)$ is non-empty.

\begin{definition}
Let $C$ be a non-singular tropical curve in a non-singular projective tropical toric surface $\T \pr_\Sigma$.
The \emph{real part of} $C$ with respect to the distribution of signs $\delta$ is the set
\[ \R C_\delta := \overline{ \bigcup\limits_{e \in \Edge (C)} \R e_\delta } \subset \R \pr_\Sigma . \]
%(where the closure means that we glue along the real part $\R v_\delta$ of vertices $v$ of $C$ dual to a 2-dimensional cell of $\Delta_C$ and along the coordinate hyperplanes of $\T \pr_\Sigma$).
%glued along the real part $\R v_\delta$ of vertices $v$ of $C$ dual to a 2-dimensional cell of $\Delta_C$ and along the coordinate hyperplanes of $\T \pr_\Sigma$. 
\end{definition}

\begin{theorem}[Viro's patchworking theorem \cite{viro2001dequantization}]
\label{ThViro2}
%Let $C, \Delta$ and $\delta$ as above.
Let $(\mathcal{C}_t)_t$ be a family of non-singular real algebraic curves in the non-singular projective toric surface $\pr_\Sigma$ converging to a non-singular tropical curve $C$ (in the sense of \Cref{ThKapMik2}), such that $C$ has Newton polygon dual to the fan $\Sigma$.
Let $(\mathcal{P}_t)_t$ be a family of polynomials defining $(\mathcal{C}_t)_t$ with distribution of signs $\delta$.
%Let $\delta$ be a distribution of signs on the dual subdivision $\Delta_C$, so that there exists a family of polynomials $(\mathcal{P}_t)_t$ defining the family of curves $(\mathcal{C}_t)_t$ and such that the signs of the monomials $P_t$ are the signs induced by $\delta$.
For $t \in \R_{>0}$ small enough, there exists a homeomorphism $\pr_\Sigma (\R ) \rightarrow \R \pr_\Sigma$ mapping the set of real points $\mathcal{C}_t (\R )$ onto $\R C_\delta$.
\end{theorem}

\begin{definition}
If $(\mathcal{C}_t)_t$ is a family of non-singular real algebraic curves satisfying \Cref{ThViro2} for a couple $(C,\delta)$, we say that for $t \in \R_{>0}$ small enough (in the sense of \Cref{ThViro2}), the curve $\mathcal{C}_t$ is near the tropical limit $(C,\delta)$.
\end{definition}

Given a distribution of signs $\delta$ on the dual subdivision $\Delta_C$ of a non-singular tropical curve $C$, \cite{renaudineau2017haas}, \cite{renaudineau2017tropical} and \cite{renaudineau2018bounding} define a real structure directly on $C$. 
We define this real structure as follows:

\begin{definition}
Let $C$ be a non-singular tropical curve.
For $\delta$ a distribution of signs on the dual subdivision $\Delta_C$, a \emph{real phase structure} on $C$ is a collection $\E := \{ \E_e \}_{e\in \Edge (C)} $ of $\Z_2$-affine spaces defined as
\[ \E_e = \{ \varepsilon \in \mathcal{R}^2 ~ | ~ \varepsilon (e) \subset \R e_\delta   \} .  \]
The couple $(C,\E )$ is called a \emph{real tropical curve}.
\end{definition}

\begin{figure}
\centering
\begin{subfigure}[t]{0.3\linewidth}
	\centering
	\includegraphics[width=\textwidth]{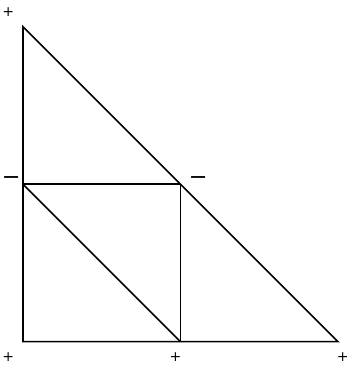}
	\caption{Distribution of signs $\delta$ on $\Delta_C$.}
	\label{SignDistrib+}
\end{subfigure}\hfill
\begin{subfigure}[t]{0.3\linewidth}
	\centering
	\includegraphics[width=\textwidth]{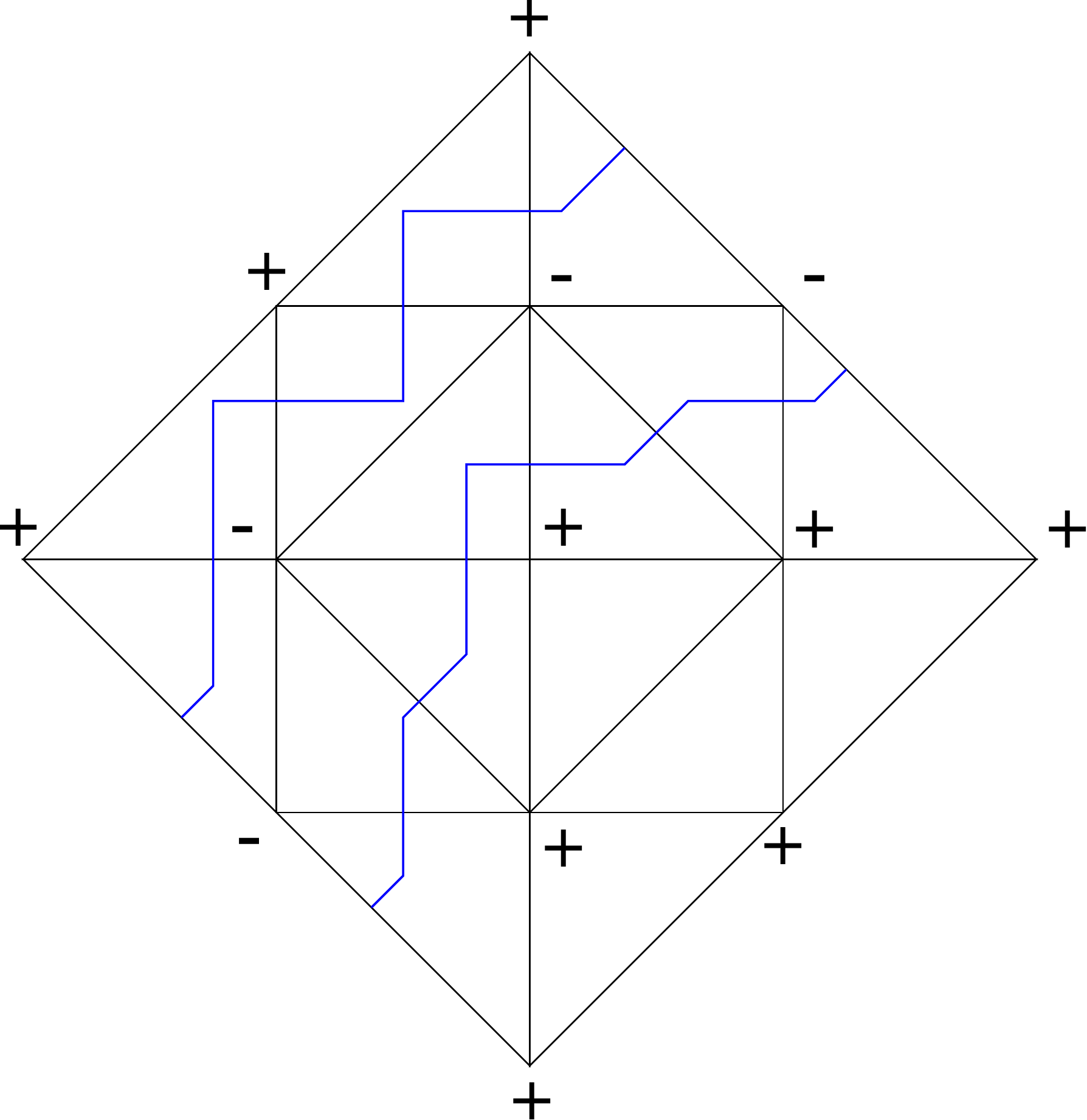}
	\caption{Real part $\R C_\delta$ on $\R \pr^2$.}
	\label{RealPart+}
\end{subfigure}\hfill
\begin{subfigure}[t]{0.3\linewidth}
	\centering
	\includegraphics[width=\textwidth]{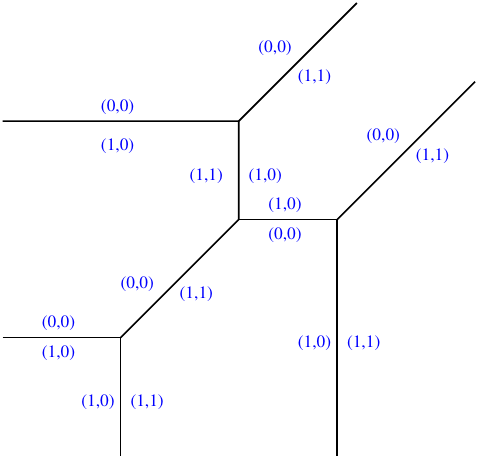}
	\caption{Real phase structure on $C$ (\Cref{ExCombPatch+}).}
	\label{RealPhase+}
\end{subfigure}
\caption{\Cref{ExCombPatch+}.}
\end{figure}

\begin{remark}
\label{RkRealPhase+}
We can define directly a real phase structure $\E$ on a non-singular tropical curve $C$ without defining first a distribution of signs on the dual subdivision, as described in \cite[Definition 3.1]{renaudineau2018bounding}.
Such a real phase structure $\E := \{ \E_e \}_e$ must satisfy the two following properties:
\begin{itemize}
\item for $e$ an edge of $C$ and $\overrightarrow{e} \in \Z_2^2$ its direction modulo 2, we have $\varepsilon + \varepsilon ' = \overrightarrow{e}$ for $\varepsilon , \varepsilon '$ the two elements of $\E_e$ (using the identification $\mathcal{R}^2 \equiv \Z_2^2$);
\item for $v$ a 3-valent vertex of $C$ incident to an edge $e$, for any element $\varepsilon \in \E_e$, there exists a unique edge $e' \neq e$ such that $v$ is incident to $e'$ and $\varepsilon \in \E_{e '}$, with $e$ and $e'$.
\end{itemize}
By \cite[Remark 3.8]{renaudineau2018bounding}, a real phase structure defined this way determines (up to multiplication by $(-1)$) a distribution of signs $\delta$ on the dual subdivision $\Delta_C$.  
\end{remark}

\begin{example}
\label{ExCombPatch+}
Let $C$ be a non-singular tropical curve of degree 2 in $\T \pr^2$ as pictured in Figure \ref{RealPhase+}.
The distribution of signs $\delta$ on the dual subdivision $\Delta_C$ is pictured in Figure \ref{SignDistrib+}.
Extending the distribution of signs to the symmetric copies of $\Delta_C$, we obtain the real part $\R C_\delta$ pictured in color in Figure \ref{RealPart+}.
We can recover the corresponding real phase structure $\E$ on $C$, pictured in color in Figure \ref{RealPhase+}, by looking for each edge $e$ of $C$ its non-empty symmetric copies $\varepsilon (e) , \varepsilon ' (e)$.
Notice that the real phase structure $\E$ can be constructed directly using \Cref{RkRealPhase+}, and we recover (up to multiplication by $(-1)$) the distribution of signs $\delta$ from the data of edges $e$ satisfying $(0,0) \in \E_e$.
\end{example}

\begin{notation}
From now on, we will write equivalently $\R C_\E$ or $\R C_\delta$ for the real part of a non-singular tropical curve $C$, as the distribution of signs $\delta$ and $-\delta$ determine the same real phase structure $\E$, and so the same real part of $C$ (we refer again to \cite[Remark 3.8]{renaudineau2018bounding} for more details).
Similarly, we will say that a non-singular real algebraic curve is \emph{near the tropical limit} $(C,\E )$ if $\delta$ determines the real phase structure $\E$.
%data of a real phase structure on $C$ is equivalent to the data of a distribution of signs on $\Delta_C$.
%(resp. for the real part of an edge $e\in \Edge (C)$, resp. for the real part of a 3-valent vertex $v$ in $C$)
\end{notation}

\subsection{Twisted edges on a tropical curve}

\begin{definition}
\label{DefTwist+}
Let $(C,\E )$ be a non-singular real tropical curve.
%A bounded edge $e$ of $C$ is said to be \emph{non-twisted} if for any $\varepsilon \in \E_e$, the two edges $e_1 , e_2$ of $C$ adjacent to $e$ such that $\varepsilon \in \E_{e_1} , \varepsilon \in \E_{e_2}$ lie on the same side of the affine line containing $e$.
A bounded edge $e$ of $C$ is said to be \emph{twisted} if for any $\varepsilon \in \E_e$, the two edges $e_1 , e_2$ of $C$ adjacent to $e$ such that $\varepsilon \in \E_{e_1}$ and $\varepsilon \in \E_{e_2}$ lie on distinct sides of the affine line containing $e$ (see for instance \Cref{amoebaedge2+}).
Otherwise, the edge $e$ is said to be \emph{non-twisted}, and satisfies the fact that for any $\varepsilon \in \E_e$, the two edges $e_1 , e_2$ of $C$ adjacent to $e$ such that $\varepsilon \in \E_{e_1} , \varepsilon \in \E_{e_2}$ lie on the same side of the affine line containing $e$ (see for instance \Cref{amoebaedge1+}). 
%In particular, for $\varepsilon \in \E_e$, the edges adjacent to $e$ containing $\varepsilon$ in their real phase are on the same side of $e$ if $e$ is not twisted, and on distinct sides if $e$ is twisted.
\end{definition}  

\begin{figure}
\centering
\begin{subfigure}[t]{0.3\linewidth}
\centering
	\includegraphics[width=\textwidth]{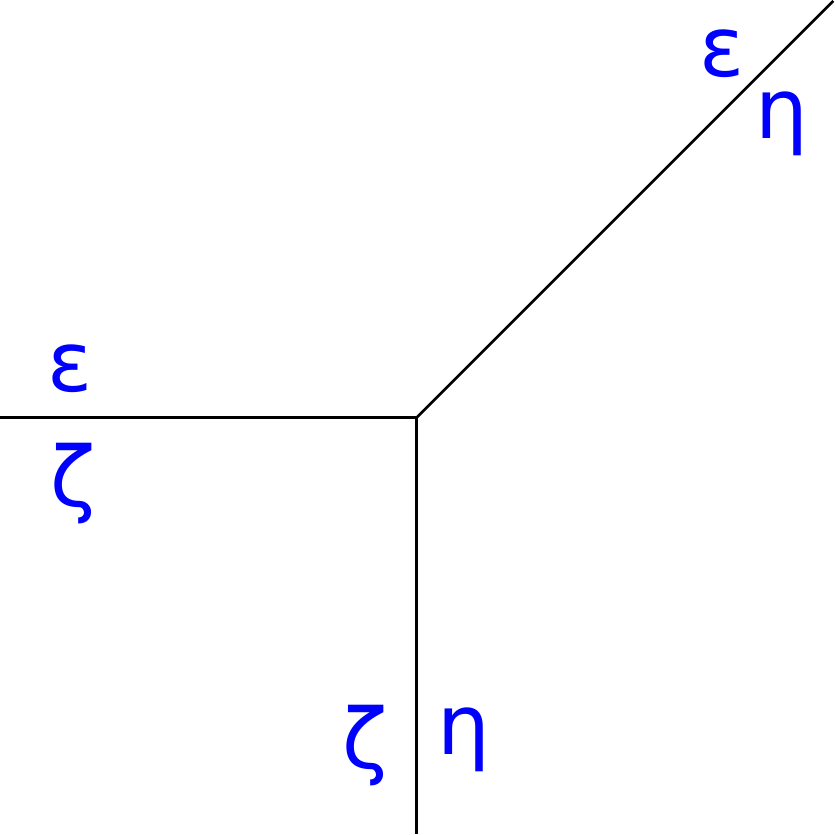}
	\caption{Real phase structure around a vertex.}
	\label{amoebavertex+}
\end{subfigure}\hfill
\begin{subfigure}[t]{0.3\linewidth}
	\centering
	\includegraphics[width=\textwidth]{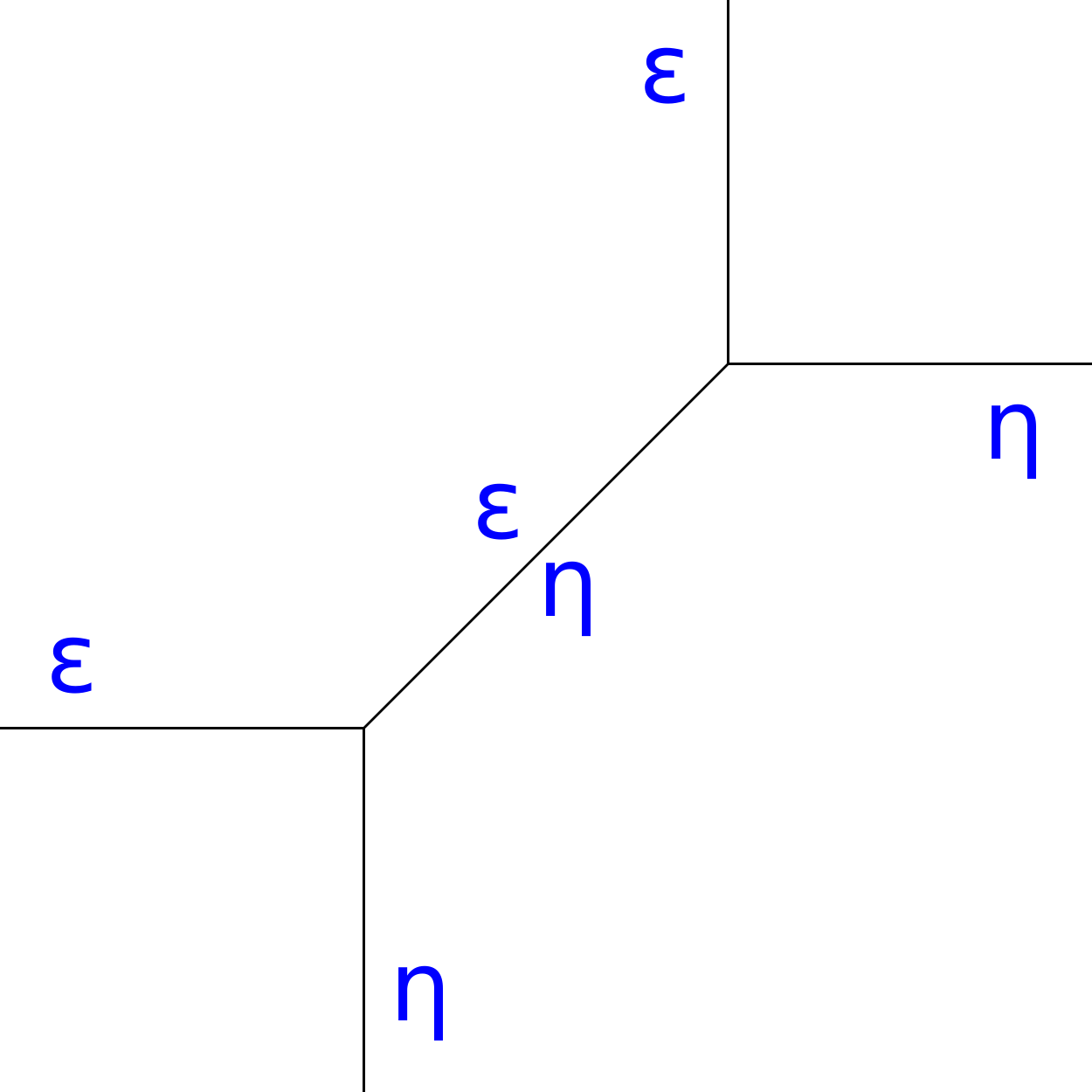}
	\caption{Real phase structure around a non-twisted bounded edge.}
	\label{amoebaedge1+}
\end{subfigure}\hfill
\begin{subfigure}[t]{0.3\linewidth}
	\centering
	\includegraphics[width=\textwidth]{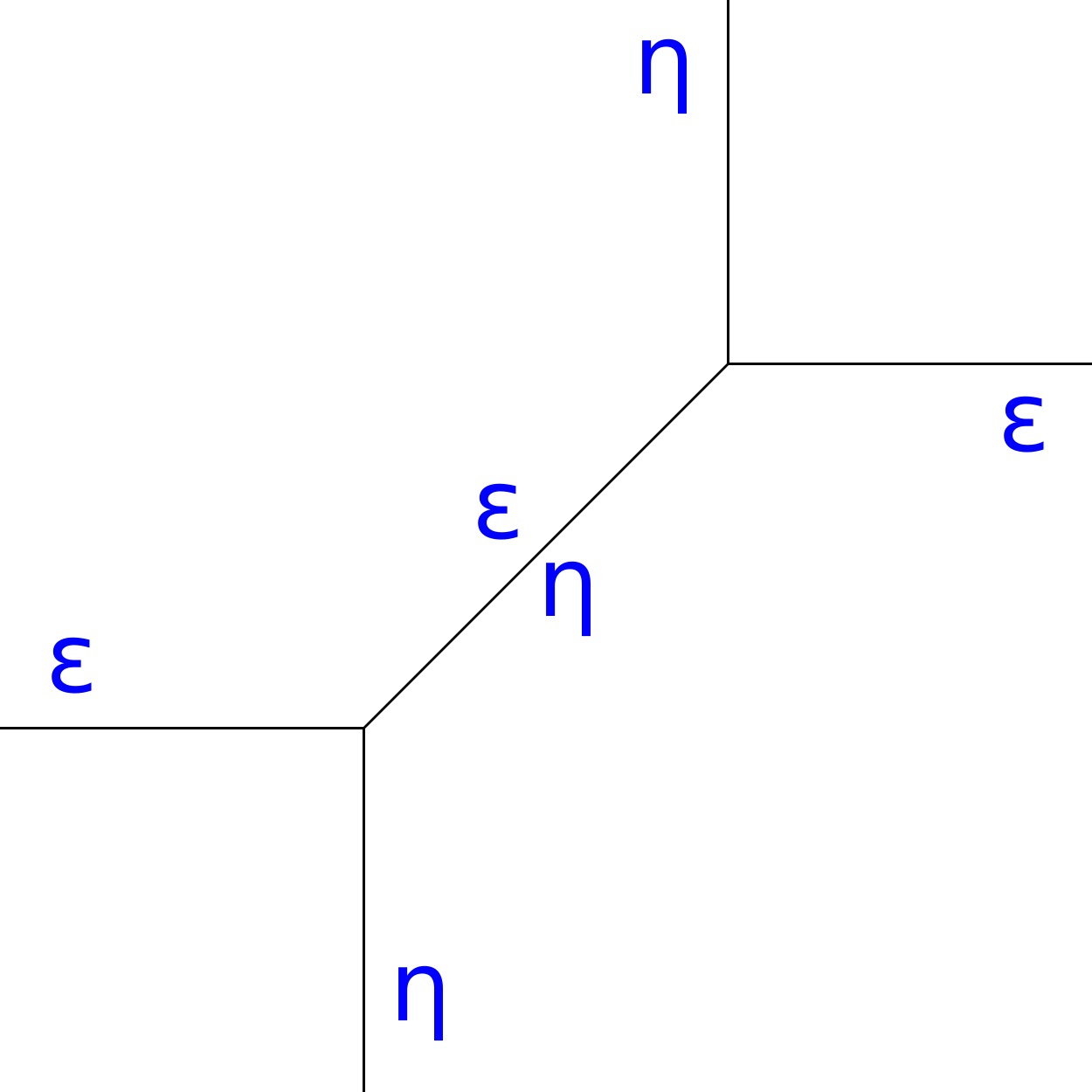}
	\caption{Real phase structure around a twisted bounded edge.}
	\label{amoebaedge2+}
\end{subfigure}
\caption{Properties of real phase structure from \Cref{RkRealPhase+} and \Cref{DefTwist+}.}
\end{figure}

We can define the twisted edges on a non-singular tropical curve in terms of \emph{amoebas}, see for instance \cite[Section 3.2]{brugalle2015brief} and \cite[Section 4.A]{renaudineau2017haas}. 
For an edge $e$ of $C$, we denote by $\overrightarrow{e} \in \Z^2_2$ the direction modulo 2 of $e$.
Not any subset $T$ of $\Edge^0 (C)$ may arise as the set of twisted edges. 
The possible configurations of twists must satisfy the following condition. 

\begin{definition}[\cite{brugalle2015brief},\cite{renaudineau2017haas}]
Let $C$ be a non-singular tropical curve.
Let $T$ be a subset of $\Edge^0 (C)$.
We say that $T$ is \emph{admissible}, or \emph{twist-admissible}, if for any cycle $\gamma$ of $C$, we have 
\begin{equation}
\label{EqAdm+}
\sum_{e \in \gamma \cap T} \overrightarrow{e} = \overrightarrow{0} \mod 2 .
\end{equation}
We denote by $\Adm (C)$ the set of all twist-admissible subsets of bounded edges $T \subset \Edge^0 (C)$.
\end{definition}

The set $\Edge^0 (C)$ of all bounded edges of $C$ can be given a $\Z_2$-vector space structure by identification with $\Z_2^{|\Edge^0 (C) |}$, so that a canonical basis is given by the vectors identified with single bounded edges of $C$.
%Then a subset $T \subset \Edge^0 (C)$ of bounded edges of $C$ can be seen as a vector in the $\Z_2$-vector space $\Edge^0 (C)$.
%Given the canonical basis of $\Edge^0 (C)$, there exists a scalar product 
%\[ \left\langle \cdot , \cdot \right\rangle_0 : \Edge^0 (C) \times \Edge^0 (C) \rightarrow \Z_2 \] 
%on $\Edge^0 (C)$ defined as follows.
%For $e$ a bounded edge of $C$ (hence $e$ is identified with an element of the canonical basis of $\Edge^0 (C)$), we have $\left\langle e , e \right\rangle_0 = 1 \in \Z_2$.
%For $e \neq e'$ two distinct bounded edges of $C$, we have $\left\langle e , e' \right\rangle_0 = \left\langle e' , e \right\rangle_0 = 0 \in \Z_2$. 
%Extending $\left\langle \cdot , \cdot \right\rangle_0$ by bilinearity, we obtain a well-defined scalar product on $\Edge^0 (C)$.

\begin{definition}
Let $T = \{ e_1 , \ldots , e_r \}$ be a subset of bounded edges of $C$.
%We call the set-theoretic description of $T$ a \emph{set of twisted edges} on $C$.
The set $T$ can be seen as the sum $\sum_{i=1}^r e_i \in \Edge^0 (C)$, with $\Edge^0 (C)$ equipped with its $\Z_2$-vector space structure.
We call the space-theoretic description of $T$ a \emph{configuration of twists} on $C$.
\end{definition}

The set $\Adm (C)$ has a $\Z_2$-vector space structure induced from $\Edge^0 (C)$ given by the $2g$ linear conditions from \Cref{EqAdm+}, for $g$ the number of cycles of $C$ bounding a connected component of the complement.

%The set $\LocAdm (C)$ also has a $Z_2$-vector space structure induced from $\Edge^0 (C)$.
%Indeed, the zero element of $\Edge^0 (C)$ (corresponding to the empty set of bounded edges) belong to $\LocAdm (C)$, as \Cref{EqAdm+} is satisfied trivially.
%Moreover, for any $T, T' \in \LocAdm (C)$, the sum $T+T'$ belongs to $\LocAdm (C)$.
%In order to see this, consider   
%In particular, the subsets $\LocAdm (C)$ and $\Adm (C)$ are $\Z_2$-subspaces of $\Edge^0 (C)$, since 
% given by the linear conditions from \Cref{EqAdm+}.
%If we give a $\Z_2$-vector space structure to $\Edge^0 (C)$ by identifying it with $\Z_2^{\# \Edge^0 (C)}$, then the set $\Adm (C)$
%has a $\Z_2$-vector space structure induced from $\Edge^0 (C)$ by the $2g$ linear conditions given by \Cref{EqAdm+}, for $g$ the topological genus of the non-singular tropical curve $C$.

\subsection{Correspondence between signs and twisted edges}

%We can now recall the statement \cite[Theorem 3.4]{brugalle2015brief}, which is a reformulation of Theorem \ref{ThViro2} in terms of twist-admissible sets.
%\begin{theorem}[\cite{brugalle2015brief}]
%Let $C$ be a non-singular tropical curve in $\R^2$ and $T \in \Adm (C)$ an admissible set of twisted edges of $C$.
%There exists a family of real algebraic curve $(\mathcal{C}_t)_{t>1}$ in $(\C^\times )^2$ which converges to $C$ in the sense of Theorem \ref{ThKapMik2} and such that the corresponding set of twisted edges is $T$.
%\end{theorem}

We recall the correspondence between combinatorial patchworking and admissible sets of twisted edges.
Let $C$ be a non-singular tropical curve with Newton polytope $\Delta$. 
For each $e \in \Edge^0 (C)$, denote by $v_1^e$ and $v_2^e$ the two vertices of the edge $e^\vee$ dual to $e$ in the dual subdivision $\Delta_C$. 
The edge $e^\vee$ is contained in exactly two 2-dimensional simplices of $\Delta_C$.
Let $v_3^e$ and $v_4^e$ denote the vertices of these two simplices which are distinct from $v_1^e, v_2^e$.
Let $\delta$ be a distribution of signs on the integer points in $\Delta_C \cap \Z^2$.
%, and let $T \in \Adm (C)$ be the corresponding set of twisted edges on $C$. 
The following proposition is described in \cite[Remark 3.9]{brugalle2015brief}

\begin{proposition}[\cite{brugalle2015brief}]
\label{PropTwistSign+}
Let $e \in \Edge^0 (C)$.
\begin{itemize}
\item If the coordinates modulo 2 of $v_3^e$ and $v_4^e$ are distinct, then the edge $e$ is twisted if and only if $\delta (v_1^e) \delta (v_2^e) \delta (v_3^e) \delta (v_4^e) = +1$. 
\item If the coordinates modulo 2 of $v_3^e$ and $v_4^e$ coincide, then the edge $e$ is twisted if and only if $\delta (v_3^e) \delta (v_4^e) = -1$. 
\end{itemize}    
\end{proposition}

In particular, for every twist-admissible set $T$ of bounded edges on $C$, there exists a distribution of signs $\delta$ on the dual subdivision $\Delta_C$ satisfying \Cref{PropTwistSign+}.
Therefore, by Viro's Patchworking Theorem (\Cref{ThViro2}), there exists a family of real algebraic curve $(\mathcal{C}_t)_{t>1}$ in $(\C^\times )^2$ which converges to $C$ in the sense of \Cref{ThKapMik2} and such that the corresponding set of twisted edges is $T$.
This fact is stated in \cite[Theorem 3.4]{brugalle2015brief} as a reformulation of Viro's Patchworking Theorem. 
%\begin{definition}
%Let $(C,\E )$ be a non-singular real tropical curve in a non-singular projective tropical toric surface $\T \pr_\Sigma$.
%We say that the real phase structure $\E$ \emph{induces} the admissible set of twisted edges $T$ on $C$ if for any distribution of signs $\delta$ determining $\E$, the twisted edges induced from $\delta$ by Proposition \ref{PropTwistSign+} are exactly the bounded edges in $T$.  
%\end{definition}

From Proposition \ref{PropTwistSign+}, we can define a canonical distribution of signs on the dual subdivision $\Delta_C$ (which does not depend on the subdivision) such that no bounded edge of $C$ is twisted.

\begin{definition}[\cite{itenberg1995counter}]
\label{DefHarnackDistrib}
The \emph{Harnack distribution of signs} $\delta_\emptyset$ on the dual subdivision $\Delta_C$ of a non-singular tropical curve $C$ is given as: 
\begin{itemize}
\item for $v= (v_1 , v_2) \in \Delta_C \cap \Z^2$ such that both $v_1$ and $v_2$ are even, we assign $\delta_\emptyset (v) = -1$;
\item for $v= (v_1 , v_2) \in \Delta_C \cap \Z^2$ such that at least one $v_i$ is odd, we assign $\delta_\emptyset (v) = +1$.
\end{itemize}
We say that a distribution of signs $\delta$ on $\Delta_C$ is of \emph{Harnack type} if there exists $\varepsilon \in \mathcal{R}^2$ a symmetry such that for every $v\in \Delta_C \cap \Z^2$, we have $\delta (v) = \delta_\emptyset ( \varepsilon (v))$.
We say that a distribution of signs $\delta$ on $\Delta_C$ is of \emph{inverse Harnack type} if there exists $\varepsilon \in \mathcal{R}^2$ a symmetry such that for every $v\in \Delta_C \cap \Z^2$, we have $\delta (v) = - \delta_\emptyset ( \varepsilon (v))$.
\end{definition}

By Proposition \ref{PropTwistSign+}, every distribution of signs of Harnack type or inverse Harnack type on $\Delta_C$ induces an empty set of twisted edges on $C$. 
Moreover, any distribution of signs inducing an empty set of twisted edges is either of Harnack or inverse Harnack type.
%We denote by $\E_\emptyset$ the real phase structure on a non-singular tropical curve $C$ equivalent to the data of the Harnack distribution of signs $\delta_\emptyset$ on the dual subdivision $\Delta_C$.

\subsection{Real components and twisted cycles}

Given a non-singular real tropical curve $(C,\E )$ in a non-singular projective tropical toric surface $\T \pr_\Sigma$, we want to find a description of the connected components of $\R C_\E$ from the data of $C$ and the admissible set of twisted edges $T$ on $C$ induced by $\E$.
The following construction is inspired from the notion of \emph{left-right cycle} in graph theory, see \cite{godsil2013algebraic}.
Let $C$ be a non-singular tropical curve in a non-singular projective tropical toric surface $\T \pr_\Sigma$ and let $T$ be a subset of bounded edges on $C$.
Choose a starting vertex $v$ on $C$, as well as one of its incident edges $e$.
% (this defines an orientation on $e$) and a side $s$ (left or right) with respect to the orientation on $e$.
Let $S$ be an alternating sequence of vertices and edges of $C$ given by walking along the left side of $C$ from $(v,e)$ until reaching an edge $e' \in T$, then walk along the right side until reaching another edge $e'' \in T$, and so on, until reaching again the vertex $v$ of $C$.
Whenever the walk reaches a 1-valent vertex of $C$, turn around that vertex and continue the walk.

%\begin{remark}
%Each time the process above reaches a 1-valent vertex $v$ of $C$, with incident edge $e$, the next element in the sequence will be the edge $e$ (with opposite orientation), and then the process continue normally (and the side $s$ is preserved). 
%\end{remark}
%
%View each vertex of $C$ as a small disk and each edge of $C$ as a thin strip.
%Select a starting point on $C$ where the side of a strip meet the boundary of a disk. 
%The chosen triple vertex-edge-side is called a \emph{flag}.
%From the flag, walk along the side of the edge, and if the edge is twisted, cross to the opposite side of the edge halfway between its endpoints.
%On reaching the neighboring vertex, walk around the boundary of the disk representing the vertex, leaving the vertex along the side of the edge lying in the same face as the side of the edge you have just arrived on.
%Extend the walk by using the same rules for negotiating edges and vertices. 
\begin{definition}
%A \emph{twisted walk} starting from $(v,e)$ on a non-singular tropical curve $C$ equipped with a subset of bounded edges $T$ is the alternating sequence $S$ of vertices and edges encountered during the process above.
A \emph{closed twisted walk} starting from $(v,e)$ on a non-singular tropical curve $C$ equipped with a subset of bounded edges $T$ is the alternating sequence $S$ of vertices and edges encountered during the process above. 
A \emph{twisted cycle} on $(C,T)$ is an equivalence class of closed twisted walks under cyclic permutation and reversal.
\end{definition}

\begin{example}
\label{ExTwistCycle}
Let $C$ be a non-singular tropical curve of degree 2 in $\T \pr^2$, as represented in black in Figure \ref{FigTwistedCycle}.
Let $\E$ be a real phase structure on $C$ inducing the admissible set of twisted edges $T = \{ e_1 , e_2 \}$, for $e_1$ the unique bounded edge of direction $(1,1)$ and $e_2$ the unique bounded edge of direction $(0,1)$.
A closed twisted walk on $C$ starting at the couple $(v,e)$, with $v$ the uppermost 1-valent vertex and $e$ its unique incident edge, is represented in blue on Figure \ref{FigTwistedCycle}.
The unique twisted cycle on $(C,\E )$ is given as the equivalence class of this closed twisted walk.
\begin{figure}
\begin{center}
\includegraphics[width=0.45\textwidth]{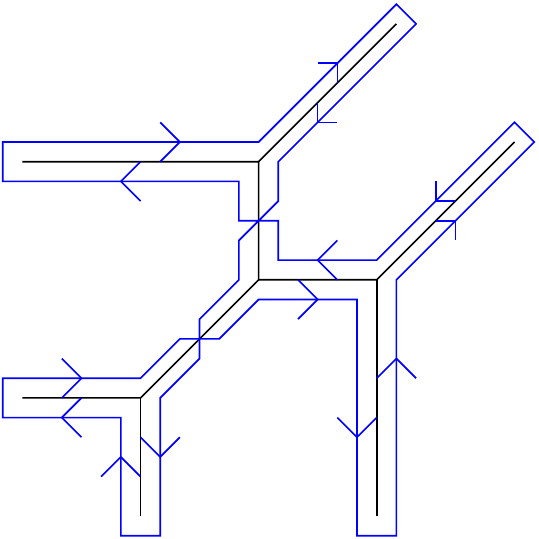}
\end{center}
\caption{Closed walk inducing a twisted cycle from Example \ref{ExTwistCycle}.}
\label{FigTwistedCycle}
\end{figure}
\end{example}

Given a twisted cycle $\rho$ on a non-singular real tropical curve $(C,\E )$ (hence a twisted cycle with respect to the admissible set of twisted edges $T$ induced by $\E$) inside a non-singular projective tropical toric surface $\T \pr_\Sigma$, we can construct its real part $\R \rho_\E$ as follows:
\begin{itemize}
\item For every two consecutive edges $e,e'$ in $\rho$ incident to a vertex $v \in \rho$ which is 3-valent in $C$, there exists a unique $\varepsilon \in \mathcal{R}^2$ such that $\E_e \cap \E_{e'} = \{ \varepsilon \}$. We then say that $(e,\varepsilon )$ and $(e',\varepsilon )$ belong to $(\rho ,\E )$. 
\item For every edge $e'' \in \rho$ incident to a vertex $v'' \in \rho$ which is 1-valent in $C$, by construction of twisted cycles $e''$ appears twice consecutively in $\rho$. Then for $\E_{e''} = \{ \varepsilon ' , \varepsilon '' \}$, we say that $(e'',\varepsilon ')$ and $(e'',\varepsilon '' )$ belong to $(\rho , \E)$. 
\end{itemize}

%Given a twisted cycle $\rho$ on a real tropical curve $(C,\E )$ (hence a twisted cycle with respect to the set of twists $T\in \Adm (C)$ induced by $\E$), we can construct a chain in $C_1 (C,\Sc_\E )$ as follows.
%\begin{itemize}
%\item For every two consecutive edges $e,e'$ in $\rho$ incident to a 3-valent vertex $v \in \rho$, there exists a unique $\varepsilon \in \Z_2^2$ such that $\E_e \cap \E_{e'} = \{ \varepsilon \}$.
%We then associate each of these edges to the cell $e\otimes w_\varepsilon \in C_1 (C,\Sc_\E)$ (resp. $e\otimes w_\varepsilon \in C_1 (C,\Sc_\E)$).
%\item For every edge $e'' \in \rho$ incident to a 1-valent vertex $v'' \in \rho$, by construction of twisted cycles $e''$ appears twice consecutively in $\rho$.
%For $\E_{e''} = \{ \varepsilon ' , \varepsilon '' \}$ the real phase on $e''$, each occurrence of $e''$ in $\rho$ is associated to one of the two cells $e'' \otimes w_{\varepsilon '} , e'' \otimes w_{\varepsilon ''} \in C_1 (C,\Sc_\E)$.
%\end{itemize}
%If a couple $(e,\varepsilon)$ satisfies one of the two cases above, we say that $(e,\varepsilon)$ belongs to $(\rho ,\E)$, or said otherwise $(e,\varepsilon) \in (\rho ,\E)$.  
%The chain $[\rho]$ associated to the twisted cycle $\rho$ is then:
%\[ [\rho] = \sum\limits_{(e,\varepsilon ) \in (\rho , \E )} e\otimes w_\varepsilon \in C_1 (C;\Sc_\E ) . \]

\begin{definition}
%Let $\rho$ be a twisted cycle on a real tropical curve $(C,\E)$ inside a tropical toric surface $\T \pr_\Sigma$.
The \emph{real part} $\R \rho_\E$ of the twisted cycle $\rho$ is given as 
\[  \R \rho_\E := \overline{ \bigcup\limits_{(e,\varepsilon) \subset (\rho ,\E)} \varepsilon (e) } \subset \R \pr_\Sigma .  \] 
\end{definition}

\begin{proposition}
\label{PropBasisTwistedCycles}
Let $(C,\E)$ be a non-singular real tropical curve in a non-singular projective tropical toric surface $\T \pr_\Sigma$.
The real part $\R \rho_\E$ of a twisted cycle $\rho$ in $(C,\E )$ is a connected component of $\R C_\E$, and every connected component of $\R C_\E$ is given as the real part of a twisted cycle on $(C,\E)$.
In particular, the set of chains \[ \left\{ [ \R \rho_\E ] \subset C_1 (\R C_\E ; \Z_2 ) ~ | ~ \rho \text{ twisted cycle on } (C,\E ) \right\} \]
induces a basis of $H_1 (\R C_\E ; \Z_2 )$.
%Moreover, each real part $\R \rho_\E$ of a twisted cycle $\rho$ in $(C,\E )$ is a connected component of $\R C_\E$.
%The chains $[\rho] \in C_1(C,\mathcal{S}_\E)$ associated to twisted  cycles $\rho$ on $(C,\E)$ form a basis of $H_1(C,\mathcal{S}_\E)$.
%In particular, any connected component of $\R C_\E$ is given as the real part of a twisted cycle of $(C,\E)$.
\end{proposition}

%\begin{remark}
%The proposition implies equivalently that the set of chains \[ \left\{ [\rho] = \sum\limits_{(e,\varepsilon ) \in (\rho,\E )} e\otimes w_\varepsilon \in C_1 (C;\Sc_\E ) ~ | ~ \rho \text{ twisted cycle of } (C,\E) \right\}  \] induces a basis of $H_1 (C; \Sc_\E)$ by Proposition. 
%\end{remark}

%In order to prove the proposition, we will need the following definition:
%
%\begin{definition}[\cite{godsil2013algebraic}]
%Let $\Gamma$ be a plane graph with only 4-valent and 2-valent vertices. 
%A \emph{straight eulerian cycle} on $\Gamma$ is a cycle $\gamma$ of $\Gamma$ such that: 
%\begin{itemize}
%\item each edge $e \in \Edge (\Gamma )$ can appear at most once in $\gamma$;
%\item the cycle $\gamma$ always leaves a vertex by the opposite edge it entered by. 
%%any two consecutive edges $e,e' \in \Edge (\Gamma )$ incident to a vertex $v\in \Vertex (\Gamma )$ are opposite in a neighborhood of $v$ (ie. in this neighborhood, the two other edges divides into two connected components, such that $e$ and $e'$ belong to distinct components).   
%\end{itemize}
%\end{definition}

\begin{proof}
Let $\rho$ be a twisted cycle in $(C,\E)$.
By definition of a twisted cycle, the underlying graph of $\rho$ is connected, hence by construction its real part $\R \rho_\E$ is a connected graph.
%Moreover, the chain $[\R \rho_\E] \subset C_1 (\R C_\E ; \Z_2)$ is a cycle.
Again by construction, every vertex of $\R \rho_\E$ is in the boundary of exactly two edges, hence for $[\R \rho_\E]$ the 1-dimensional chain in $C_1 (\R C_\E ; \Z_2)$ induced by the real part $\R \rho_\E$, the boundary chain $\partial [\R \rho_\E]$ is trivial in $C_0 (\R C_\E ; \Z_2)$.
Then the chain $[\R \rho_\E ]$ induces a class (also called $[\R \rho_\E]$) of the homology group $H_1 (\R C_\E , \Z_2 )$.
Since $\R C_\E$ is 1-dimensional, the class $[\R \rho_\E]$ is non-trivial in $H_1 (\R C_\E , \Z_2 )$.
Then since $\R \rho_\E$ is connected, it is a connected component of the real part $\R C_\E$.

Now let $c$ be a connected component of $\R C_\E$.
Let $\varphi : \R C_\E \rightarrow C$ be the absolute value map. 
%Now let $c$ be a connected component of $\R C_\E$, and 
%\begin{align*}
%\varphi : \R C_\E & \rightarrow C \\
%\varepsilon (e) & \mapsto e 
%\end{align*}
%for $e$ an edge of $C$ and $\varepsilon \in \mathcal{R}^2$.
The map $\varphi$ sends any two consecutive edges of $c$ to two (not necessarily distinct) adjacent edges of $C$.
Therefore the map $\varphi$ sends $c$ to a cycle $\rho$ on $C$ (in the sense of equivalence class of a closed walk under cyclic permutation and reversal).
If $c$ is inflected locally around an edge $e$, then a closed walk inducing the cycle $\rho$ must change side at the edge $\varphi (e)$ in $C$, since $C$ induces a convex polyhedral subdivision of $\T \pr_\Sigma$.
Similarly, if $c$ is locally convex around an edge $e$, then a closed walk inducing the cycle $\rho$ goes along the edge $\varphi (e)$ in $C$ without changing side.
Therefore $\rho$ is a twisted cycle on $(C,\E )$. 
% a walk $S ' := \varphi (S)$ on $C$ (for $S$ a walk on $\R C_\E$) changes of side at an edge $e$ if and only if the real part $\R C_\E$ is inflected locally around each edge $e' \in \varphi^{-1} (e)$.
%Therefore the walk $S'$ changes of side at the edge $e$ if and only if the real phase structure $\E$ induces a twist on $e$.
%Hence the equivalence class $\varphi (c)$ is a twisted cycle, so $c$ is the real part of a twisted cycle.   
\end{proof}

\section{Prescribed number of real components}
\label{Sec4}

\subsection{Exact sequence for real tropical curves}

Let $C$ be a non-singular tropical curve in a non-singular projective tropical toric surface $\T \pr_\Sigma$.
The \emph{tangent cosheaf} $\cF_1$ on $C$ is defined as follows.
For $e$ an edge of $C$, we set $\cF_1 (e) := \Z_2 \langle \overrightarrow{e} \rangle$, with $\overrightarrow{e}$ the direction modulo 2 of $e$.
For $v$ a 3-valent vertex of $C$, we set $\cF_1 (v) := \Z_2^2$, and for $v'$ a 1-valent vertex of $C$, we set $\cF_1 (v') := 0$. 
We refer to \cite{renaudineau2018bounding} for more details, and to \cite{itenberg2019tropical} for the initial notion of $p$-multitangent cosheaves $\cF_p$ with $\Q$-coefficients.
The \emph{cellular chain complex} $C_\bullet (C;\cF_1 )$ is given by 
\begin{align*}
\bigoplus\limits_{e \in \Edge (C)} \cF_1 (e) =: C_1 (C ; \cF_1 )  & \rightarrow C_0 (C ; \cF_1 ) := \bigoplus\limits_{v \in \Vertex (C)} \cF_1 (v) \\
e \otimes \overrightarrow{e} & \mapsto (\partial e) \otimes \overrightarrow{e}. 
\end{align*}
Then the homology groups $H_q (C; \cF_1 )$ are given as $H_q (C_\bullet (C; \cF_1 ))$.
The homology group $H_1(C;\cF_1)$ is generated by the class $[C]$ induced by the chain \[ \sum_{e\in \Edge (C)} e \otimes \overrightarrow{e} \subset C_1 (C ; \cF_1 ) . \]
The group $H_0 (C; \cF_1 )$ has dimension $g$, for $g:= \dim H_1 (C ; \Z_2 )$ the number of \emph{primitive cycles} on the non-singular tropical curve $C$.

\begin{definition}
Let $C$ be a non-singular tropical curve in a non-singular projective tropical toric surface $\T \pr_\Sigma$.
A cycle $\gamma$ of the non-singular tropical curve $C$ is called \emph{primitive} if it bounds a connected component of $\T \pr_\Sigma \backslash C$.
\end{definition}

Let $\{ \gamma_1 , \ldots ,\gamma_g \}$ be the set of primitive cycles on $C$.
The set of classes \[ \{ [\gamma_i^*] ; i=1,\ldots , g \} \] form a basis of $H_0 (C; \cF_1 )$, for $[\gamma_i^*]$ induced by the chain
\[ \sum\limits_{e \in \Edge (\gamma_i) } v_e \otimes \overrightarrow{e} \subset C_0 (C;\cF_1 ) ,  \]
with $v_e$ the first vertex of the edge $e$ after choosing a cyclic order on $\gamma_i$.

Using these notions of tropical homology and several results from \cite{renaudineau2018bounding}, we will be able to determine the number of connected components of the real part of a non-singular real tropical curve.

\begin{definition}
Let $(C,\E)$ be a non-singular real tropical curve, and let $g$ be the number of primitive cycles on $C$.
We say that $(C,\E)$ is \emph{maximal} if its real part $\R C_\E$ consists of $g+1$ connected components.
More generally, we say that $(C,\E)$ is a \emph{$(M-r)$} real tropical curve if its real part $\R C_\E$ consists of $g+1-r$ connected components.
\end{definition}

\begin{theorem}[\cite{renaudineau2018bounding}]
\label{ThExactSeq}
Let $(C,\E)$ be a non-singular real tropical curve in a non-singular projective tropical toric surface $\T \pr_\Sigma$. 
We have the following long exact sequence:
\[ \begin{tikzcd}
0 \arrow{r} & H_1(C;\cF_1) \arrow{r} & H_1(\R C_\E ; \Z_2 ) \arrow{r} & H_1(C;\Z_2) \arrow{lld}{\partial} &  \\
 & H_0(C;\cF_1) \arrow{r} & H_0(\R C_\E ; \Z_2 ) \arrow{r} & H_0(C;\Z_2 ) \arrow{r} & 0 .
\end{tikzcd} \]
In particular, the number of connected components of $\R C_\E$ is equal to $\dim \ker \partial + 1$, or said otherwise, the non-singular real tropical curve $(C,\E)$ is a $(M - r)$-real tropical curve for $r:= \rank \partial$. 
\end{theorem}

\begin{proof}
We apply \cite[Proposition 4.10]{renaudineau2018bounding} for $p=0$ and for $p=1$ to the case of non-singular real tropical curves.
\end{proof}

For $C$ a non-singular tropical curve with primitive cycles $\{ \gamma_1 ,\ldots , \gamma_g \}$ and $T$ a subset of bounded edges of $C$, we introduce a $g\times g$ symmetric matrix with $\Z_2$-coefficients keeping track of the (mod 2)-intersection numbers $| \gamma_i \cap \gamma_j \cap T |\mod 2$ for $i,j = 1,\ldots , g$.

\begin{definition}
\label{DefMatrixTwist}
Let $C$ be a non-singular tropical curve, let $\{ \gamma_1, \ldots , \gamma_g \}$ be the primitive cycles on $C$, and let $T$ be a subset of bounded edges on $C$.
Seeing the primitive cycles as subsets of bounded edges of $C$, we define the $g\times g$ symmetric matrix $A_T$ with $\Z_2$-coefficients as
\[ A_T := \left( |\gamma_i \cap \gamma_j \cap T | \mod 2 \right)_{i,j = 1,\ldots ,g} . \]
\end{definition}

For $(C,\E)$ a non-singular real tropical curve in a non-singular projective tropical toric surface $\T \pr_\Sigma$, we want to be able to compute the connecting homomorphism $\partial$ of \Cref{ThExactSeq} in terms of the admissible set of twisted edges $T$ induced by $\E$.
We have a non-degenerate pairing:
\[ \langle \cdot , \cdot \rangle : H_0 (C; \cF_1 ) \times H_1 (C; \Z_2 ) \rightarrow \Z_2 \]
induced from the pairing on integral homology groups for non-singular tropical curves in \cite{shaw2011tropical} (a similar non-degenerate pairing defined between tropical homology and cohomology groups is also defined in \cite{brugalle2015brief} (Section 7.8) and \cite{mikhalkin2014tropical} (Section 3.2)).
Renaudineau and Shaw \cite{renaudineau2018bounding} used this pairing in order to recover the whole map $\partial$ as follows.

\begin{theorem}[\cite{renaudineau2018bounding}]
\label{ThConnectHom+}
Let $(C,\E )$ be a non-singular real tropical curve with admissible set of twisted edges $T$.
Let $\{ \gamma_1, \ldots , \gamma_g \}$ be the primitive cycles on $C$, seen as subsets of bounded edges of $C$.
Then 
\[ \langle \partial [\gamma_i] , [\gamma_j] \rangle = | \gamma_i \cap \gamma_j \cap T | \mod 2 .   \] 
In particular, we have $\dim \ker A_T = \dim \ker \partial$, for $\partial$ the connecting homomorphism in \Cref{ThExactSeq}. 
\end{theorem}

The theorem above is a reformulation of \cite[Theorem 7.2]{renaudineau2018bounding} using the non-degenerate pairing formulation used in the proof of \cite[Theorem 7.5]{renaudineau2018bounding}.
From now on, the connecting homomorphism from Theorem \ref{ThConnectHom+} will be denoted $\partial_T$, as it is determined by the admissible set of twisted edges $T$ on $(C,\E )$. 
%From now on, we will say that a non-singular real tropical curve $(C,\E )$ is $(M-r)$ if there exists a real algebraic curve converging to $(C,\E )$ which is $(M-r)$. 
%We will give in this section several combinatorial conditions in order to obtain a prescribed number of connected components in the real part $\R C_\E$ of a non-singular real tropical curve $(C,\E)$, which will be given in terms of the set of twists $T$ induced by $\E$.  

\subsection{Haas conditions for dividing and maximal curves}
 
\begin{definition}
We say that a non-singular real tropical curve $(C,\E)$ is \emph{dividing} if $(C,\E)$ is the tropical limit of a family of non-singular real algebraic curves $(\mathcal{C}_t)_t$ such that for $t \in \R_{>0}$ small enough, the curve $\mathcal{C}_t$ is dividing.
\end{definition}

Haas gave a necessary and sufficient condition to construct dividing curves via combinatorial patchworking, which translates in terms of real tropical curves as follows:

\begin{theorem}[\cite{haas1997real}]
\label{ThDividing+}
Let $(C,\E )$ be a non-singular real tropical curve with admissible set of twisted edges $T$.
Then $(C,\E )$ is dividing if and only if for any cycle $\gamma$ in $C$ (seen as a subset of bounded edges of $C$), we have 
\begin{equation}
\label{EqDiv+}
| \gamma \cap T | = 0 \mod 2,  
\end{equation}
\end{theorem}

\begin{definition}
A twist-admissible set $T$ of bounded edges of $C$ is \emph{dividing} if \Cref{EqDiv+} is satisfied for all cycles $\gamma$ in $C$. 
%A locally-admissible set $T$ of bounded edges of $C$ is \emph{locally dividing} if there exists a non-singular tropical curve $C' \subset C$ with $T \subset \Edge^0 (C')$ such that \Cref{EqAdm+} and \Cref{EqDiv+} are satisfied for all cycles $\gamma$ in $C'$.
%A configuration of twists $T \in \Adm (C)$ is \emph{dividing} if \Cref{EqDiv+} is satisfied for all cycles $\gamma$ in $C$.
We denote by $\Div (C) \subset \Adm (C)$ the set of all dividing sets of bounded edges on $C$.
%, and by $\LocDiv(C)$ the set of all locally dividing set of bounded edges on $C$.
\end{definition}

\begin{remark}
Let $g$ be the number of primitive cycles of a non-singular tropical curve $C$.
The linear conditions from \Cref{EqDiv+} imply that the set $\Div (C)$ has a $\Z_2$-vector space structure induced from $\Adm (C)$.
In particular, the $\Z_2$-vector space $\Div (C)$ has codimension $l$ in $\Adm (C)$, for $l\leq g$ the number of primitive cycles $\gamma$ of $C$ such that $\gamma$ contains at least one edge of direction $(1,0) \mod 2$, one edge of direction $(0,1) \mod 2$ and one edge of direction $(1,1) \mod 2$.
\end{remark}

\begin{definition}
Let $C$ be a non-singular tropical curve in $\R^2$.
We say that an edge $e$ of $C$ is \emph{exposed} if $e$ lies on the boundary of an unbounded connected component of $\R^2 \backslash C$.
For $C'$ the compactification of $C$ in a non-singular projective tropical toric surface $\T \pr_\Sigma$, we say that an edge $e$ of $C'$ is exposed if the corresponding edge of $C$ is exposed.  
We denote by $\Exp (C)$ the set of all exposed edges on $C$.
\end{definition}

Haas gave a necessary and sufficient condition to construct $M$-curves via combinatorial patchworking \cite[Theorem 7.3.0.10]{haas1997real}, which can be translated in terms of real tropical curves as follows.

\begin{figure}
\begin{center}
\includegraphics[width=0.45\textwidth]{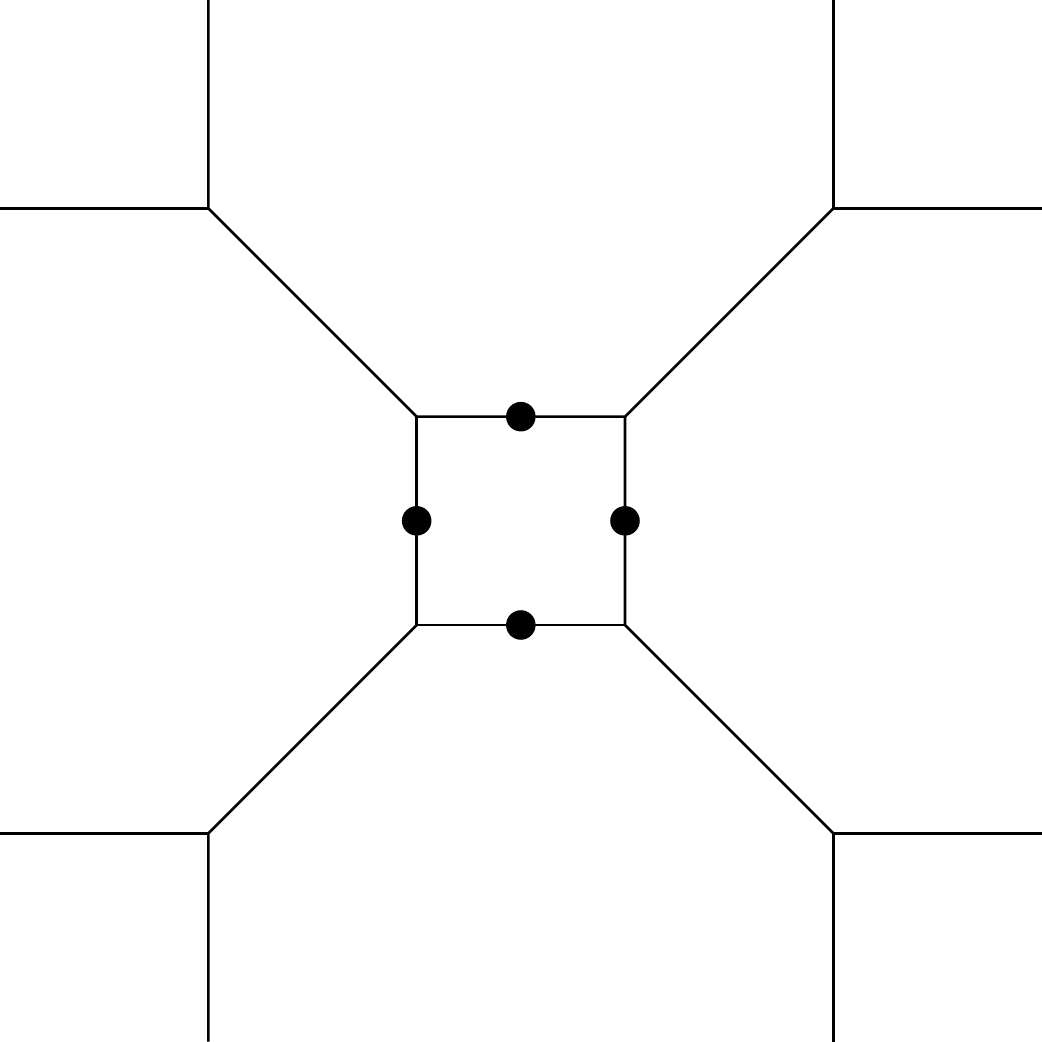}
\end{center}
\caption{Tropical curve with maximal set of twisted edges in Example \ref{ExMCurve}}
\label{FigMCurve}
\end{figure}

\begin{theorem}[\cite{haas1997real}]
\label{ThHaas}
Let $(C,\E)$ be a non-singular real tropical curve in a non-singular projective tropical toric surface $\T \pr_\Sigma$ with admissible set of twisted edges $T$.
Then $(C,\E )$ is maximal if and only if 
\begin{enumerate}
\item \label{ItemMax1} the set of twisted edges $T$ is dividing;
\item \label{ItemMax2} every edge of $T$ is exposed.
\end{enumerate}
\end{theorem}

Note that \Cref{ItemMax1} of \Cref{ThHaas} comes from the fact that every maximal real algebraic curve is dividing.

\begin{definition}
We say that a twist-admissible subset $T$ of bounded edges of a non-singular tropical curve $C$ is \emph{maximal} if $T$ satisfies \Cref{ItemMax1} and \Cref{ItemMax2} in \Cref{ThHaas}.
We denote by $\Max (C)$ the set of all maximal subsets $T$ of bounded edges of $C$.
\end{definition}

\begin{remark}
Let $C$ be a non-singular tropical curve with $g$ primitive cycles.
%Let $g$ be the topological genus of a non-singular tropical curve $C$.
\Cref{ItemMax2} of \Cref{ThHaas} imply $l$ linear conditions on the space $\Div (C)$, for $l \leq \binom{g}{2}$ the number of pairs of distinct primitive cycles of $C$ with non-empty intersection.
Then the set $\Max (C)$ has a $\Z_2$-vector space structure induced from $\Div (C)$, and is of codimension $l$ in $\Div (C)$.
The same set without twist-admissibility condition (\Cref{EqAdm+}) has a $\Z_2$-vector space structure induced from $\Edge^0 (C)$, as shown in \cite{renaudineau2017haas}.
\end{remark}

\begin{example}
\label{ExMCurve}
Every non-singular real tropical curve $(C,\E )$ with empty set of twisted edges is maximal.
%For every non-singular tropical curve $C$, the empty configuration of twists is a maximal configuration.
%For $C$ a non-singular tropical curve in $\T \pr^2$, a non-singular real algebraic curve converging to $(C,\E )$ with $\E$ inducing the empty configuration of twists is called a \emph{simple Harnack} curve, in reference of Harnack's classical construction of maximal curves (see for example \cite[\S 5.3]{benedetti1991real} or \cite[\S 11.6]{bochnak2013real} for a description of Harnack's construction).
In \Cref{FigMCurve}, the edges marked by black dots on the tropical curve form an example of non-empty maximal set of bounded edges.
Indeed, all these edges are exposed, and the unique cycle in the tropical curve contains an even number of twisted edges.
%The edges dual to twisted edges (represented in blue on the superposition) all have a vertex contained in the boundary of $\Delta_C$, which is an equivalent condition for a bounded edge to be exposed.
\end{example}

\subsection{(M-1) curves}

We want to give several new combinatorial conditions, starting by $(M-1)$ curves. 
Before stating the theorem, we need the following definition:

\begin{figure}
\centering
\begin{subfigure}[t]{0.475\linewidth}
	\centering
	\includegraphics[width=0.9\textwidth]{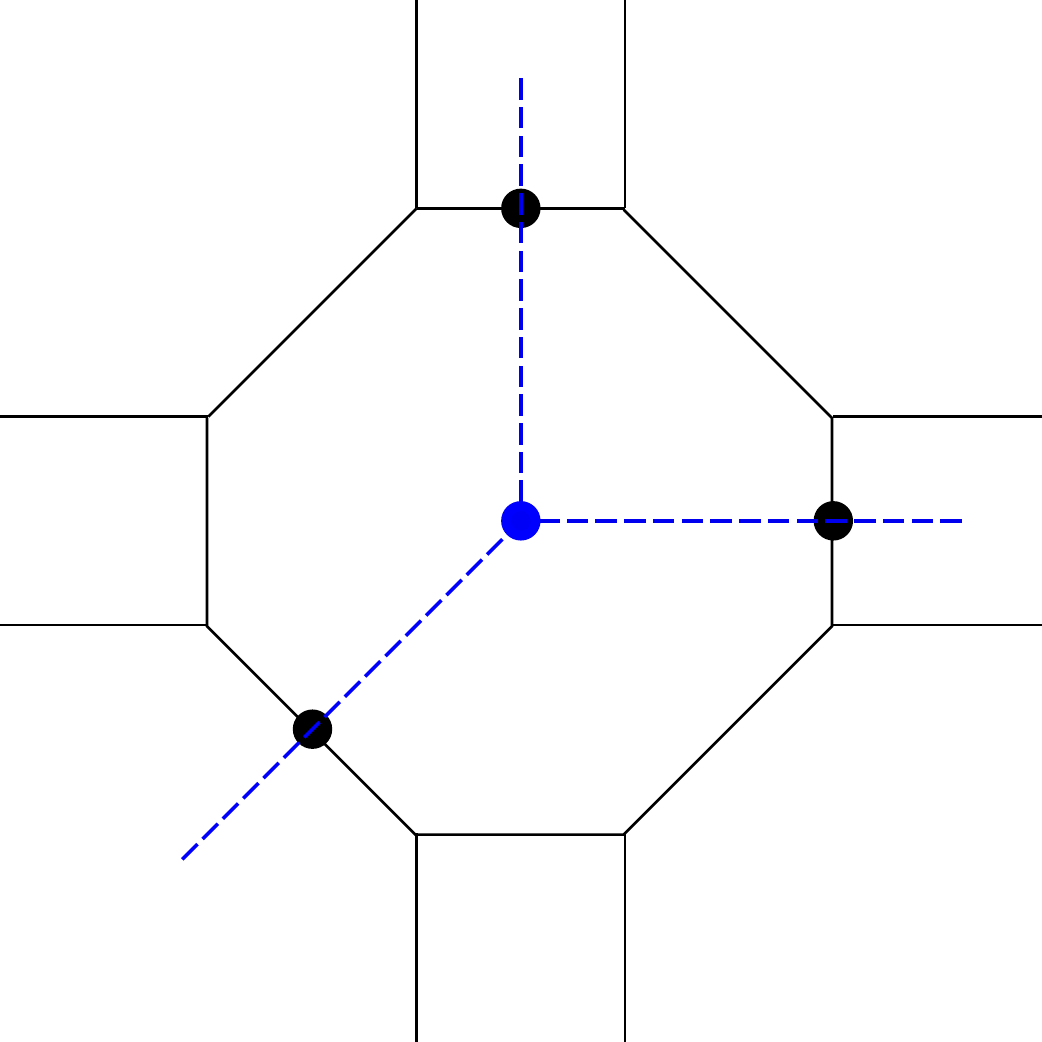}
	\caption{Complete graph $K_1$ and exposed twisted edges.}
	\label{FigK1}
\end{subfigure}\hfill
\begin{subfigure}[t]{0.475\linewidth}
	\centering
	\includegraphics[width=0.9\textwidth]{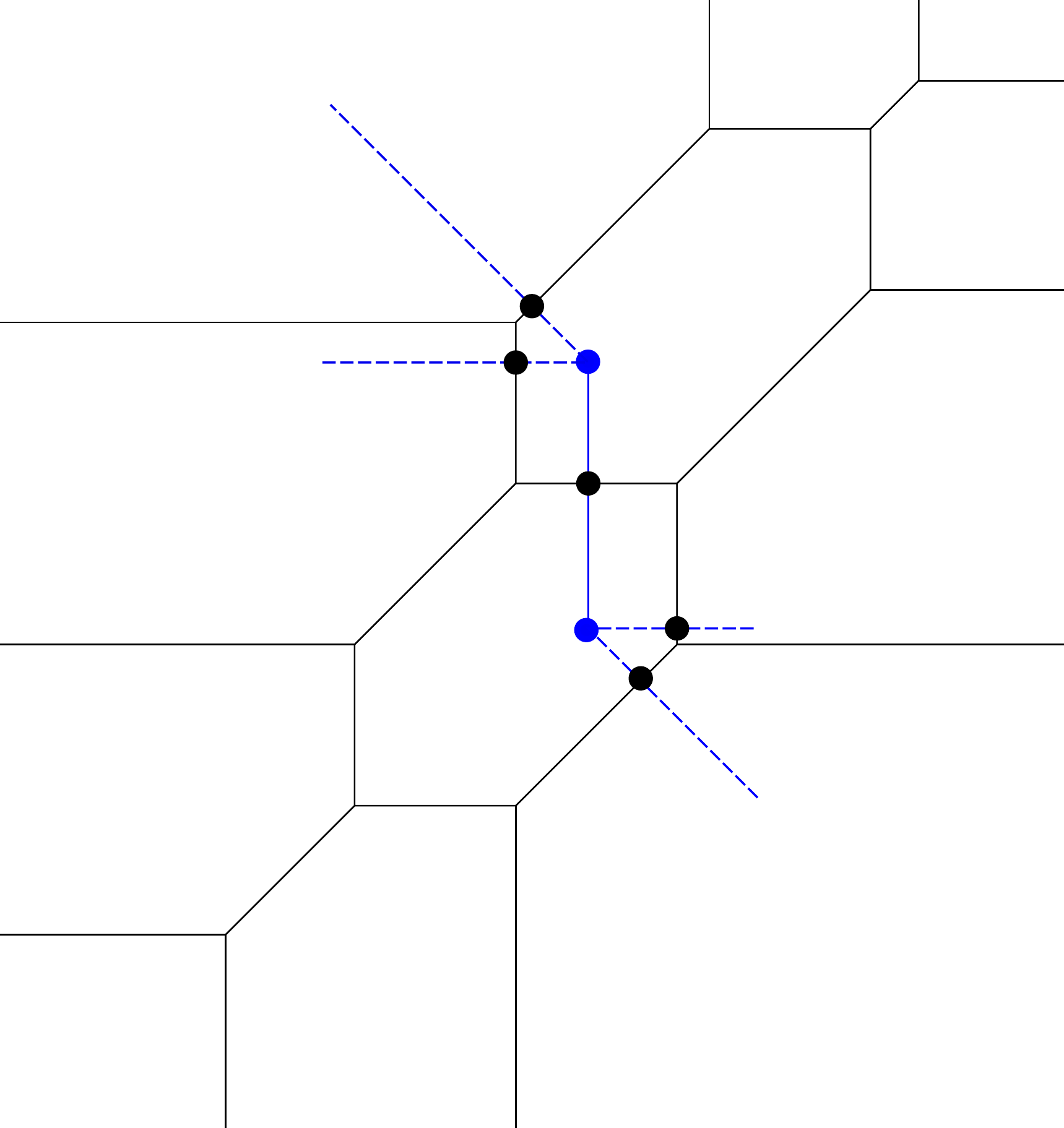}
	\caption{Complete graph $K_2$ and exposed twisted edges.}
	\label{FigK2}
\end{subfigure}
%\begin{subfigure}[t]{0.3\linewidth}
%	\centering
%	\includegraphics[width=\textwidth]{K3}
%	\caption{Complete graph $K_3$ and exposed twisted edges.}
%	\label{FigK3}
%\end{subfigure}
\caption{Example \ref{ExM-1}.}
\label{type12}
\end{figure}

\begin{definition}
Let $C$ be a non-singular tropical curve in a non-singular projective tropical toric surface $\T \pr_\Sigma$, and let $T$ be a set of bounded edges on $C$.
The \emph{dual graph of non-exposed twisted edges} $\Gamma_T$ is defined as follows.
\begin{itemize}
\item The set of vertices $\Vertex (\Gamma_T )$ of $\Gamma_T$ is the set of primitive cycles of $C$.
\item The set of edges $\Edge (\Gamma_T )$ of $\Gamma_T$ is given by the pairs of distinct primitive cycles $\gamma \neq \gamma '$ such that $\gamma \cap \gamma ' \cap T \neq \emptyset$ (seeing the primitive cycles as subsets of bounded edges of $C$).  
\end{itemize} 
\end{definition}

\begin{theorem}[\Cref{IntroThM-1}]
\label{ThM-1}
Let $(C,\E )$ be a non-singular real tropical curve in a non-singular projective tropical toric surface $\T \pr_\Sigma$, with admissible set of twisted edges $T$.
Then $(C,\E)$ is a $(M-1)$ real tropical curve if and only if:
\begin{enumerate}
\item \label{Item1M-1} the dual graph of non-exposed twisted edges $\Gamma_T$ consists of a complete subgraph $K_n \subset \Gamma_T$ on $1\leq n\leq 4$ vertices plus isolated vertices;
\item \label{Item2M-1} every primitive cycle $\gamma$ in $\Vertex (K_n)$, seen as a subset of bounded edges of $C$, satisfies \[ |  \gamma \cap T | = 1 \mod 2 ; \] 
\item \label{Item3M-1} every primitive cycle $\gamma$ in $\Vertex (\Gamma_T) \backslash \Vertex (K_n)$, seen as a subset of bounded edges of $C$, satisfies \[ | \gamma \cap T | = 0 \mod 2 . \]
\end{enumerate} 
\end{theorem}

\begin{remark}
\label{RemComplete}
Since a non-singular tropical curve $C$ in a non-singular projective tropical toric surface $\T \pr_\Sigma$ is a planar graph, the dual graph of non-exposed twisted edges cannot contain a complete subgraph $K_n$ on $n \geq 5$ vertices by Kuratowski's Theorem \cite{kuratowski1930probleme}. 
\end{remark}

\begin{figure}
\centering
\begin{subfigure}[t]{0.475\linewidth}
	\centering
	\includegraphics[width=0.9\textwidth]{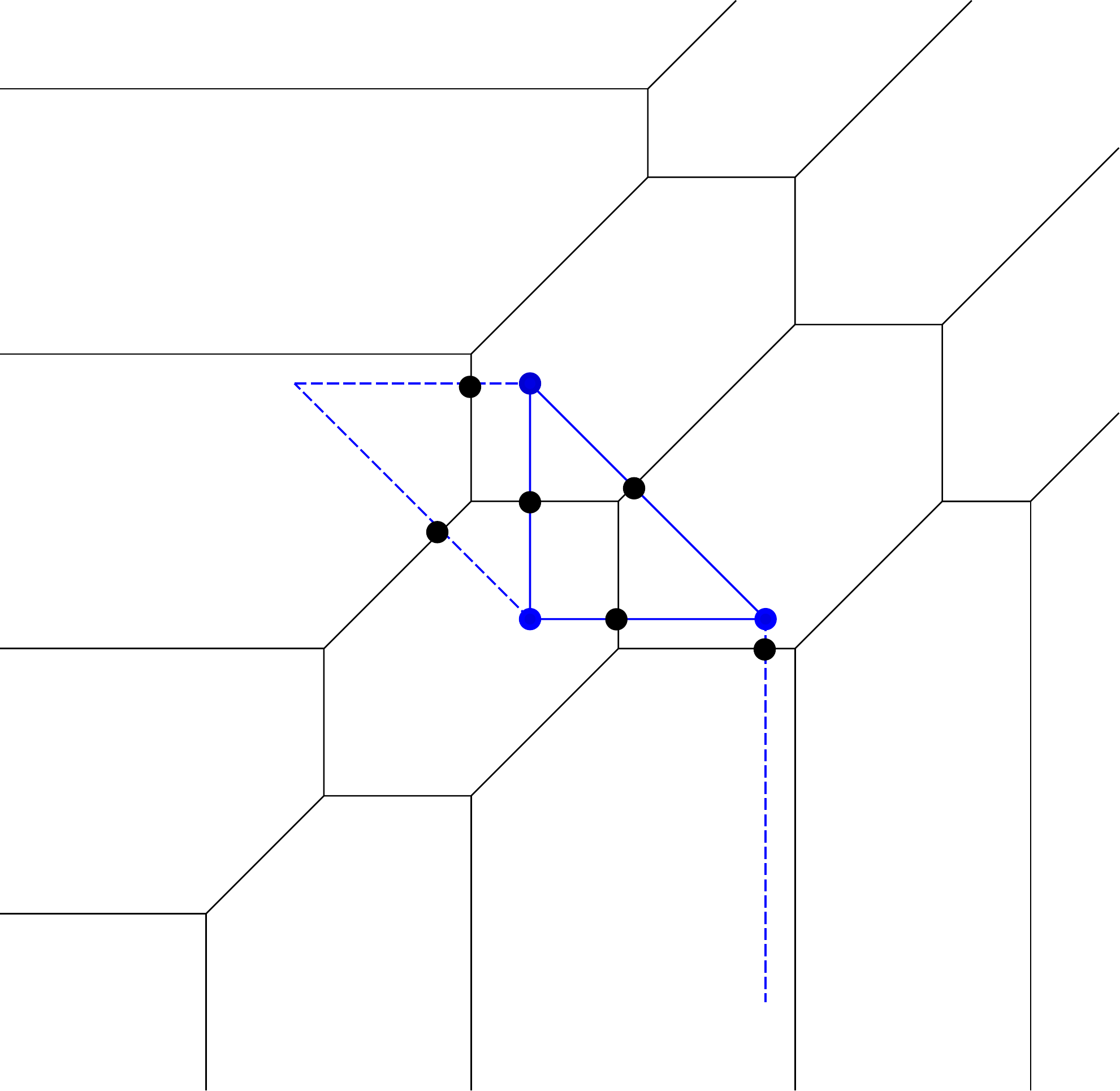}
	\caption{Complete graph $K_3$ and exposed twists.}
	\label{FigK3}
\end{subfigure}\hfill
\begin{subfigure}[t]{0.475\linewidth}
	\centering
	\includegraphics[width=0.9\textwidth]{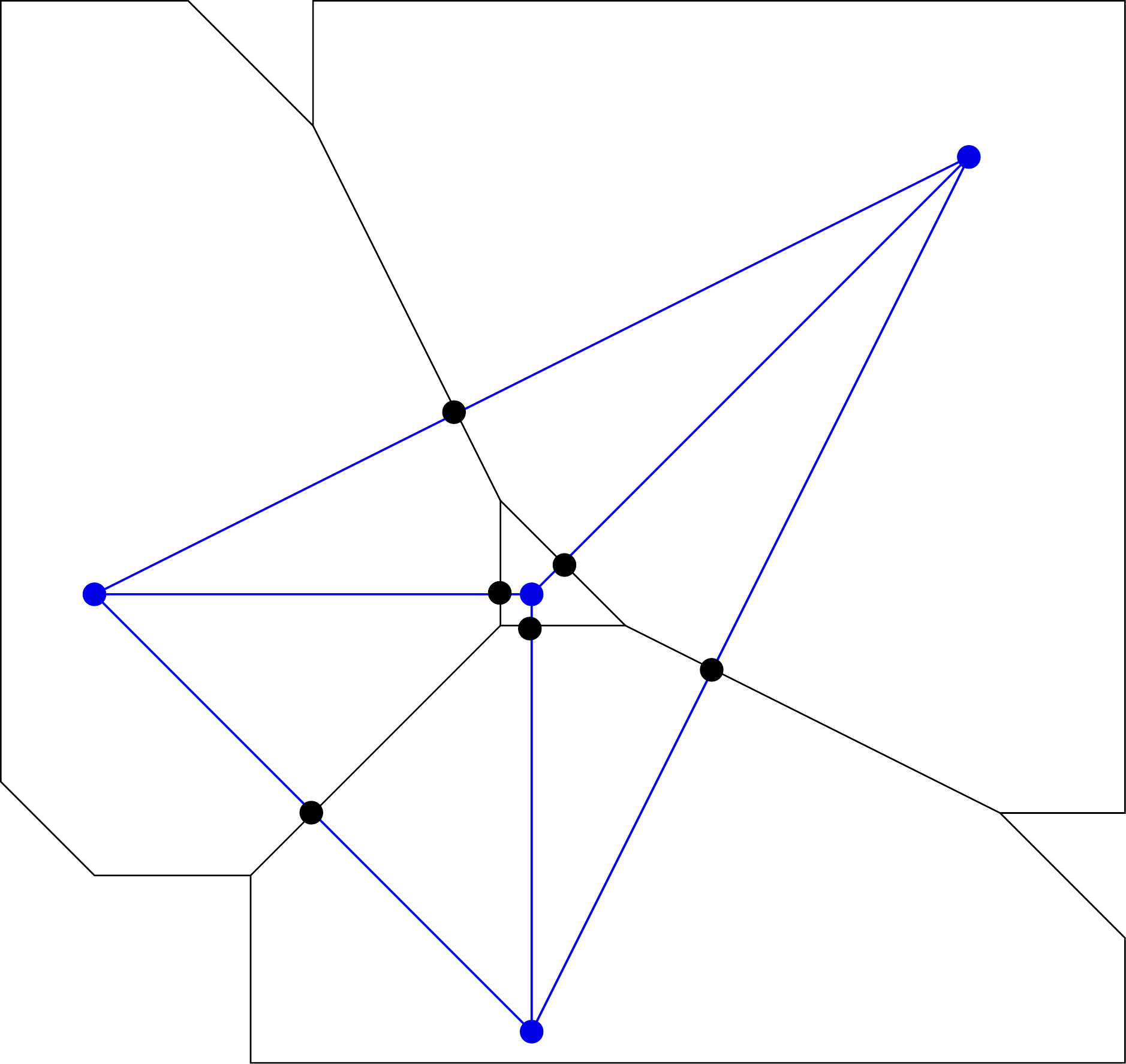}
	\caption{Complete graph $K_4$.}
	\label{FigK4}
\end{subfigure}
\caption{Example \ref{ExM-1}.}
\label{type34}
\end{figure}

%\begin{definition}
%{\color{red} Remove ?}
%We say that a configuration of twists $T\in \Edge^0 (C)$ on a non-singular tropical curve $C$ satisfying \Cref{Item1M-1}, \Cref{Item2M-1} and \Cref{Item3M-1} of Theorem \ref{ThM-1} is a \emph{$(M-1)$ configuration of twists}.
%\end{definition}
\begin{example}
\label{ExM-1}
Figures \ref{type12} and \ref{type34} represent examples of tropical curves with set of twisted edges (marked by black dots) such that their dual graph of non-exposed twisted edges consists of a complete graph $K_n, 1\leq n\leq 4$, represented in plain blue.
In particular, the graph $K_1$ in \Cref{FigK1} is the single blue point.
The dotted blue edges are dual to the exposed twisted edges.
All those sets of twisted edges give rise to $(M-1)$ real tropical curves by \Cref{ThM-1}.
Note that we do not represent all the edges of the tropical curve in \Cref{FigK4}.
\end{example}

In order to prove the theorem, we recall several standard definitions in graph theory (see for instance \cite{godsil2013algebraic}).

\begin{definition}
Let $\Gamma$ be a graph (without loop and without multiple edges) with $\Vertex (\Gamma ) =\{ v_1 , \ldots , v_r  \}$.
The \emph{adjacency matrix} $N$ of $\Gamma$ is the $r\times r$-matrix with coefficients in $\Z_2$ such that $N_{ij} = 1$ if the vertices $v_i$ and $v_j$ are adjacent, and $N_{ij} = 0$ otherwise (in particular, the diagonal coefficients of $N$ are all zero as $\Gamma$ has no loop). 
The \emph{valency matrix mod 2} of $\Gamma$ is the $r\times r$ diagonal matrix $D$ with coefficients $D_{ii}$ given as the valency of $v_i$ modulo 2.
The \emph{Laplacian mod 2} of $\Gamma$ is the matrix $Q_\Gamma := D+N$ (in particular it is a symmetric $r\times r$-matrix with coefficients in $\Z_2$).
\end{definition}

\begin{proof}[Proof of Theorem \ref{ThM-1}]\renewcommand{\qedsymbol}{}
Let $\{ \gamma_1 ,\ldots , \gamma_g \}$ be the primitive cycles of $C$.
Let $Q_{\Gamma_T}$ be the Laplacian mod 2 of the dual graph of non-exposed twisted edges $\Gamma_T$ with respect to $\{ \gamma_1 ,\ldots , \gamma_g \}$.
Seeing the primitive cycles as subsets of bounded edges of $C$, let $D_T$ be the $g\times g$ diagonal matrix with $\Z_2$-coefficients 
\[ (D_T)_{ii} = | \gamma_i \cap T \cap \Exp (C) | \mod 2   . \]
Then the matrix $A_T$ is given by $A_T = Q_{\Gamma_T} + D_T$.
Now $(C,\E )$ is a $(M-1)$ real tropical curve if and only if  $\rank A_T = 1$ by \Cref{ThConnectHom+}.  
A symmetric $g\times g$-matrix of rank 1 with coefficients in $\Z_2$ is the product of a vector $z\in \Z_2^g$ with its transpose. Up to permutation of its coefficients, we can write \[ z^t = (1,\ldots ,1,0,\ldots,0). \]
Therefore, up to renumbering the primitive cycles $\gamma_i$, the matrix $A_T$ has the form
\[ \left(\begin{NiceMatrix}
1 & \Cdots & 1 & 0 & \Cdots & 0 \\
\Vdots & \Ddots & \Vdots & \Vdots & \Ddots & \Vdots \\
1 & \Cdots & 1 & 0 & \Cdots & 0 \\
0 & \Cdots & 0 & 0 & \Cdots & 0 \\
\Vdots & \Ddots & \Vdots & \Vdots & \Ddots & \Vdots \\
0 & \Cdots & 0 & 0 & \Cdots & 0 \\
\end{NiceMatrix}\right) . \]
Let $Q = D+N$ be the diagonal bloc of $A_T$ with only non-zero coefficients, with $D$ the diagonal matrix of $Q$.
We obtain the following facts:
\begin{enumerate}
\item the matrix $N$ is the adjacency matrix of a complete graph $K_n$ (with $1\leq n \leq 4$ by Remark \ref{RemComplete}), since all its non-diagonal coefficients are non-zero;
\item every primitive cycle in $\Vertex (\Gamma_T) \backslash \Vertex (K_n)$ is isolated, since the non-diagonal coefficients of $A_T$ not belonging to the minor $Q$ are all zero;
\item every primitive cycle $\gamma$ in $\Vertex (K_n)$ satisfies $| \gamma \cap T | = 1 \mod 2$, since the diagonal coefficients of $Q$ are all non-zero; 
\item every primitive cycle $\gamma$ in $\Vertex (\Gamma_T) \backslash \Vertex (K_n)$ satisfies $|\gamma \cap T | = 0 \mod 2$, since the diagonal coefficients of $A_T$ not belonging to the minor $Q$ are all zero. \quad \quad \quad \quad $\blacksquare$
\end{enumerate} 
\end{proof}

\begin{definition}
Let $C$ be a non-singular tropical curve, and let $T,T'$ be two twist-admissible set of bounded edges on $C$.
We say that $T$ and $T'$ are \emph{cycle-disjoint} if for every primitive cycle $\gamma$ on $C$, we have 
\begin{align*}
\gamma \cap T \neq \emptyset & \Rightarrow \gamma \cap T' = \emptyset \\
\text{ and } \gamma \cap T' \neq \emptyset & \Rightarrow \gamma \cap T = \emptyset .
\end{align*}
\end{definition}

\begin{corollary}
\label{CorM-rNonDiv}
Let $(C,\E)$ be a non-singular real tropical curve with admissible set of twisted edges $T = T_1 \sqcup \ldots \sqcup T_r$, such that the dual graphs of non-exposed twisted edges $\Gamma_{T_i}$ are all disjoint and they all satisfy \Cref{ThM-1}.
Then $(C,\E )$ is a $(M-r)$ real tropical curve which is non-dividing. 
\end{corollary}

\begin{proof}
By assumption, the $T_i$'s are all cycle-disjoint, hence there exists an ordering of the primitive cycles of $C$ such that the matrix $A_T$ satisfying this ordering is bloc-diagonal, with each non-zero bloc being the non-zero bloc of the matrix $A_{T_i}$ satisfying the same ordering of primitive cycles.
Then \[ \rank A_T = \sum_{i=1}^r \rank A_{T_i} = r \] by \Cref{ThM-1}, hence $(C,\E)$ is a $(M-r)$ real tropical curve by \Cref{ThConnectHom+}.
Since every bloc has non-zero diagonal, we obtain by \Cref{ThDividing+} that $(C,\E)$ is non-dividing. 
\end{proof}

\subsection{Dividing (M-2) curves}

%\begin{remark}
%Notice that the condition above appears in Theorem \ref{ThHaas}. 
%Indeed, any maximal curve $\mathcal{C}$ is dividing, as $\mathcal{C}(\R )$ has $g+1$ connected components and $\mathcal{C}(\C )$ is a Riemann surface with $g$ holes, hence $\mathcal{C}(\C )\backslash \mathcal{C}(\R )$ must be disconnected. 
%\end{remark}

Let $(C,\E)$ be a non-singular dividing real tropical curve.
By \Cref{CorM-rNonDiv}, the set of twisted edges on $(C,\E )$ cannot be given as cycle-disjoint union of $(M-1)$ configurations of twists, that is a cycle-disjoint union of configurations of twists satisfying \Cref{ThM-1}.
%Recall that every $(M-1)$-curve is non-dividing.
%In particular, if a real tropical curve $(C,\E )$ is $(M-2)$ and dividing, the configuration of twists $T\in \Div (C)$ induced by $\E$ cannot be obtained as the sum of two distinct $(M-1)$ configurations of twists.
This motivates the following result.
%We will now prove that the theorem above holds also in the case of real structures above a trivalent graph:
%
%\begin{theorem}
%
%Let $C$ be a real trivalent graph with a set of twists $T \in \Edge^0 (C)$, let $S$ be a surface constructed from $C$ and let $\tau_C$ be a real structure above $C$ inducing $T$.
%The curve $S^{\tau_C}$ is dividing (ie. $S \backslash S^{\tau_C}$ is disconnected) if and only if for any cycle $\alpha$ in $C$, we have $\# \alpha \cap T = 0 \mod 2$.
%\end{theorem}
%
%\begin{proof}
%The curve $S^{\tau_C}$ is dividing if and only if the surface $S/ \tau_C$ is orientable.
%Therefore, any cycle $\gamma_\alpha \in H_1 (S /\tau_C , S^{\tau_C} ;\Z_2)$ lifted from a cycle $\alpha \in H_1 (C;\Z_2)$ must be orientable.
%This is possible if and only if the real structure the restriction of $\tau_C$ to a neighborhood of $\gamma_\alpha$ (seen in $S$) is (up to composition with an even power of the Dehn twist $DT$) isotopic to $\conj$, ie. the cycle $\gamma_\alpha$ must go through an even number of "half" Dehn twists $\DT \circ \conj$, which is equivalent to the condition 
%\[ \# \alpha \cap T = 0 \mod 2 .  \]  
%\end{proof}

\begin{theorem}[\Cref{IntroThM-2}]
\label{ThM-2}
Let $(C,\E )$ be a non-singular real tropical curve in a non-singular projective tropical toric surface $\T \pr_\Sigma$, with dividing set of twisted edges $T$.
Then $(C,\E)$ is a dividing $(M-2)$ real tropical curve if and only if the dual graph of non-exposed twisted edges $\Gamma_T$ consists of 
\begin{enumerate}
\item \label{Item1M-2} either a complete planar bipartite graph and isolated vertices;
\item \label{Item2M-2} or a complete planar tripartite graph and isolated vertices.
\end{enumerate}
\end{theorem} 

\begin{remark}
\label{RemComplete2}
Since a non-singular tropical curve $C$ in a non-singular projective tropical toric surface $\T \pr_\Sigma$ is a planar graph, by \cite{kuratowski1930probleme}, the dual graph of non-exposed twisted edges $\Gamma_T$ cannot contain a complete bipartite subgraph $K_{3,3}$ on two sets of 3 vertices.
\end{remark}

\begin{example}
\label{ExM-2}
The underlying non-singular tropical curve in Figure \ref{FigM-2Div} is of bidegree $(4,4)$ in $\T \pr^1 \times \T \pr^1$.
In Figure \ref{FigK22}, we have a set of twisted edges (marked by black dots) with dual graph of non-exposed twisted edges (represented in plain blue) consisting of a complete bipartite graph $K_{2,2}$ and isolated vertices.
Notice that there are no exposed twisted edges in this case.
In Figure \ref{FigK221}, we have a set of twisted edges (marked by black dots) with dual graph of non-exposed twisted edges (represented in plain blue) consisting of a complete tripartite graph $K_{2,2,1}$ and isolated vertices.
The edges dual to exposed twisted edges are represented in dotted blue. 
Those two sets of twisted edges give rise to dividing $(M-2)$ real tropical curves by \Cref{ThM-2}.

\begin{figure}
\centering
\begin{subfigure}[t]{0.45\linewidth}
	\centering
	\includegraphics[width=\textwidth]{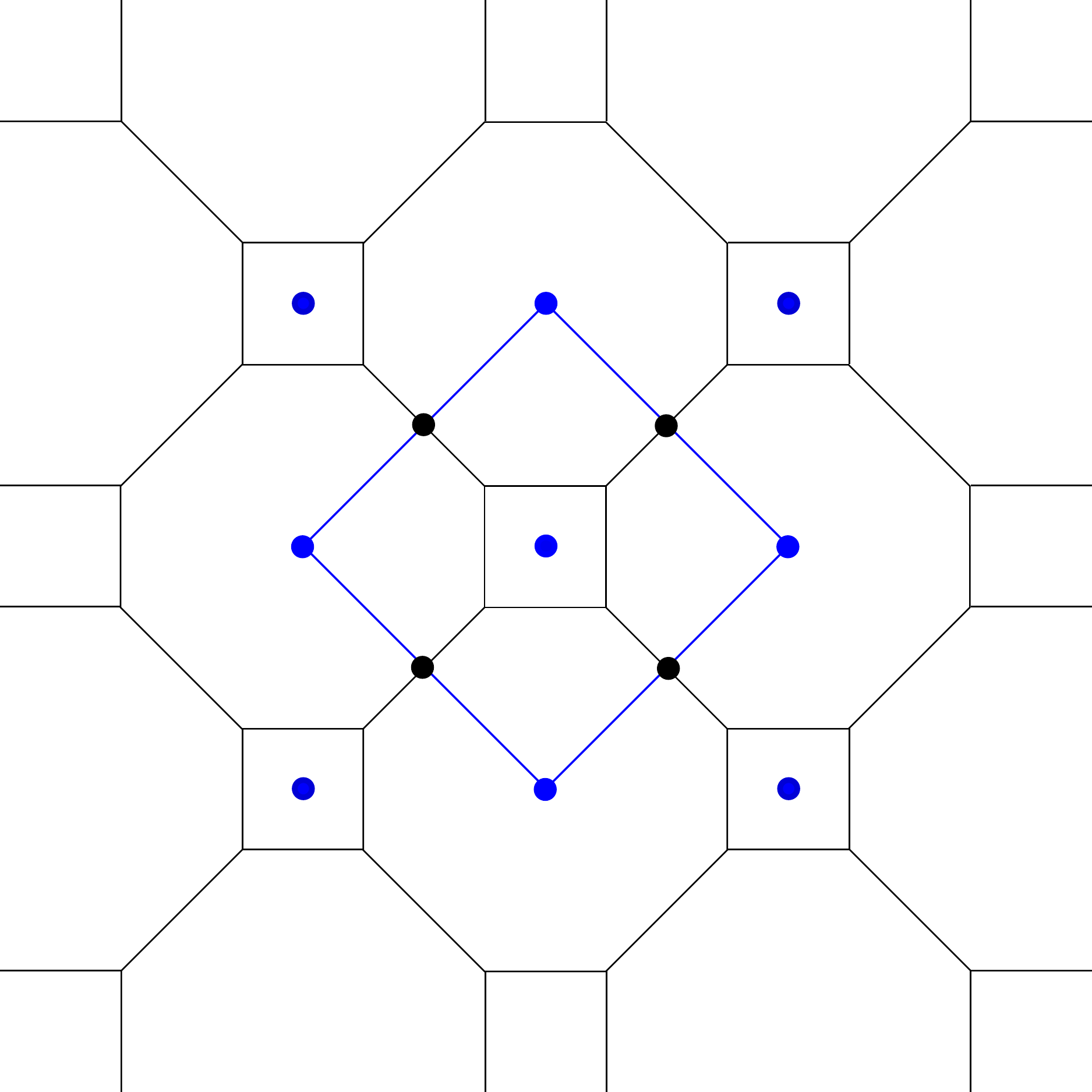}
	\caption{Complete bipartite dual graph of non-exposed twisted edges $K_{2,2}$.}
	\label{FigK22}
\end{subfigure}\hfill
\begin{subfigure}[t]{0.45\linewidth}
	\centering
	\includegraphics[width=\textwidth]{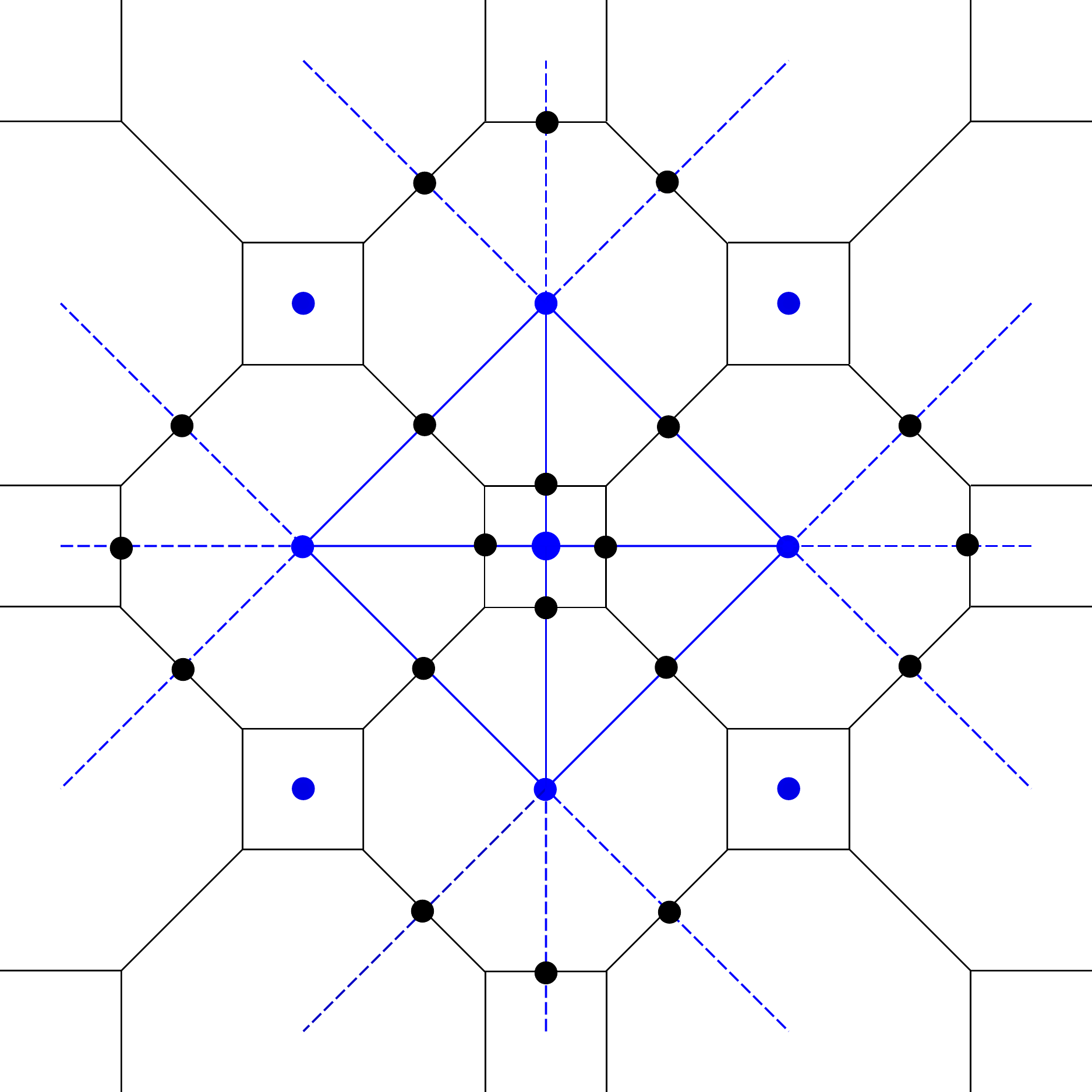}
	\caption{Complete tripartite dual graph of non-exposed twisted edges $K_{2,2,1}$.}
	\label{FigK221}
\end{subfigure}
\caption{Example \ref{ExM-2}.}
\label{FigM-2Div}
\end{figure}
\end{example}
%\begin{definition}
%{\color{red} Remove ?}
%We say that a configuration of twists $T\in \Edge^0 (C)$ on a non-singular tropical curve $C$ satisfying either \Cref{Item1M-2} or \Cref{Item2M-2} of Theorem \ref{ThM-2} is a \emph{$(M-2)$-dividing configuration of twists}.
%{\color{red} Locally admissible is not really relevant but locally dividing may be}
%\end{definition}

\begin{proof}[Proof of Theorem \ref{ThM-2}]
Let $\{ \gamma_1 ,\ldots , \gamma_g \}$ be the set of primitive cycles of $C$.
Let $Q_{\Gamma_T}$ be the Laplacian mod 2 of the dual graph of non-exposed twisted edges $\Gamma_T$ with respect to the ordering in $\{ \gamma_1 ,\ldots , \gamma_g \}$.
Seeing the primitive cycles as subsets of bounded edges of $C$, let $D_T$ be the $g\times g$ diagonal matrix with $\Z_2$-coefficients 
\[ (D_T)_{ii} = | \gamma_i \cap T \cap \Exp (C) | \mod 2   . \]
Then the matrix $A_T$ is given by $A_T = Q_{\Gamma_T} + D_T$.
The non-singular real tropical curve $(C,\E )$ is a $(M-2)$ real tropical curve if and only if $\rank A_T = 2$ by \Cref{ThConnectHom+}.
This is equivalent to the following condition on the column space $\Col (A_T)$:
 \[ \Col (A_T) = \Z_2 \langle v_1 , v_2 \rangle \subset \Z_2^g , \] with $v_1,v_2$ linearly independent.
Since $T$ is a dividing set of twisted edges, by \Cref{ThDividing+} the diagonal coefficients of $A_T$ are zero.
Up to renumbering the primitive cycles, we can then assume that $(A_T)_{1,2} = (A_T)_{2,1} = 1$.
We can then set $v_1 := (A_T)_{\bullet , 1}$ and $v_2 := (A_T)_{\bullet , 2}$.
Again since the diagonal coefficients of $A_T$ are zero, for any $j\neq j'$ such that $(A_T)_{\bullet , j} = (A_T)_{\bullet , j'}$ as a vector of $\Z_2^g$, the primitive cycles $\gamma_j , \gamma_{j'}$ satisfy $\gamma_j \cap \gamma_{j'} \cap T = \emptyset$. 
Similarly, by the condition $(A_T)_{1,2} = (A_T)_{2,1} = 1$, for any $j\neq j'$ such that \[ (A_T)_{\bullet , j} \neq (A_T)_{\bullet , j'} , \quad (A_T)_{\bullet , j} \neq 0 \text{ and } (A_T)_{\bullet , j'} \neq 0 \] as vectors of $\Z_2^g$, the primitive cycles $\gamma_j$ and $\gamma_{j'}$ satisfy $\gamma_j \cap \gamma_{j'} \cap T \neq \emptyset$. 

We now have two possible cases:
\begin{itemize}
\item either no column of $A_T$ is equal to the vector $v_1 + v_2$, and then the minor $M$ of $A_T$ given by all the non-zero rows and columns is the adjacency matrix of a complete bipartite graph $K_{m,n}$;
%, with its two sets of vertices given by 
%\begin{align*}
%V_1 & := \{ \gamma_1^\vee \} \cup \{ \gamma_j^\vee ~ | ~ v_1 \text{ satisfies } (\ref{Eq1}) \text{ for } j  \} \\ 
%V_2 & := \{ \gamma_2^\vee \} \cup \{ \gamma_j^\vee ~ | ~ v_2 \text{ satisfies } (\ref{Eq1}) \text{ for } j \} ;
%\end{align*} 
\item or there exists a column of $A_T$ equal to $v_1+v_2$, and then the minor $M$ of $A_T$ given by all the non-zero rows and columns is the adjacency matrix of a complete tripartite graph $K_{l,m,n}$.
%\begin{align*}
%V_1 & := \{ \gamma_1^\vee \} \cup \{ \gamma_j^\vee ~ | ~ v_1 \text{ satisfies } (\ref{Eq1}) \text{ for } j  \} , \\
%V_2 & := \{ \gamma_2^\vee \} \cup \{ \gamma_j^\vee ~ | ~ v_2 \text{ satisfies } (\ref{Eq1}) \text{ for } j \} , \\
%V_3 & :=  \{ \gamma_j^\vee ~ | ~ v_1+ v_2 \text{ satisfies } (\ref{Eq1}) \text{ for } j \} .
%\end{align*}
\end{itemize}
Adding the minor $M$ to the minor of $D_T$ given by the same rows and columns as $M$, we obtain the Laplacian mod 2 of the graph $K_{m,n}$ (respectively $K_{l,m,n}$) as a minor of $Q_{\Gamma_T}$, hence $K_{m,n}$ (respectively $K_{l,m,n}$) is a subgraph of $\Gamma_T$, and $\Gamma_T$ has only isolated vertices outside this subgraph.  
\end{proof}

\begin{corollary}
\label{CorM-2sDiv}
Let $(C,\E)$ be a non-singular real tropical curve with admissible set of twisted edges $T = T_1 \sqcup \ldots \sqcup T_s$, such that the dual graphs of non-exposed twisted edges $\Gamma_{T_i}$ are all disjoint and they all satisfy \Cref{ThM-2}.
Then $(C,\E )$ is a dividing $(M-2s)$ real tropical curve. 
\end{corollary}

\begin{proof}
By assumption, the $T_i$'s are all cycle-disjoint, hence there exists an ordering of the primitive cycles of $C$ such that the matrix $A_T$ satisfying this ordering is bloc-diagonal, with each non-zero bloc being the non-zero bloc of the matrix $A_{T_i}$ satisfying the same ordering of primitive cycles.
Then \[ \rank A_T = \sum_{i=1}^s \rank A_{T_i} = 2s \] by \Cref{ThM-2}, hence $(C,\E)$ is a $(M-2s)$ real tropical curve by \Cref{ThConnectHom+}.
Since every bloc has zero diagonal, we obtain by \Cref{ThDividing+} that $(C,\E)$ is dividing. 
\end{proof}

\section{Ovals in the real part of a tropical curve}
%\subsection{Real part of a tropical curve with only ovals}

\subsection{Partition of dividing configuration of twists}

We want to describe a set of generators of the $\Z_2$-vector space $\Div (C)$, in order to formulate relative topology criteria for dividing real tropical curves in the next sections.

\begin{definition}
\label{DefMulti+Circuit}
Let $C$ be a non-singular tropical curve in a non-singular projective tropical toric surface $\T \pr_\Sigma$.
A \emph{multi-bridge} $T$ on $C$ is a dividing set of twisted edges on $C$ with common direction modulo 2, such that $C\backslash T$ is disconnected and for any proper subset $T' \subsetneq T$, the graph $C\backslash T'$ is connected.
%  the set of dual edges $T^\vee \subset \Delta_C$ form a connected graph with only 2-valent vertices and possibly two distinct 1-valent vertices lying on $\partial \Delta_C = \partial \Delta \cap \Delta_C$.
A \emph{circuit} $T$ on $C$ is a multi-bridge on $C$ such that every edge of $T$ is non-exposed.
%the set of dual edges $T^\vee \subset \Delta_C$ form a (graph-theoretic) cycle not meeting the boundary $\partial \Delta$.
% admissible set of twisted edges on $C$ with common direction modulo 2 and such that the corresponding dual edges in $\Delta_C$ form a cycle with at most one vertex lying on $\partial \Delta_C$.
\end{definition}

\begin{proposition}
\label{PropPartition}
Let $C$ be a non-singular tropical curve in a non-singular projective tropical toric surface $\T \pr_\Sigma$.
Then the set of all multi-bridges on $C$ spans the $\Z_2$-vector space $\Div (C)$. 
Moreover, if every primitive cycle of $C$ has at most two edges of each possible direction modulo 2, then the set of all multi-bridges form a basis of $\Div (C)$.
\end{proposition}

\begin{proof}
Let $T \in \Div (C)$.
Assume $e_1$ is an edge of $T$ that does not belong to a primitive cycle of $C$, for $T$ seen as a subset of bounded edges of $C$.
The edge $e_1$ is then a multi-bridge, and the identification $\Edge^0 (C) \simeq \Z_2^{|\Edge^0 (C)|}$ sends $e_1$ to the configuration of twists $T_1 \in \Div (C)$.
Hence the sum $T+T_1$ is a configuration of twists $\widehat{T_1} \in \Div (C)$, since $\Div (C)$ is a $\Z_2$-space, such that $T_1$ and $\widehat{T_1}$ are disjoint seen as subsets of bounded edges of $C$.
Therefore, we can restrict to the case where all the edges of $T$ lie on primitive cycles of $C$.

Assume $e_2$ is an edge of $T$ lying on a primitive cycle $\gamma$ of $C$, for $T$ seen as a subset of bounded edges of $C$.
By \Cref{EqAdm+} and \Cref{EqDiv+}, there exists an edge $e_3$ distinct from $e_2$ with $\overrightarrow{e_2} = \overrightarrow{e_3} \mod 2$ in the set $\gamma \cap T$, for $\gamma$ and $T$ seen as subsets of bounded edges of $C$.
If both $e_2$ and $e_3$ are exposed, the set $\{ e_2 , e_3 \}$ is a multi-bridge on $C$, and for $T_2 := e_2 + e_3 \in \Div (C)$ the associated configuration of twists, the sum $T+T_2$ is a configuration of twists $\widehat{T_2} \in \Div (C)$ such that $T_2$ and $\widehat{T_2}$ are disjoint seen as subsets of bounded edges of $C$.
If at least one of the two edges $e_2 , e_3$ is non-exposed (say $e_3$), we use again \Cref{EqAdm+} and \Cref{EqDiv+} on the unique primitive cycle $\gamma '$ distinct from $\gamma$ containing $e_3$, and so on, until we cannot add a new edge in this manner (if $e_2$ is non-exposed, one has to add edges in the same way as for $e_3$).
The resulting subset of bounded edges $\{ e_2 , \ldots , e_k \}$ is a multi-bridge, and for $T_3 := \sum_{i=2}^k e_i\in \Div (C)$ the corresponding configuration of twists, the sum $T+T_3$ is a configuration of twists $\widehat{T_3} \in \Div (C)$ such that $T_3$ and $\widehat{T_3}$ are disjoint seen as subsets of bounded edges of $C$. 
Therefore, any configuration of twists $T \in \Div (C)$ is given as a sum of configuration of twists associated to multi-bridges on $C$.  
Moreover, if every primitive cycle of $C$ has at most two edges of each possible direction modulo 2, any two distinct multi-bridges on $C$ are disjoint, hence the configurations of twists associated to the multi-bridges on $C$ form a basis of $\Div (C)$. 
%Let $T$ be an admissible dividing set of twisted edges on $C$.
%For each primitive cycle $\gamma$ in $C$, the edges of $\gamma \cap T$ can be partitioned into pairs with common direction modulo 2.
%In particular, for $\gamma , \gamma '$ two adjacent primitive cycles (seen as subsets of bounded edges of $C$), if $\gamma \cap \gamma ' \cap T = \{ e_0 \}$, by twist-admissibility (\Cref{EqAdm+}) and dividing condition (\Cref{EqDiv+}), there exists an edge $e \neq e_0$ in $\gamma \cap T$ and an edge $e' \neq e_0$ in $\gamma ' \cap T$ such that $e, e' , e_0$ have common direction modulo 2.
%Hence the set of all multi-bridges (seen as vectors in $\Div (C)$) spans $\Div (C)$.
%Moreover, if every primitive cycle of $C$ has at most two edges of each possible direction modulo 2, then the choice of edges $e\neq e_0$ in $\gamma \cap T$ and $e' \neq e_0$ in $\gamma ' \cap T$ is unique.
%Thus the multi-bridges on $C$ do not have any common edge, hence the set of all multi-bridges (seen as vectors in $\Div (C)$) form a basis of $\Div (C)$.  
\end{proof}

\begin{corollary}
\label{CorPartitionCircuit}
Let $C$ be a non-singular tropical curve in a non-singular projective tropical toric surface $\T \pr_\Sigma$.
Then the set of all circuits on $C$ spans a $\Z_2$-vector subspace of $\Div (C)$.
\end{corollary}

\begin{definition}
Let $C$ be a non-singular tropical curve in $\T \pr_\Sigma$.
We denote by $\Circuit (C)$ the $\Z_2$-vector subspace of $\Div (C)$ spanned by the circuits on $C$.
%A multi-bridge $B$ on $C$ is said to be \emph{even} if the set $B^\vee \subset \Delta_C$ dual to the set of edges of $B$ contains an even integer point in $\Delta_C \cap \Z_2$, ie.~ both coordinates of the point are even, and either $B$ is a circuit or the two endpoints of $B^\vee$ are even.
\end{definition}

\subsection{Real tropical curves with only ovals}

\label{SectionAllOvals}

For $(C,\E)$ a non-singular real tropical curve in the non-singular tropical toric surface $\T \pr_\Sigma$, we say that a connected component $\R \rho_\E$ of the real part $\R C_\E$ is an \emph{oval} if $\R \rho_\E$ divides $\R \pr_\Sigma$ into two connected components.

\begin{lemma}
\label{Lem1valentOval}
Let $(C,\E)$ be a non-singular real tropical curve in a non-singular tropical projective toric surface $\T \pr_\Sigma$.
Let $\rho$ be a twisted cycle on $(C,\E )$ which does not contain any 1-valent vertex of $C$.
Then the real part $\R \rho_\E$ is an oval.
\end{lemma}

\begin{proof}
Since $\rho$ does not contain any 1-valent vertex of $C$, the real part $\R \rho_\E$ of $\rho$ is strictly contained in an orthant of the real algebraic torus $(\R^\times )^2 \subset \R \pr_\Sigma$. Since $\R \rho_\E$ is homeomorphic to the circle $S^1$, it must then divide that orthant into two connected components.
Therefore $\R \rho_\E$ is an oval.
\end{proof}

By definition, the homology class $[\R \rho_\E ]$ of an oval is trivial in the homology group of the ambient space $H_1 (\R \pr_\Sigma ; \Z_2)$.
Then if all the connected components of the real part $\R C_\E$ of a non-singular real tropical curve $(C,\E)$ are ovals, we obtain that the homology class $[\R C_\E]$, given as the sum of homology class of the connected components, is trivial in $H_1 (\R \pr_\Sigma ; \Z_2)$. 

If we only assume that the class $[\R C_\E]$ is trivial in $H_1 (\R \pr_\Sigma ; \Z_2)$, we have the following equivalent condition on the Newton polygon of $C$.

\begin{proposition}
\label{PropTrivialClassEquiv}
Let $(C,\E )$ be a non-singular real tropical curve in $\T \pr_\Sigma$, and let $\Delta$ be the Newton polygon of $C$ dual to the fan $\Sigma$.
The homology class $[\R C_\E]$ is trivial in $H_1 (\R \pr_\Sigma ; \Z_2)$ if and only if every edge of $\Delta$ has even lattice length.
\end{proposition}

\begin{definition}
\label{DefStrictEvenDegree}
A non-singular real tropical curve $(C,\E )$ in $\T \pr_\Sigma$ is said to be of \emph{strict even degree} if every edge of the associated Newton polygon $\Delta$ has even lattice length (in particular, the non-singular real tropical curve $(C,\E )$ satisfies \Cref{PropTrivialClassEquiv}).
\end{definition}

\begin{proof}[Proof of \Cref{PropTrivialClassEquiv}]
The homology class $[\R C_\E]$ is trivial in $H_1 (\R \pr_\Sigma ; \Z_2)$ if and only if $\R C_\E$ intersects each 1-dimensional face of the stratification of $\R \pr_\Sigma$ an even number of times.
Now there is a one-to-one correspondence between the 1-valent vertices of $C$ on a 1-dimensional stratum of $\T \pr_\Sigma$ and the intersection points of $\R C_\E$ with the corresponding 1-dimensional stratum of $\R \pr_\Sigma$.
Therefore the homology class $[\R C_\E]$ is trivial in $H_1 (\R \pr_\Sigma ; \Z_2)$ if and only if the non-singular tropical curve $C$ intersect each 1-dimensional stratum of $\T \pr_\Sigma$ in an even number of points.
The latter condition can be translated into the desired Newton polygon condition.
\end{proof}

\begin{corollary}
\label{CorOnlyOvalsNecessary}
Let $(C,\E )$ be a non-singular real tropical curve in $\T \pr_\Sigma$.
If all the connected components of the real part $\R C_\E$ are ovals, then $(C,\E)$ is of strict even degree in $\T \pr_\Sigma$.
\end{corollary}

We want to study the sufficient conditions for $\R C_\E$ to contain only ovals, in two distinct directions.
%The first examples are given by the following proposition.
%Let $\{ \gamma_1 , \ldots , \gamma_g \}$ be the primitive cycles of $C$.
%Since no edge of $C$ is twisted, the real part $\R C_{\E_\emptyset}$ has $g+1$ connected components by \Cref{ThHaas}.
%For each primitive cycle $\gamma_i$, there exists a twisted cycle $\rho_i$ in $(C,\E_\emptyset )$ with edge support the edges of $\gamma_i$.
%By \Cref{Lem1valentOval}, each real part $\R (\rho_i)_{\E_\emptyset}$ is an oval, hence the associated homology class $[\R (\rho_i)_{\E_\emptyset}]$ is trivial in $H_1 (\R \pr_\Sigma )$.
%Since all the edges of $\Delta$ have even lattice length, by \Cref{PropTrivialClassEquiv}, the homology class $[\R C_\E]$ is trivial in $H_1 (\R \pr_\Sigma )$.
%Let $\rho$ be the remaining twisted cycle on $(C,\E_\emptyset )$.
%We can write the homology class $[\R \rho_\E ]$ as the sum of classes (with $\Z_2$-coefficients) $[\R C_\E ] + \sum_i [\R (\rho_i)_{\E_\emptyset}]$, hence the class $[\R \rho_\E ]$ is trivial in $H_1 (\R \pr_\Sigma )$.
%Therefore $\R \rho_\E$ is an oval.
%\end{proof}
%Given \Cref{PropAllOvalsEmptySet}, we obtain that the zero element of the $\Z_2$-vector space $\Adm (C)$ gives rise to a real part with only ovals if the Newton polygon condition from \Cref{CorOnlyOvalsNecessary} is satisfied.
%We can then ask what is the biggest $\Z_2$-vector subspace of $\Adm (C)$ so that the corresponding real parts have only ovals.
%In the following propositions, we will study this question in two distinct directions.
The first idea comes from the well known fact that a non-singular real algebraic curve of even degree in $\pr^2$ has real part consisting only of ovals.
%We want to extend this in the tropical setting to non-singular tropical curves given as gluing of non-singular tropical curves of even degree in $\T \pr^2$. 

We denote by $\Delta_d$ the 2-dimensional lattice polytope given as the convex hull of the set $\{ (0,0) , (d,0) , (0,d) \} \subset \R^2_{\geq 0}$.
We will say that a lattice polytope $\Delta$ is a \emph{lattice transformation of} a lattice polytope $\Delta'$ if $\Delta$ is given as the image of $\Delta '$ via the composition of a lattice translation with a linear map in $\GL (2,\Z)$.

\begin{proposition}
\label{PropNewtonOval}
Let $(C,\E )$ be a non-singular real tropical curve of strict even degree in $\T \pr_\Sigma$.
Assume that the Newton polygon $\Delta$ of $C$, dual to the fan $\Sigma$, is a lattice transformation of some $\Delta_{2k}$, for $k\geq 1$.
%Assume moreover that for any 1-dimensional face $\tau$ of this coarsening dual to a multi-bridge $B$ on $C$, either all the edges of $B$ are twisted or none of them are twisted.
Then all the connected components of the real part $\R C_{\E}$ are ovals.
\end{proposition}

\begin{proof}
Assume first that $C$ is of even degree in $\T \pr^2$, so that the Newton polygon of $C$ is the lattice polytope $\Delta_{2k}$ for some $k\geq 1$.
The homology group $H_1 (\R \pr^2 ; \Z_2)$ is isomorphic to $\Z_2$, thus the real part $\R C_\E$ contains at most one pseudo-line, as they would otherwise intersect.
By assumption on the Newton polygon, for each 1-dimensional stratum $\tau$ of $\T \pr^2$, we have $| C\cap \tau | = 0 \mod 2$.
Now there is a one-to-one correspondence between the 1-valent vertices of $C$ on a 1-dimensional stratum of $\T \pr_\Sigma$ and the intersection points of $\R C_\E$ with the corresponding 1-dimensional stratum of $\R \pr_\Sigma$.
Therefore the homology class $[\R C_\E ]$ is trivial in $H_1 (\R \pr^2 ;\Z_2 )$. Thus, the real part $\R C_\E$ does not contain a pseudo-line, hence every connected component of $\R C_\E$ must be an oval.

Assume now that $\Delta$ is a lattice transformation of some $\Delta_{2k}$.
The lattice transformation induces a homeomorphism between $(\R \pr_\Sigma , \R C_\E )$ and $(\R \pr^2 , \R C_{\E'}')$, where $C'$ is a non-singular tropical curve of even degree in $\T \pr^2$.
As we saw above, the connected components of $\R C_{\E'}'$ are all ovals.
Therefore all the connected components of $\R C_\E$ are ovals. 
\end{proof} 

The second idea we want to explore is to consider the configurations of twists preserving the ``all ovals" property no matter the ambient tropical toric surface.

\begin{proposition}
\label{PropTwistBoundOval}
Let $(C,\E )$ be a non-singular real tropical curve of strict even degree in $\T \pr_\Sigma$.
Assume that every edge in the set of twisted edges $T$ on $C$ is non-exposed. 
Then all the connected components of the real part $\R C_{\E}$ are ovals.
\end{proposition}

\begin{proof}
By assumption, there exists a unique twisted cycle $\rho$ on $(C,\E )$ intersecting the 1-dimensional strata of $\T \pr_\Sigma$.
Since every edge of the Newton polygon $\Delta$ has even lattice length, the homology class $[\R C_\E]$ is trivial in $H_1 (\R \pr_\Sigma ; \Z_2 )$ by \Cref{PropTrivialClassEquiv}.
By \Cref{Lem1valentOval}, every twisted cycle in $(C,\E)$ distinct from $\rho$ has real part an oval.
We can write the homology class $[\R \rho_\E]$ as the sum of homology classes (with $\Z_2$-coefficients) $[\R C_\E] + \sum_{i=1}^r [\R (\rho_i)_\E ]$, with each $\rho_i$ being a twisted cycle on $(C,\E )$ distinct from $\rho$.
Since all the summands are trivial in $H_1 (\R \pr_\Sigma ; \Z_2)$, the class $[\R \rho_\E]$ is trivial in $H_1 (\R \pr_\Sigma ; \Z_2)$, hence $\R \rho_\E$ is an oval.  
\end{proof} 

\begin{example}
A non-singular real tropical curve $(C,\E )$ of strict even degree in $\T \pr_\Sigma$ has real part only ovals if its set of twisted edges $T$ is given as cycle-disjoint union of sets of twisted edges of the same form as in \Cref{FigK4}. 
\end{example}

\begin{corollary}
\label{CorAllOvalsCircuit}
Let $(C,\E)$ be a non-singular real tropical curve of strict even degree in $\T \pr_\Sigma$ with configuration of twists $T\in \Circuit (C)$.
Then all the connected components of the real part $\R C_{\E}$ are ovals.
\end{corollary}

\section{Even and odd ovals}
\label{SectionEvenOddOval}
Let $(C,\E)$ be a non-singular real tropical curve in $\T \pr_\Sigma$ with real part consisting only of ovals.
Let $\R \rho_\E$ be such an oval.
We say that $\R \rho_\E$ is \emph{even} if $\R \rho_\E$ is contained in an even number of discs in $\R \pr_\Sigma$ bounded by another oval of $\R C_\E$, and otherwise we say that $\R \rho_\E$ is \emph{odd}.

\subsection{Counting in terms of distribution of signs}

We want to count the number of even and odd ovals in $\R C_\E$ based on a distribution of signs $\delta$ associated to $\E$ on the dual subdivision $\Delta_C$ of $C$.
Recall that a non-singular tropical curve $C$ is of \emph{strict even degree} in $\T \pr_\Sigma$ if all the edges of the Newton polygon $\Delta$ of $C$ (dual to the fan $\Sigma$) are of even lattice length.

\begin{lemma}
\label{LemSignParity}
Let $C$ be a non-singular tropical curve of strict even degree in $\T \pr_\Sigma$, and let $\delta$ be a distribution of signs on $\Delta_C$.
Assume that the connected components of $\R C_\delta$ are all ovals.
Any two connected components $v^\vee ,(v')^\vee \subset (\T \pr_\Sigma \backslash C)^*$ are contained in the same number modulo 2 of ovals of $\R C_\delta$ if and only if the integer points $v,v' \in \Delta_C^* \cap \Z^2$ dual to $v^\vee ,(v')^\vee$ satisfy $\delta (v) = \delta (v ')$.
\end{lemma}

\begin{proof}
%The non-singular real tropical curve $(C,\E)$ is of even degree in $\T \pr^2$, hence every connected component of $\R C_\E$ is an oval.
Let $(\mathcal{C}_t)$ be a family of non-singular real algebraic curves converging to the tropical limit $(C,\delta )$.
For $t \in \R_{>0}$ small enough, the curve $\mathcal{C}_t$ is a non-singular real algebraic curve such that the pair $(\pr_\Sigma (\R) , \mathcal{C}_t (\R ))$ is homeomorphic to $(\R \pr_\Sigma, \R C_\delta )$ by \Cref{ThViro2}.
In particular, we have a one-to-one correspondence between the connected components of $\pr_\Sigma (\R) \backslash \mathcal{C}_t (\R )$ and the connected components of $\R \pr_\Sigma \backslash \R C_\delta $.
Let $f_t$ be a defining polynomial of $\mathcal{C}_t$, such that the signs of the monomials of $f_t$ induces the distribution of signs $\delta$ on $\Delta_C \cap \Z^2$.
Let $D_t$ be a connected component of $\pr_\Sigma (\R) \backslash \mathcal{C}_t (\R )$ and $v^\vee$ the corresponding connected component of $\R \pr_\Sigma \backslash \R C_\delta $.
Let $v \in \Delta_C^* \cap \Z^2$ be the integer point dual to $v^\vee$.
By definition of $\delta$, we have $f_t |_{D_t} > 0$ if and only if $\delta (v) > 0$, and similarly $f_t |_{D_t} < 0$ if and only if $\delta (v) < 0$.
% for every integer point $v \in \Delta_C^* \cap \Z^2$ dual to a connected component of $(\T \pr^2 \backslash C)^*$ contained in $v^\vee$.
The value of $f_t (x)$ along a path in $\pr_\Sigma (\R)$ intersecting $\mathcal{C}_t (\R )$ transversely changes of sign at each intersection of the path with an oval.
Therefore, any two connected components $v^\vee ,(v')^\vee \subset (\T \pr_\Sigma \backslash C)^*$ are contained in the same number of ovals modulo 2 of $\R C_\delta$ if and only if the dual integer points $v,v' \in \Delta_C^* \cap \Z^2$ satisfy $\delta (v) = \delta (v ')$.
\end{proof}

\begin{corollary}
\label{CorCompDistrib}
Let $C$ be a non-singular tropical curve of strict even degree in $\T \pr_\Sigma$, and let $\delta$ and $\delta '$ be distributions of signs on $\Delta_C$ such that all the connected components of the real parts $\R C_\delta$ and $\R C_{\delta '}$ are ovals.
Assume that there exists a connected component $v_0^\vee$ of $(\T \pr^2 \backslash C)^*$ contained in the same number modulo 2 of ovals in $\R C_\delta$ and $\R C_{\delta '}$.
Assume that $\delta (v_0) = \delta ' (v_0)$, with $v_0 \in \Delta_C^* \cap \Z^2$ the integer point dual to $v_0^\vee$.
Then a connected component $v^\vee$ in $(\T \pr^2 \backslash C)^*$ is contained in the same number modulo 2 of ovals of $\R C_\delta$ and $\R C_{\delta '}$ if and only if $\delta (v) = \delta ' (v)$, with $v \in \Delta_C^* \cap \Z^2$ the integer point dual to $v^\vee$.
\end{corollary}

\begin{proof}
Let $s$ be the number modulo 2 of ovals of $\R C_\delta$ containing $v_0^\vee$.
By assumption, the number $s$ is also the number modulo 2 of ovals of $\R C_{\delta '}$ containing $v_0^\vee$.
By \Cref{LemSignParity} a connected component $v^\vee$ is contained in $s$ ovals modulo 2 in $\R C_\delta$ if and only if $\delta (v) = \delta (v_0)$.
Again by \Cref{LemSignParity}, a connected component $v^\vee$ is contained in $s$ ovals modulo 2 in $\R C_{\delta '}$ if and only if $\delta '(v) = \delta ' (v_0)$.
Therefore a connected component $v^\vee$ is contained in the same number modulo 2 of ovals of $\R C_\delta$ and $\R C_{\delta '}$ if and only if $\delta (v) = \delta ' (v)$.
%Since $\delta (v_0) = \delta ' (v_0)$ by assumption, we obtain that $v^\vee$ is contained in an even number of ovals of both $\R C_\delta$ and $\R C_{\delta '}$ if and only if 
%\[ \delta (v) = \delta ' (v) = \delta (v_0) , \]
%and in a odd number of ovals of both $\R C_\delta$ and $\R C_{\delta '}$ if and only if 
%\[ \delta (v) = \delta ' (v) \neq \delta (v_0) . \]
\end{proof}

In order to use \Cref{LemSignParity} and \Cref{CorCompDistrib}, we show that there exists a distinguished twisted cycle in $(C,\E)$ with real part an even oval, called the \emph{special oval} in reference to \cite{haas1997real}. 

\begin{lemma}
\label{LemSpecialOval}
Let $(C,\E)$ be a non-singular real tropical curve of strict even degree in $\T \pr_\Sigma$ such that all the edges of the set of twisted edges $T$ on $C$ are non-exposed.
%all the connected components of the real part $\R C_\E$ are ovals.
Then there exists a unique twisted cycle $\rho_0$ on $(C,\E)$ containing some unbounded edges of $C$ and such that the real part $\R (\rho_0)_\E$ is an oval of $\R C_\E$ not contained in any other oval.
%Moreover, this twisted cycle is unique if the set $T^\vee \subset \Delta_C$ dual to the set of twisted edges $T$ on $C$ does not meet the edges of the Newton polygon $\Delta$ of $C$.
\end{lemma}

\begin{definition}
Let $(C,\E)$ be a non-singular real tropical curve satisfying the assumptions of \Cref{LemSpecialOval}.
% in $\T \pr_\Sigma$ such that all the connected components of $\R C_\E$ are ovals.
The unique twisted cycle $\rho_0$ in $(C,\E)$ satisfying the conditions of \Cref{LemSpecialOval} is called the \emph{special twisted cycle}, and the real part $\R (\rho_0)_\E$ is called the \emph{special oval}.  
\end{definition}

Note that we showed that the real part of the special twisted cycle is an oval in the proof of \Cref{PropTwistBoundOval}.

\begin{proof}[Proof of \Cref{LemSpecialOval}]
%The real part $\R C_\E$ is non-empty and $C$ is of even degree, hence $\R C_\E$ contains at least one oval, which is not contained in any other oval, hence at least one even oval. 
Since all the edges of $T$ are non-exposed, a closed twisted walk starting at a one-valent vertex $v$ of $C$ does not meet any twisted edge.
In particular, such a closed twisted walk goes only through exposed edges, and goes through every unbounded edge of $C$ twice.
Hence there exists a unique twisted cycle $\rho_0$ on $(C,\E )$ going through all the unbounded edges of $C$.
In that case, all ovals distinct from $\R (\rho_0)_\E$ are strictly contained in an orthant of the real algebraic torus $(\R^\times )^2 \subset \R \pr_\Sigma$ while the oval $\R (\rho_0)_\E$ intersects several orthants.
Hence the real part $\R (\rho_0)_\E$ cannot be contained in any other oval of $\R C_\E$, therefore $\rho_0$ is the unique special twisted cycle of $(C,\E)$.
\end{proof}

%\begin{corollary}
%\label{CorSpecialTwistedCycle}
%Let $C$ be a non-singular tropical curve of even degree in $\T \pr^2$, and let $T , T'\in \Div (C)$ so that $T\cap T' = \emptyset$ and $T'$ is a multi-bridge on $C$.
%Let $\rho$ be a special twisted cycle on $(C,T)$.
%Then there exists a special twisted cycle $\rho '$ on $(C, T+T')$ and a connected component $C_0 '$ of $C\backslash T'$ such that $\rho \cap C_0' \subset \rho ' \cap C_0$.  
%\end{corollary}
%
%\begin{proof}
%Assume first that $\rho$ does not meet $T'$.
%Then we can identify $\rho$ with the twisted cycle $\rho '$ on $(C,T+T')$ with same edge support, so that $\rho '$ is a special twisted cycle.
%
%Assume now that $\rho$ meets $T'$.
%Let $C_0 '$ and $C_1 '$ be the two connected components of $C\backslash T'$, and let $C_i = \overline{C_i' \cup T'}$ for $i=0,1$.
%There exists a twisted cycle $\rho_i$ on $(C_i, T|_{C_i})$ containing all the edges of $\rho \cap C_i$.
%We obtain a twisted cycle $\rho '$ on $(C,T + T')$ by gluing $\rho_0$ and $\rho_1$ along $T'$.
%In particular, the interior of an oval induced by $\rho '$ is given by gluing subsets of the interiors of the ovals induced by $\rho_0$ and $\rho_1$, and similarly for the exteriors.
%Now since $\rho$ is a special twisted cycle, we obtain by construction that an oval induced by $\rho'$ cannot be contained in any other oval.
%Therefore $\rho '$ is a special twisted cycle on $(C,T+T')$ since it contains some unbounded edges of $C$.
%\end{proof}

\subsection{Even circuits}

We consider the difference of number of even and odd ovals between two distinct real parts of a common non-singular real tropical curve, so that the corresponding sets of twisted edges differ by an \emph{even} circuit.

\begin{definition}
Let $ v = (v_1,v_2) \in \Delta_C \cap \Z^2$ be an integer point of the dual subdivision of a non-singular tropical curve $C$.
We say that $v$ is \emph{even} if both $v_1$ and $v_2$ are even, and we say that $v$ is \emph{odd} otherwise. 
A multi-bridge $B$ on $C$ is said to be \emph{even} if the dual set $B^\vee \subset \Delta_C$ contains an even integer point, and otherwise we say that $B$ is \emph{even-free}.
\end{definition}

Recall that $\Circuit (C)$ denotes the $\Z_2$-vector subspace of $\Div (C)$ such that a configuration of twists in $\Circuit (C)$ is given as a sum of circuits in $C$.
We denote by $\EvCirc (C) \subset \Circuit (C)$ the $\Z_2$-vector subspace spanned by even circuits. 
%Note that $\EvCirc (C)$ is a $\Z_2$-vector subspace
%Note that these are vector subspaces of $\Div (C)$ as they are generated by some subsets of the generators of $\Div (C)$, using \Cref{PropPartition}. 
From now on, for $(C,\E )$ a non-singular real tropical curve, we denote by $p_\E$ and $n_\E$ the number of even and odd ovals in $\R C_\E$.

\begin{proposition}
\label{PropEvenTwists}
Let $(C,\E )$ be a non-singular real tropical curve of strict even degree in $\T \pr_\Sigma$.
%, such that all the connected components of the real part $\R C_\E$ are ovals.
Assume that the set of twisted edges $T$ induced by $\E$ belongs to $\Circuit (C)$.
Let $T' \in \EvCirc (C)$ such that $T \cap T' = \emptyset$, and let $\E '$ be a real phase structure on $C$ inducing the configuration of twists $T+T' \in \Circuit (C)$.
%Let $p_i , n_i , i=1,2$ be the number of primitive cycles of $C$, meeting $T'$ and not meeting $T+T'$, and bounding a connected component of $\T \pr^2 \backslash C$ dual to an even, odd interior integer point in $\Delta_{T+T' ,i} \cap \Z^2$.  
Then we have $p_{\E '} \leq p_\E$ and $n_{\E '} \leq n_\E $.
%\begin{align*}
%p_{\E '} - n_1 -p_2 & \leq p_\E \leq p_{\E '} \\
%n_{\E '} - p_1 - n_2 & \leq n_\E \leq n_{\E '} .
%\end{align*}
%Moreover, if $\T \pr_\Sigma = \T \pr^2$, we can re
%replace $T' \in \EvCirc (C)$ by $T' \in \Even (C)$ in the assumptions.
\end{proposition}

In order to prove \Cref{PropEvenTwists}, we will need the following definition.

\begin{definition}
\label{DefZone}
Let $C$ be a non-singular tropical curve in a non-singular projective tropical toric surface $\T \pr_\Sigma$, and let $T$ be an admissible set of twisted edges on $C$.
%Let $C_0 ' , \ldots , C_k '$ be the connected components of $C\backslash T$, and let $C_i$ be the gluing of $C_i '$ with the edges of $T$ adjacent to $C_i '$, for $i= 0,\ldots ,k$.
%We denote by $\Delta_{C_i}$ is the subcomplex of $\Delta_C$ with 2-dimensional faces dual to the vertices of $C_i$ and edges dual to the edges of $C_i$.
%The of the dual subdivision $\Delta_C$ is given as 
%\[ \Delta_{C,T} := \bigsqcup_{i=0}^k \Delta_{C_i} / \sim , \]
%where , and $\sim$ is the gluing along the edges dual to the edges of $T$.
The \emph{zone decomposition} $\Delta_T$ of the Newton polygon $\Delta$ of $C$ is the subdivision of $\Delta$ induced by the set of edges $T^\vee$ dual to the edges of $T$.
We can write
\[ \Delta_T := \bigsqcup_{i=0}^k Z_{i}^T / \sim , \]
where $Z_i^T$ is a 2-dimensional face of $\Delta_T$ and $\sim$ is the gluing along the edges of $T^\vee$.
%dual to the edges of $T$. 
The 2-dimensional faces $Z_i^T$ are called \emph{zones} \cite{haas1997real}.
A \emph{zone partition} $Y_T := Y_0^T \sqcup Y_1^T$ is a partition of the zones of $\Delta_T$ such that any two zones intersecting in a 1-dimensional set belong to distinct $Y_j^T$. 
%the decomposition of the Newton polytope $\Delta$ of $C$ induced by the edges $e \subset \Delta_C$ dual to twisted edges $e^\vee \in T$.
\end{definition}

\begin{proof}[Proof of \Cref{PropEvenTwists}]
We want to construct an injective homomorphism
\[ \phi : H_1 (\R C_{\E '}; \Z_2 ) \rightarrow H_1 (\R C_\E ; \Z_2) \]
such that the image by $\phi$ of the class of an oval in $\R C_{\E '}$ of parity $s\in \{ \text{even, odd} \}$ is given as the sum of classes of some ovals in $\R C_\E$ of parity $s$.
By definition of multi-bridge, the graph $C\backslash T'$ has exactly two connected components $C_0 ' , C_1 '$.
Let $C_i := \overline{C_i ' \cup T'}$ for $i=0,1$.
Up to action of $\mathcal{R}^2$, we can assume that $\E$ and $\E '$ are equal on $C_1$.
Said otherwise, we can assume that $\R (C_1)_\E = \R (C_1)_{\E '}$.
Let $\varepsilon \in \Z_2^2$ be the direction modulo 2 of the edges of $T'$.
Seeing $\varepsilon$ as a symmetry in $\mathcal{R}^2$, we get that $\R (C_0)_\E = \varepsilon (\R (C_0)_{\E '})$.

Let $Z_i$ be the zone of $\Delta_{T'}$ containing the edges dual to edges of $C_i'$, for $\Delta_{T'}$ the zone decomposition induced by $T'$.
By \Cref{LemSpecialOval}, there exists a connected component $v_0^\vee$ of $(\T \pr_\Sigma \backslash C)^*$ lying neighbourly outside special ovals of both $\R C_\E$ and $\R C_{\E '}$.
Up to renumbering the zones of $\Delta_{T'}$, we can assume that the integer point $v_0$ dual to $v_0^\vee$ belong to the union of symmetric copies $Z_1^* \cap \Z^2$.
We can choose two distributions of signs $\delta$ and $\delta'$ on $\Delta_C \cap \Z^2$ inducing the real phase structures $\E$ and $\E '$ and such that $\delta = \delta'$ on $Z_1 \cap \Z^2$.
In particular, we obtain $\delta (v_0) = \delta ' (v_0)$, hence the assumptions of \Cref{CorCompDistrib} are satisfied.

Every edge of $T'$ is dual to an edge of $\Delta_C$ containing an even vertex.
We denote by $C_{i,j} '$ the connected components of $C \backslash (T \cup T')$, such that each $C_{i,j} '$ is contained in the connected component $C_i '$ of $C \backslash T'$.
Let $Z_{i,j}$ be the zone in $\Delta_{T+T'}$ containing the edges dual to edges of $C_{i,j}'$ for any $i,j$.
Since the distributions of signs $\delta$ and $\delta'$ are equal on $Z_1 \cap \Z^2$, we obtain that for any $Z_{i,j}$ in the zone decomposition $\Delta_{T+T'}$, the distributions of signs $\delta$ and $\delta'$ are either both of Harnack type or both of inverse Harnack type on $Z_{i,j} \cap \Z^2$.
Since $\R (C_0)_\E = \varepsilon (\R (C_0)_{\E '})$, we obtain in particular that for any integer point $v$ in the union of symmetric copies $Z_0^* \cap \Z^2$, the equality $\delta (v) = \delta' (\varepsilon (v))$ is satisfied. 
 
We start the construction of the homomorphism $\phi$ by considering the ovals given as real part of twisted cycles that do not meet the set of twisted edges $T'$.
Every twisted cycle of $(C,\E')$ contained in a connected component $C_i '$ of $C\backslash T'$, for $i=1,2$, can be identified with the twisted cycle in $(C,\E )$ with the same edge support.
If such a twisted cycle $\rho$ is contained in $C_1 '$, we can also identify the real parts $\R \rho_\E$ and $\R \rho_{\E'}$, and if $\rho$ is contained in $C_0 '$, we have $\R \rho_\E = \varepsilon (\R \rho_{\E'})$.

Assume that $\rho$ is a twisted cycle of $(C,\E ')$ and $(C,\E)$ which is contained in $C_1 '$.
Since the distribution of signs $\delta$ and $\delta '$ are equal on $Z_1 \cap \Z^2$, by \Cref{CorCompDistrib} the ovals $\R \rho_\E$ and $\R \rho_{\E'}$ are contained in the same number modulo 2 of ovals of $\R C_\E$ and $\R C_{\E '}$, respectively.
Therefore we set $\phi [\R \rho_{\E'}] = [\R \rho_\E]$ for every twisted cycle $\rho$ of $(C,\E ')$ contained in the connected component $C_1 '$.

Assume now that $\rho$ is a twisted cycle of $(C,\E ')$ and $(C,\E)$ contained in $C_0 '$.
Since $\R (C_0)_\E = \varepsilon (\R (C_0)_{\E '})$, the oval $\R \rho_\E$ of $\R C_\E$ is obtained from the oval $\R \rho_{\E '}$ of $\R C_{\E '}$ by applying the symmetry $\varepsilon \in \mathcal{R}^2$.
Then since $\delta (v) = \delta ' (\varepsilon (v))$ for every integer point $v \in Z_0^* \cap \Z^2$, by \Cref{CorCompDistrib} the ovals $\R \rho_\E$ and $\R \rho_{\E'}$ are contained in the same number modulo 2 of ovals of $\R C_\E$ and $\R C_{\E '}$, respectively.
Therefore we set $\phi [\R \rho_{\E'}] = [\R \rho_\E]$ for every twisted cycle $\rho$ of $(C,\E ')$ contained in the connected component $C_0 '$.

Let finally $\rho$ be a twisted cycle in $(C,\E ')$ going through an edge of $T'$.
%The twisted cycle $\rho$ is given as a gluing of twisted cycles on $C_1$ and $C_0$ along edges of $T'$. 
By assumption on $T'$, the edges of $\rho \cap T'$ lie in the boundary of connected components of $\T \pr_\Sigma \backslash C$ dual to even integer points in $\Delta_C \cap \Z^2$.
Recall that $\delta = \delta '$ on $Z_1^* \cap \Z^2$ and $\delta = \delta' \circ \varepsilon$ on $Z_0^* \cap \Z^2$.
Let $\R (\rho ')_\E$ be the real part of some twisted cycle $\rho '$ of $(C,\E)$ satisfying $(\rho ' \cap C_1 ) \subset \rho$, and let $\R (\rho '')_\E$ is the real part of some twisted cycle $\rho ''$ of $(C,\E)$ satisfying $(\rho '' \cap C_0 ) \subset \rho$.
Then for each even integer point $v \in \Delta_C \cap \Z^2$, there exists a symmetric copy $\varepsilon (v^\vee )$ of the connected component $v^\vee \subset (\T \pr_\Sigma \backslash C)$ such that $\varepsilon (v^\vee )$ lies either neighbourly inside or neighbourly outside $\R \rho_{\E '} , \R (\rho ')_\E$ and $\varepsilon(\R (\rho '')_\E)$.
Since $\delta , \delta '$ and $\delta' \circ \varepsilon$ are equal on the integer points $(T')^\vee \cap \Z^2$, by \Cref{CorCompDistrib}, the oval $\R \rho_{\E '}$ is contained in the same number modulo 2 of ovals in $\R C_{\E '}$ as the number modulo 2 of ovals of $\R C_{\E}$ containing $\R (\rho ')_\E$ and $\R (\rho '')_\E$. 
Therefore we set $\phi [\R \rho_{\E'}] = \Gamma \in H_1 (\R C_\E ; \Z_2 )$, for $\rho$ any twisted cycle of $(C,\E ')$ going through an edge of $T'$ and for $\Gamma$ the sum of classes $[\R (\rho''')_\E]$ of ovals of $\R C_\E$ with associated twisted cycles $\rho'''$ satisfying $(\rho''' \cap C_i) \subset \rho$ for some $i$.

Then the homomorphism $\phi$ is injective and preserve the parity of ovals, hence we obtain $p_{\E '} \leq p_\E$ and $n_{\E '} \leq n_\E $.
\end{proof}

\begin{remark}
In the case $\T \pr_\Sigma = \T \pr^2$, we can relax the assumptions of \Cref{PropEvenTwists} so that $T \in \Div (C)$ and $T'$ belong to the $\Z_2$-vector subspace generated by even multi-bridges, using \Cref{PropNewtonOval}.
However, one needs to be consider the (not necessarily unique) special oval in this case.
\end{remark}
 
\begin{remark} 
We obtain in \Cref{PropEvenTwists} a result similar to Haas' result for maximal curves in $\pr^2$ obtained via combinatorial patchworking \cite[Theorem 10.6.0.5]{haas1997real}.
These two results justify that in order to obtain examples of non-singular real tropical curves with many even or odd ovals, we can restrict to the case of even-free configuration of twists.
Since all the duals of configurations of twists satisfying \Cref{ThM-1} contain an even point, we can then restrict to even-free multi-bridges.
\end{remark}

\subsection{Even-free circuits}

Recall that a multi-bridge $B$ is \emph{even-free} if every integer point in $B^\vee \cap \Z^2$ is odd.
We denote by $\EvFree (C)$ the $\Z_2$-vector subspace of $\Div (C)$ generated by even-free multi-bridges, and we denote by $\EvFreeCirc (C)$ the $\Z_2$-vector subspace of $\Circuit (C)$ generated by even-free circuits.
%$\EvFree (C) \subset \Div (C)$ the $\Z_2$-vector subspace of $\Div (C)$ generated by even-free multi-bridges, and by 
%\begin{remark}
%The set $\EvFree (C)$ is induced from $\Div (C)$ by $l$ linear conditions, for $l$ at most the number of even integer points $v \in \Delta_C \cap \Z^2$ such that an interior edge of $\Delta_C$ is incident to $v$.
%In the case where $C$ is a non-singular tropical curve of degree $2k$ in $\T \pr^2$, we have $l\leq \binom{k+2}{2}$.
%Then the set $\EvFree (C) \subset \Div (C)$ is a $\Z_2$-subspace of $\Div (C)$ of codimension at most $l$.
%Note that the dual graph of a non-dividing configuration of twists must contain an even integer point, combining the conditions for non-dividing and twist-admissible.
%\end{remark}

\begin{definition}
Let $T \in \Circuit (C)$ a configuration of twists on a non-singular tropical curve $C$.
The unique zone $Z^T$ of $\Delta_T$ meeting the boundary edges of the Newton polygon $\Delta$ of $C$ is called the \emph{special zone} \cite{haas1997real}.
\end{definition}

\begin{proposition}
\label{LemEvenFreeCount}
Let $(C,\E )$ be a non-singular real tropical curve of strict even degree in $\T \pr_\Sigma$, with set of twisted edges $T \in \EvFreeCirc (C)$.
Let $Y_T := Y_1^T \sqcup Y_0^T$ be a zone partition such that $Y_1^T$ contains the special zone $Z^T$.
Let $p_j$ and $n_j$ be the number of even and odd interior integer points in $Y_j^T$.
Then 
\begin{align*}
p_\E & \geq n_1 + p_0 + 1 \\
n_\E & \geq p_1 + n_0 .
\end{align*}
Moreover, the real part $\R C_\E$ has at least $n_1 + p_0$ empty even ovals and $p_1+n_0$ empty odd ovals.
%If $\T \pr_\Sigma = \T \pr^2$, we can replace $T \in \EvFreeCirc (C)$ by $T \in \EvFree (C)$ in the assumptions.
\end{proposition}

\begin{proof}
Let $(v_0)^\vee$ be a connected component of $(\T \pr_\Sigma \backslash C)^*$ lying neighbourly outside the real part $\R (\rho_0 )_\E$ of the special twisted cycle $\rho_0$.
Since $\R (\rho_0 )_\E$ is a special oval, the connected component $v_0^\vee$ lies outside every oval of $\R C_\E$.
Let $\delta$ be a distribution of signs on $\Delta_C$ inducing the real phase structure $\E$.
Then by \Cref{LemSignParity}, a connected component $v^\vee \subset (\T \pr^2 \backslash C)^*$ is contained in an even number of ovals of $\R C_\E$ if and only if $\delta (v) = \delta (v_0)$, for $v \in \Delta_C^* \cap \Z^2$ the integer point dual to $v^\vee$.

Let $Z_1^T$ and $Z_0^T$ be two adjacent zones in the zone decomposition $\Delta_T$, that is the intersection $Z_1^T \cap Z_0^T$ is 1-dimensional.
Let $v'$ be an even integer point in $Z_1^T$ and let $v''$ be an even integer point in $Z_0^T$.
Since $T\in \EvFreeCirc (C)$, the distribution of signs $\delta$ must satisfy $\delta (v') \neq \delta (v'')$.
In particular, the distribution of signs $\delta$ is of Harnack type on $Z_1^T\cap \Z^2$ if and only if $\delta$ is of inverse Harnack type on $Z_0^T\cap \Z^2$.
Therefore we can assume that $\delta$ is of Harnack type on each zone in $Y_1^T$ and of inverse Harnack type on each zone in $Y_0^T$.

For each integer point $v \in \Delta_C \cap \Z^2$ contained in the interior of a zone of $\Delta_T$, there exists a twisted cycle $\rho_v$ in $(C,\E)$ bounding the dual connected component $v^\vee \subset \T \pr_\Sigma \backslash C$.
Moreover, there exists a unique symmetric copy $\varepsilon (v) \in \Delta_C^* \cap \Z^2$ of $v$ such that the oval $\R (\rho_{v})_\E$ bounds the symmetric copy $\varepsilon (v^\vee ) \subset (\T \pr_\Sigma \backslash C)^*$. 
Then the oval $\R (\rho_{v})_\E$ is even if and only if $\delta (\varepsilon (v)) = \delta (v_0)$.

We obtain the four following cases.
\begin{itemize}
\item If the point $v$ is even and contained in a zone of $Y_1^T$, then $\R (\rho_{v})_\E$ is an odd oval.
\item If the point $v$ is odd and contained in a zone of $Y_1^T$, then $\R (\rho_{v})_\E$ is an even oval.
\item If the point $v$ is even and contained in a zone of $Y_0^T$, then $\R (\rho_{v})_\E$ is an even oval.
\item If the point $v$ is odd and contained in a zone of $Y_0^T$, then $\R (\rho_{v})_\E$ is an odd oval.
\end{itemize}

Hence we get 
\begin{align*}
p_\E & \geq n_1 + p_0  \\
n_\E & \geq p_1 + n_0 .
\end{align*} 
Since the special twisted cycle $\rho_0$ cannot be obtained as a twisted cycle of the form $\rho_v$, we obtain then $p_\E \geq n_1 + p_0 + 1$.
The ``Moreover" statement comes from the fact that no oval of $\R C_\E$ can be contained in an oval of the form $\R (\rho_v)_\E$, since they each bound a single connected component of $(\T \pr_\Sigma \backslash C)^*$.
\end{proof}

\begin{remark}
If $\T \pr_\Sigma = \T \pr^2$, we can relax the assumption $T\in \EvFreeCirc (C)$ to $T\in \EvFree (C)$ in \Cref{LemEvenFreeCount}, by \Cref{PropNewtonOval}.
However, one needs to consider the (not necessarily unique) special twisted cycle in that case.
\end{remark}

\subsection{Even-free complete bipartite graphs}

For $(C,\E)$ a non-singular real tropical curve of strict even degree in $\T \pr_\Sigma$ with configuration of twists $T$ belonging to $\EvFreeCirc (C)$, we want to count the number of even and odd ovals of $\R C_\E$ not already counted by \Cref{LemEvenFreeCount}.
We restrict first to a case similar to the configuration of twists Itenberg and Haas used to construct counter-examples to Ragsdale conjecture (\cite{itenberg1993contre}, \cite{haas1995multilucarnes}, \cite{itenberg2001number}).

\begin{lemma}
\label{LemCircGraph}
Let $(C,\E)$ be a non-singular dividing $(M-2)$ real tropical curve in $\T \pr_\Sigma$, such that its set of twisted edges $T$ belong to $\Circuit (C)$.
Then $T$ is dual to either a complete bipartite graph $T^\vee \subset \Delta_C$ of the form $K_{2,2l}$ for some $l \geq 1$, or to a complete tripartite graph of the form $K_{2,2,2}$.
\end{lemma}

\begin{proof}
By assumption, all the edges of $T$ are non-exposed.
Thus, by \Cref{ThM-2}, we get that $T^\vee$ is either a complete planar bipartite or a complete planar tripartite graph.
By \Cref{RemComplete2}, the graph $T^\vee$ cannot contain a subgraph of the form $K_{3,3}$.
Hence $T^\vee$ is of the form
\begin{enumerate}
\item \label{ItemGraph1} either $K_{1,n}$, for some $n\geq 1$;
\item \label{ItemGraph2} or $K_{2,n}$, for some $n\geq 1$;
\item \label{ItemGraph5} or $K_{1,1,n}$, for some $n\geq 1$;
\item \label{ItemGraph3} or $K_{1,2,2}$;
\item \label{ItemGraph4} or $K_{2,2,2}$.
\end{enumerate}

A complete bipartite graph $K_{1,n}$ is a tree, hence it does not contain any cycle.
Thus $T^\vee$ cannot satisfy \Cref{ItemGraph1}.

If the dual graph of non-exposed twisted edges $\Gamma_T$ is a  graph $K_{1,1,n}$ (plus isolated vertices), then by twist-admissibility condition (\Cref{EqAdm+}) the set of twisted edges must contain some exposed edges.
Hence $T^\vee$ cannot satisfy \Cref{ItemGraph5}.

Similarly, as we can see in \Cref{FigK221}, if the dual graph of non-exposed twisted edges $\Gamma_T$ is a graph $K_{1,2,2}$ (plus isolated vertices), then by twist-admissibility condition (\Cref{EqAdm+}) the set of twisted edges must contain some exposed edges.
Hence $T^\vee$ cannot satisfy \Cref{ItemGraph3}.

Therefore, $T^\vee$ must satisfy either \Cref{ItemGraph2} or \Cref{ItemGraph4}.
Now if $T^\vee$ satisfies \Cref{ItemGraph2}, then $n$ must be even, as otherwise by twist-admissibility condition (\Cref{EqAdm+}) the set of twisted edges must contain some exposed edges.
Thus in the case of \Cref{ItemGraph2}, the graph $T^\vee$ must additionally be of the form $K_{2,2l}$ for some $l\geq 1$.
\end{proof}

\begin{lemma}
\label{LemEvenOddTwistBipartite}
Let $(C,\E)$ be a non-singular real tropical curve of strict even degree in $\T \pr_\Sigma$, such that its set of twisted edges $T$ belong to $\EvFreeCirc (C)$ and is dual to a complete bipartite graph $T^\vee \subset \Delta_C$ of the form $K_{2,2l}$, for some $l \geq 1$.
Let $Y_T := Y_1^T \sqcup Y_0^T$ be the zone partition such that $Y_1^T$ contains the special zone.
There are exactly $l+1$ twisted cycles meeting $T$ with real part an even oval and $l-1$ twisted cycles meeting $T$ with real part an odd oval.
\end{lemma}

\begin{proof}
By \Cref{ThM-2}, we get that $(C,\E)$ is a non-singular dividing $(M-2)$ real tropical curve, thus $\R C_\E$ has $g-1$ connected components, for $g$ the number of primitive cycles on $C$.
By assumption on $T$, there are $g-2l-2$ primitive cycles $\gamma$ on $C$ such that $\gamma \cap T = \emptyset$ (with $\gamma$ seen as a set of bounded edges), so that there are $g-2l-2$ twisted cycles on $(C,\E)$ bounding a single connected component of $\T \pr_\Sigma \backslash C$.
By \Cref{LemSpecialOval}, there is a unique twisted cycle on $(C,\E)$ going through all unbounded edges of $C$, and this twisted cycle does not meet $T$. 
Thus, there are exactly $g-2 - (g-2l-2) = 2l$ twisted cycles on $(C,\E)$ meeting $T$.
Their real parts are partitioned into even and odd ovals as follows.
Since $T^\vee$ is a graph of the form $K_{2,2l}$ and all edges of $T$ are non-exposed, we get that for each zone $Z$ in the zone decomposition $\Delta_T$, distinct from the special zone, there exists a unique twisted cycle $\rho_Z$ on $(C,\E )$ with edge support dual to the boundary $\partial Z$.

Since every edge of $T$ is non-exposed, the real part $\R (\rho_Z)_\E$ is contained in a single orthant $\varepsilon$ of $\R \pr_\Sigma$.
Moreover, the two integer points $v , v' \in \partial Z \cap \Z^2$ such that $\varepsilon (v^\vee)$ and $\varepsilon ((v')^\vee)$ lie neighbourly inside $\R (\rho_Z)_\E$ are exactly the integer points of $\partial Z$ belonging to the set of $2l$ vertices of $K_{2,2l}$, for $v^\vee , (v')^\vee \subset \T \pr_\Sigma \backslash C$ the connected components dual to $v$ and $v'$.

Let $\delta$ be a distribution of signs on $\Delta_C \cap \Z^2$ inducing the real phase structure $\E$, and let $v_0^\vee \in (\T \pr_\Sigma \backslash C)^*$ a connected component lying outside every oval of $\R C_\E$.
Since $T$ is even-free, we can assume that $\delta$ is of Harnack type on every zone of $Y_1^T$ and of inverse Harnack type on every zone of $Y_0^T$, see the proof of \Cref{LemEvenFreeCount}.

If $Z$ belongs to $Y_0^T$, we get that \[ \delta (v_0) \neq \delta (\varepsilon (v)) \text{ and } \delta (v_0) \neq \delta (\varepsilon (v')), \] so that $\R (\rho_Z)_\E$ is an even oval by \Cref{LemSignParity}.
Similarly, if $Z$ belongs to $Y_1^T$, we get that \[ \delta (v_0) = \delta (\varepsilon (v)) = \delta (\varepsilon (v')), \] so that $\R (\rho_Z)_\E$ is an odd oval by \Cref{LemSignParity}.

The last twisted cycle $\rho_{Z^T}$ has edge support dual the edges in $\partial Z^T \cap K_{2,2l}$, for $Z^T$ the special zone in $\Delta_T$.
The two integer points $v , v' \in \partial Z^T \cap \Z^2$ such that $\varepsilon (v^\vee)$ and $\varepsilon ((v')^\vee)$ lie neighbourly inside $\R (\rho_Z)_\E$ are exactly the integer points of $\partial Z$ belonging to the set of $2$ vertices of $K_{2,2l}$.
We get that \[ \delta (v_0) = \delta (\varepsilon (v)) = \delta (\varepsilon (v')), \] so that $\R (\rho_Z)_\E$ is an even oval by \Cref{LemSignParity}. 

Therefore, we just need to know the number of zones in each $Y_j^T$, distinct from the special zone, to count the number of even and odd ovals coming from twisted cycles meeting $T$.
There are $l$ zones in $Y_0^T$, which together with the special zone $Z^T$ give rise to $l+1$ even ovals.
There are $l-1$ zones in $Y_1^T$ distinct from $Z^T$, which give rise to $l-1$ odd ovals.
\end{proof}

\begin{theorem}
\label{ThEvenOddCountBipartite}
Let $(C,\E)$ be a non-singular dividing real tropical curve of strict even degree in $\T \pr_\Sigma$.
Assume that the set of twisted edges $T$ on $C$ induced by $\E$ is of the form $T := T_1 \sqcup \ldots \sqcup T_s \in \EvFreeCirc (C)$ such all $T_i$'s are cycle-disjoints and each $T_i$ is dual to a complete bipartite graph of the form $K_{2,2l_i}$.
Let $Y_T := Y_1^T \sqcup Y_0^T$ be the zone partition such that $Y_1^T$ contains the special zone $Z^T$.
Let $p_j$ and $n_j$ be the number of even and odd interior integer points in $Y_j^T$.
Then
\begin{align*}
p_\E & = n_1 + p_0 + 1 + \sum\limits_{i=1}^s (l_i+1) \\
n_\E & = p_1 + n_0 + \sum\limits_{i=1}^s (l_i-1) .
\end{align*}
\end{theorem}

\begin{proof}
By \Cref{LemEvenFreeCount}, the real part $\R C_\E$ has $n_1+p_0+1$ even ovals and $p_1+n_0$ odd ovals coming from twisted cycles not meeting $T$.
By \Cref{LemEvenOddTwistBipartite} applied on each $T_i$ (this is possible since all $T_i$'s are disjoint and cycle-disjoint), there are $l_i+1$ even ovals and $l_i-1$ odd ovals coming from twisted cycles meeting $T_i$.
Therefore we obtain the result.
\end{proof}

\subsection{Counter-examples to Ragsdale conjecture}

\begin{figure}
\centering
\begin{subfigure}[t]{0.45\textwidth}
	\centering
	\includegraphics[width=\textwidth]{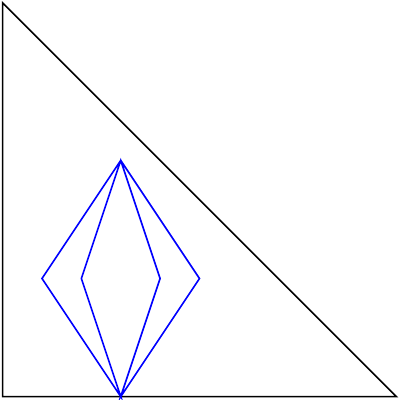}
	\caption{Itenberg's zone decomposition in degree 10.}
	\label{FigItenberg}
\end{subfigure}\hfill
\begin{subfigure}[t]{0.45\textwidth}
	\centering
	\includegraphics[width=0.69\textwidth]{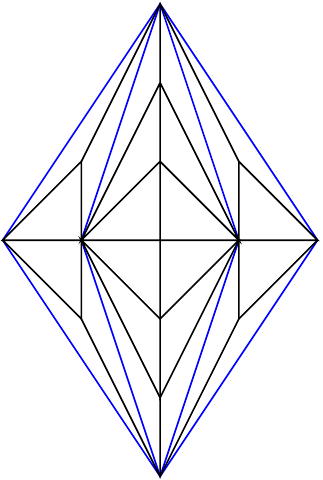}
	\caption{Subdivision inside Itenberg's configuration.}
	\label{FigItenbergZoom}
\end{subfigure}
\caption{Example \ref{ExItenberg}.}
\end{figure}

Recall that for a non-singular real algebraic curve $\mathcal{C}$ of degree $d = 2k$ in $\pr^2$, Ragsdale's conjecture states that the number of even ovals of $\mathcal{C}(\R )$ is less than or equal to $R(k)+1$, and the number of odd ovals of $\mathcal{C} (\R )$ is less or equal to $R (k)$, for $R(k):= \frac{3k(k-1)}{2}$ the number of odd interior integer points in the lattice polygon $\Delta_{2k}$.
We want to use the results of preceding sections in order to construct new counter-examples to Ragsdale's conjecture, in a similar manner as Itenberg \cite{itenberg1993contre}, \cite{itenberg2001number} and Haas \cite{haas1995multilucarnes}.

%We will consider a zone decomposition $\Delta_T$ associated to a non-singular tropical curve $C$ of degree $d$ in $\T \pr^2$ arranged in a similar way as the construction from \cite{itenberg2001number}.
%Recall from \Cref{ThM-2} that if the dual $T^\vee \subset \Delta (C)$ of a configuration of twists $T\in \Adm (C)$ is given as a complete bipartite graph (not meeting the boundary $\partial \Delta_C$), then a non-singular real tropical curve $(C,\E )$ with set of twisted edges $T$ is a dividing $(M-2)$ real tropical curve.
We first construct intermediate counter-examples coming from dividing $(M-2)$ real tropical curve $(C,\E )$ of degree $2k$ in $\T \pr^2$, for $k\geq 5$, with set of twisted edges dual to a specific complete bipartite graph.

%Let $\Delta_{2k}$ be the lattice polygon given as convex hull of the integer points $(0,0)$, $(2k,0)$ and $(0,2k)$, for $k\geq 5$.
Let $t$ be an integer satisfying 
\[  0 \leq t \leq \left\lfloor \frac{k-5}{2} \right\rfloor .\]
Let $m$ be an integer satisfying \[ 1 \leq m \leq \left\lfloor \frac{2k-1-4t}{6} \right\rfloor . \]
Let $K_{2,4m}$ be a complete bipartite graph in $\Delta_{2k}$ with two sets of vertices of the following form.
%For $t= 0 , \ldots , \left\lfloor \frac{k-5}{2} \right\rfloor$, we want to place a complete bipartite subgraph $K_{2,4l}, l\geq 1$ in $\Delta_C$ satisfying:
\begin{itemize}
\item the set of $4m$ vertices is of the form \[ \{ B_i = (b_i, 3+4t) , i = 1, \ldots , 4m \} \subset \Delta_{2k} \cap \Z^2 \] with $b_{i+1} = b_{i} + 1$ if $i= 1\mod 2$ and $b_{i+1} = b_{i} + 2$ if $i= 0 \mod 2$.  
\item The set of $2$ vertices is of the form \[ \{ A_1 = (a_1 , 6+4t ) ,  A_2  = (a_2 , 4t)   \} \subset \Delta_{2k} \cap \Z^2 , \]
such that $a_1, a_2$ are odd and 
\begin{align*}
a_1 + b_1 & \neq 0 \mod 3 , \quad a_1 + b_{4m}  \neq 0 \mod 3 \\
a_2 + b_1 & \neq 0 \mod 3 , \quad a_2 + b_{4m}  \neq 0 \mod 3.
\end{align*}
\end{itemize}  

The \emph{length} $L$ of a configuration $K_{2,4m}$ as described above is the difference $L:= b_{4m}  - b_1 = 4 + 6(m-1)$.
Note that the condition $1 \leq m \leq \left\lfloor \frac{2k-1-4t}{6} \right\rfloor$ is given so that the length of $K_{2,4m}$ is less than or equal to the horizontal length of $\Delta_{2k}$ at height $3+4t$.
Said otherwise, we have $L \leq 2k-3-4t$.
%, and we require that $e_t^1 = 1 \mod 6$ and $e_t^{4l} = 5 \mod 6$ in order to have an even-free configuration of twists.
%\begin{itemize}
%\item the set of $2$ vertices is of the form \[ \{ P_t = (p_t , 4t+6) ,  P_t ' = (p_t ' , 4t)   \} \subset \Delta_2k \cap \Z^2 ; \]
%
%\item the set of $4l$ vertices is given by integer points \[ \{ E_t^i = (e_t^i, 4t+3) , i = 1, \ldots , 4l \} \subset \Delta_C \cap \Z^2 \] with $e_t^{i+1} = e_t^{i} + 1$ if $i= 1\mod 2$ and $e_t^{i+1} = e_t^{i} + 2$ if $i= 0 \mod 2$.  
%\end{itemize}  
%the even integer points at the good place in the unique subdivision of $K_{2,4l}$ possible.

\begin{figure}
\centering
\begin{subfigure}[t]{0.45\textwidth}
	\centering
	\includegraphics[width=\textwidth]{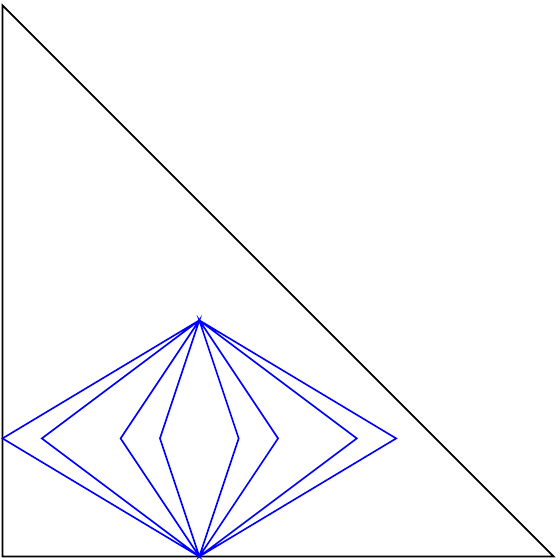}
	\caption{A zone decomposition in degree 14.}
	\label{FigCounterDeg14}
\end{subfigure}\hfill
\begin{subfigure}[t]{0.45\textwidth}
	\centering
	\includegraphics[width=\textwidth]{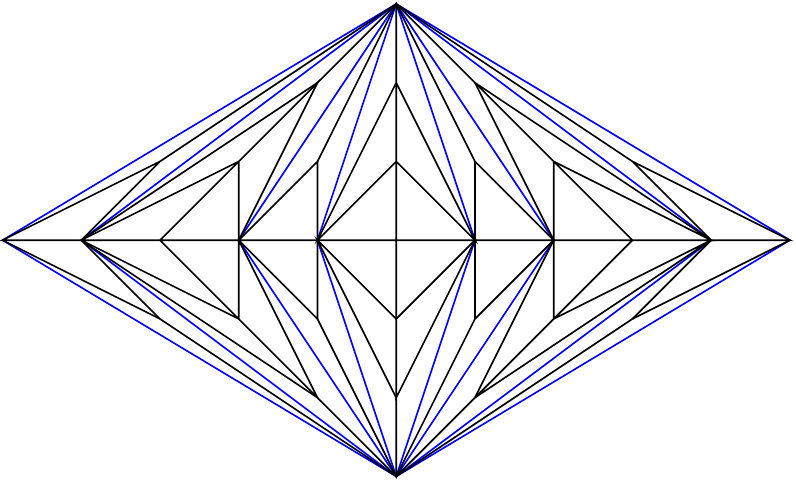}
	\caption{Subdivision inside the graph $K_{2,8}$.}
	\label{FigCounterDeg14Zoom}
\end{subfigure}
\label{FigCounter14}
\caption{Example \ref{ExRagsdale14}}
\end{figure}

\begin{lemma}
\label{LemCounterEx}
Let $(C,\E )$ be a non-singular tropical curve of degree $2k$ in $\T \pr^2$, such that the dual graph $T^\vee$ is a complete bipartite graph $K_{2,4m}$ satisfying the conditions above.
Then $(C,\E )$ is a dividing $(M-2)$ real tropical curve with real part $\R C_\E$ having $R (k) + 2m$ even ovals. 
\end{lemma}

\begin{proof}
Let $\Gamma$ be the subgraph of $T^\vee$ dual to the set of non-exposed twisted edges of $(C,\E)$. 
Then $\Gamma$ is either a complete bipartite graph $K_{2,4m}$, or a complete bipartite graph $K_{2,4m-1}$, or a complete bipartite graph $K_{1,4m}$, or a complete bipartite graph $K_{1,4m-1}$.
By \Cref{ThM-2}, we obtain that $(C,\E)$ is a dividing $(M-2)$ real tropical curve.

Assume first that $\Gamma$ is a complete bipartite graph $K_{2,4m}$.
In this case, we will be able to apply \Cref{ThEvenOddCountBipartite}, and we will show afterwards that the count remains unchanged in the other cases.

Let $Y$ be the union of all zones of $\Delta_T$ except the special zone $Z^T$.
%$\overline{K_{2,4l}}$ be the convex hull of the graph $K_{2,4l}$.
We can assume that the distribution of signs $\delta$ on the dual subdivision $\Delta_C$ is of Harnack type $Y_1^T$, for $Y_T = Y_1^T \sqcup Y_0^T$ the zone partition such that the special zone $Z^T$ belong to $Y_1^T$.
%outside $\overline{K_{2,4l}}$.
%Since $C$ is of even degree in $\T \pr^2$, by \Cref{PropNewtonOval} the real part $\R C_\E$ consists only of ovals.
%In particular, since the graph $K_{2,4m}$ satisfies \Cref{LemCircGraph} except for the fact that it meets the boundary of $\Delta_2k$, we will count the number of even and odd ovals of $\R C_\E$ as if the configuration of twists $T$ was belonged to $\EvFreeCirc(C)$, and then check that the count should remain the same even if the graph $K_{2,4m}$ has up to $2$ vertices on the boundary of $\Delta_{2k}$.  
%We can assume that all the twisted edges in $T$ are exposed by seeing the graph $K_{2,4l}$ inside a bigger Newton polygon $\Delta$ of $C$ (since $C$ is of even degree in $\T \pr^2$, we can preserve the desired properties on special twisted cycle and zone partition even if $K_{2,4l}$ meets the edges of $\Delta$).
%Let $\E_\emptyset$ be the real phase structure on $C$ inducing an empty set of twisted edges and such that the associated distribution of signs $\delta_{\emptyset}$ is equal to $\delta$ on $Z^T$.
%outside $\overline{K_{2,4l}}$.
The Newton polygon $\Delta_{2k}$ has $R(k)$ odd interior integer points, all lying in the closure of the union of zones $Y_1^T$.
%Since the empty set of twisted edges is even-free, the number of even ovals of $\R C_{\E_\emptyset}$ is equal to  by \Cref{ThmEvenOddCount}.
%There are $5 + 10 (l-1)$ odd interior integer points in the union of zones $Y \cap Y_1^T$, 
There are $4m+2$ odd integer points lying on the graph $K_{2,4l}$, hence we obtain $n_1 = R(k)- 4m -2$ with the notation from \Cref{ThEvenOddCountBipartite}. 
%hence $14l -3$ even ovals in $\R C_{\E_{\emptyset}}$ coming from integer points inside $\overline{K_{2,4l}}$, using again \Cref{ThmEvenOddCount}.
Moreover, there are $p_0 = 4m$ even integer points in $Y_0^T$, with the notation from \Cref{ThEvenOddCountBipartite}.
By \Cref{LemEvenFreeCount}, we obtain already that \[ p_\E \geq n_1 + p_0 + 1 = R (k)-1. \]
The remaining even ovals to count come from the twisted cycles going through an edge of $T$.
Since $T^\vee$ is a complete bipartite graph of the form $K_{2,4m}$, we obtain by \Cref{ThEvenOddCountBipartite} that the real part $\R C_\E$ has $R(k)-1 + (2m + 1) = R(k) +2m$ even ovals.

Now, if $T^\vee$ meets the boundary of $\Delta_{2k}$ in one or two of its vertices, the number of connected component of $\R C_\E$ is unchanged as we saw earlier, and all those connected components are ovals by \Cref{PropNewtonOval}.
Then there are either $4m$ or $4m+1$ odd integer points lying on the subgraph $\Gamma$, so that $\R C_\E$ has either $R(k)-4m$ or $R(k)-4m-1$ even ovals coming from twisted cycles with edge support a primitive cycle of $C$.
In the proof of \Cref{LemEvenOddTwistBipartite}, the special zone and the zones in $Y_0^T$ give rise to even ovals.
If one of the vertices of $T^\vee$ lying on the boundary of $\Delta_{2k}$ belongs to the set of $2$ vertices of the bipartite structure of $T^\vee$, then the even oval coming from the special zone is the same as the special oval.
If one of the vertices of $T^\vee$ lying on the boundary of $\Delta_{2k}$ belongs to the set of $4m$ vertices of the bipartite structure of $T^\vee$, then the even oval coming from the zone of $Y_0^T$ meeting $\Delta_{2k}$ on their boundaries is the same as the special oval.
Said otherwise, there is a unique twisted cycle on $(C,\E)$ going through every unbounded edge of $C$ and through the edges of $T$ dual to the ``exterior edges" of the graph $K_{2,4m}$, so that the real part of this twisted cycle cannot be contained in any other oval of $\R C_\E$, hence it is an even oval.

%Now, the distribution of signs $\delta$ is of inverse Harnack type on every 2-dimensional face of $\Delta_T$ containing only even interior integer points and of Harnack type on every other 2-dimensional face of $\Delta_T$.
%Since $T \in \EvFreeCirc (C)$ (up to assuming that $K_{2,4l}$ does not meet the edges of $\Delta$), we can use \Cref{ThmEvenOddCount} to count the contributions to even and odd ovals.
%In $\overline{K_{2,4l}} \backslash K_{2,4l}$, we have $4l$ points contributing to $p_2$, as well as $5+ 10 (l-1)$ points contributing to $n_1$.
%The set of twisted edges $T$ is a disjoint union of circuits $T_1 \sqcup \ldots \sqcup T_{2l}$ with dual meeting at exactly two common integer points.
%Then $T_1$ is the only circuit contributing to $c_1$, with the notation from \Cref{ThmEvenOddCount}.
%For each $T_i$ with $i\neq 1$, the set $(T_i)^\vee \backslash ((T_1 + \ldots + T_{i-1})^\vee \cap \Z^2)$ consists of two connected components containing each a single (odd) integer point.
%Then $d_1 = 2(l-1)$ with the notation from \Cref{ThmEvenOddCount}.
%Therefore, by \Cref{ThmEvenOddCount}, the real part $\R C_\E$ has $R(k) + 2 c_1 + d_1 = R(k)+2l$ even ovals. 
%
%Moreover, there are $2l+1$ left-right cycles on the graph $K_{2,4l}$  contributing to $c_1$ (seeing $K_{2,4l}$ as a union of $2l$ circuits with 4 edges).
%Therefore by \Cref{ThmEvenOddCount} there are $16l-4$ even ovals in $\R C_\E$ coming from the integer points inside $\overline{K_{2,4l}}$.
%We obtain that $\R C_\E$ has $R (k) + 1 + 16l-4 - 14l + 3 = R (k) + 2l$ even ovals.   
\end{proof}

\begin{example}
\label{ExItenberg}
Let $(C,\E)$ be a non-singular real tropical curve of degree 10 in $\T \pr^2$ such that the zone decomposition $\Delta_T$ is induced by a complete bipartite graph $K_{2,4}$ as given in Figure \ref{FigItenberg}.
The dual graph of non-exposed twisted edges $\Gamma_T$ consists of a complete bipartite graph $K_{1,4}$ plus isolated vertices, and every cycle of $C$ has an even number of twists, so by Theorem \ref{ThM-2} the real tropical curve $(C,\E )$ is a dividing $(M-2)$ real tropical curve.
The integer point at the center of the subdivision given in Figure \ref{FigCounterDeg14Zoom} is $(3,3) \in \Delta_C \cap \Z^2$, and the extremal points are $(3,0), (3,6) , (1,3)$ and $(5,3)$.
This configuration of twists is Itenberg's counter-example to Ragsdale conjecture \cite{itenberg1993contre}.
By \Cref{LemCounterEx}, the real part $\R C_\E$ has $R(5) +2 = 32$ even ovals, hence one even oval more than Ragsdale bound in degree 10.
\end{example}

\begin{table}
\[ \begin{array}{|l|c|c|c| }
\hline
 & x_t \equiv 0 \mod 3  & x_t \equiv 1 \mod 3  & x_t \equiv 2 \mod 3  \\ 
\hline
P_t & (3,4t+6) & ( 3,4t+6) & (3 , 4t+6) \\ 
P_t ' & (3 , 4t) & ( 3,4t) & (3,4t) \\ 
E_t^1 & (1,4t+3) & (1,4t+3) & (1,4t+3) \\ 
E_t^{4l} & (x_t -4 , 4t+3) & (x_t -2 , 4t+3) & (x_t,4t+3) \\
\hline 
\end{array}  
  \]
\caption{Case $t$ even.}
\label{Table1}  
\end{table}

\begin{table}
\[ \begin{array}{|l|c|c|c| }
\hline
 & x_t \equiv 0 \mod 3  & x_t \equiv 1 \mod 3  & x_t \equiv 2 \mod 3  \\ 
\hline
P_t & (x_t -3,4t+6) & (x_t -4 ,4t+6) & (x_t - 5,4t+6) \\ 
P_t ' & ( x_t +3, 4t) & (x_t + 2 ,4t) & (x_t +1,4t) \\ 
E_t^1 & (1,4t+3) & (1,4t+3) & (1,4t +3) \\ 
E_t^{4l} & ( x_t - 4 , 4t+3) & (x_t - 2 , 4t+3) & ( x_t ,4t+3) \\
\hline 
\end{array}  
  \]
\caption{Case $t$ odd.}
\label{Table2}  
\end{table}

\begin{example}
\label{ExRagsdale14}
Let $(C,\E)$ be a non-singular real tropical curve of degree 14 in $\T \pr^2$ such that the zone decomposition $\Delta_T$ is induced by a complete bipartite graph $K_{2,8}$ as given in Figure \ref{FigCounterDeg14}.
The dual graph of non-exposed twisted edges $\Gamma_T$ consists of a complete bipartite graph $K_{1,7}$ plus isolated vertices, and every cycle of $C$ has an even number of twists, so by Theorem \ref{ThM-2} the non-singular real tropical curve $(C,\E )$ is a dividing $(M-2)$ real tropical curve.
The integer point at the center of the subdivision given in Figure \ref{FigCounterDeg14Zoom} is $(5,3) \in \Delta_C \cap \Z^2$, and the extremal points are $(5,0), (5,6) , (0,3)$ and $(10,3)$.
By \Cref{LemCounterEx}, the real part $\R C_\E$ has $R(7)+4 = 67$ even ovals, hence 3 even oval more than Ragsdale's bound.
\end{example}

Let us now describe a construction of counter-examples of Ragsdale conjecture in degree $2k \geq 10$, by taking a disjoint union of complete bipartite graphs of the form $K_{2,4m_t}$ in $\Delta_{2k}$ satisfying the conditions of \Cref{LemCounterEx}, for \[ t \in \left\{ 0, \ldots , \left\lfloor \frac{k-5}{2} \right\rfloor \right\}  ,\] in a similar fashion as Itenberg's construction in \cite{itenberg2001number}.
Let $x_t = 2k - (4t+3)$, so that the integer points $(x_t , 4t+3)$ lie on the boundary $\partial \Delta_{2k}$.
Each bipartite graph $K_{2,4l_t}$ will have its set of $4m_t$ vertices lying on the horizontal line at height $4t+3$, so that the length $L_t$ of $K_{2,4m_t}$ will be the maximal integer satisfying $L_t \leq x_t$ and $L_t = 4 \mod 6$.
 
Since $x_t = 1 \mod 2$, we can determine suitable coordinates for the extremal vertices $B_1^t , B_{4m_t}^t$ of $K_{2,4l_t}$ in terms $x_t \mod 3$ instead of $x_t \mod 6$. 
We give in Tables \ref{Table1} and \ref{Table2} the coordinates of the extremal points for our construction, where Table \ref{Table1} correspond to the case where $t = 0 \mod 2$ and Table \ref{Table2} correspond to the case where $t = 1 \mod 2$.   

\begin{theorem}[\Cref{IntroThRagsdale}]
\label{ThCounterRagsdale}
Let $(C,\E )$ be a non-singular dividing tropical curve of degree $2k \geq 10$ in $\T \pr^2$ with zone decomposition $\Delta_T$ induced by the disjoint union of graphs 
\[ \bigsqcup\limits_{t=0}^{\left\lfloor \frac{k-5}{2} \right\rfloor} K_{2,4m_t}  \]
with vertices described by \Cref{Table1} and \Cref{Table2}.
%Let
%\[ N_i := \# \left\{ t\in \left\{ 0 , \ldots , \left\lfloor \frac{k-5}{2} \right\rfloor  \right\}  ~ | ~ x_t \equiv i \mod 3 \right\} . \]
Then $(C,\E)$ is a dividing $(M-2 \left\lfloor \frac{k-3}{2} \right\rfloor )$ real tropical curve with real part $\R C_\E$ having 
 \[ R(k) + 1 + \frac{k^2 - 5k + s(k)}{6}\] 
% \[ R(k) + 1 + \frac{k^2 - 7k + 12}{6} + \frac{2}{3} N_1 + \frac{4}{3} N_2  \] 
even ovals, with $0 \leq s(k) \leq 10$ determined by the value of $k \mod 6$. 
The number of even ovals depending on the value of $k \mod 6$ are displayed in \Cref{Table3}. 
\end{theorem}

\begin{table}
\[ \begin{array}{|c|c|}
\hline
k \mod 6 & p_\E \\ 
\hline
0 & R(k) + 1 + \frac{k^2 - 5k}{6} \\ 
1 & R(k) + 1 + \frac{k^2 - 5k+10}{6} \\
2 & R(k) + 1 + \frac{k^2 - 5k+8}{6} \\
3,5 & R(k) + 1 + \frac{k^2 - 5k+6}{6} \\
4 & R(k) + 1 + \frac{k^2 - 5k+4}{6} \\
\hline 
\end{array}  
  \]
\caption{Values in \Cref{ThCounterRagsdale}.}
\label{Table3}  
\end{table}
%\begin{remark}
%The mean value of the expression in Theorem \ref{ThCounterRagsdale} is 
%\[ R(k) + 1 + \frac{k^2 - 5k + 5 + (-1)^k }{6} .  \]
%This expression gives the correct number of even ovals for $k\equiv 3 , 4, 5 \mod 6$, the floor of this expression gives the correct number of even ovals for $k\equiv 0 \mod 6$ and the ceil of this expression gives the correct number of even ovals for $k \equiv 1,2 \mod 6$. 
%\end{remark}

\begin{proof}
By \Cref{CorM-2sDiv}, we get that $(C,\E )$ is a dividing $(M-2 \left\lfloor \frac{k-3}{2} \right\rfloor )$ real tropical curve.

Using Tables \ref{Table1} and \ref{Table2}, we want to compute the number $m_t$ for each graph $K_{2,4 m_t}$, then compute the total gain of even ovals using Lemma \ref{LemCounterEx}.
The length of a configuration $K_{2,4 m_t}$ is equal to $4+6(m_t -1) = 6 m_t - 2$.
If $x_t \equiv 0 \mod 3$ (or $x_t \equiv 1 \mod 3$, or $x_t \equiv 2 \mod 3$), the length is given by $2k -4t - 8$ (or $2k -4t - 6$ , or $2k -4t - 4$).
We then obtain that
\[ l_t = \begin{cases} \frac{k-2t-3}{3} \\
\text{resp. } \frac{k-2t-2}{3} \\
\text{resp. } \frac{k-2t-1}{3} .
\end{cases}  \]
By \Cref{LemCounterEx}, the gain of even ovals of a configuration $K_{2,4 m_t }$ is $2m_t -1$.
Since \[ m_t \geq \frac{k-2t-3}{3}  \]
for all $t = 0 , \ldots , \left\lfloor \frac{k-5}{2} \right \rfloor$, we compute the gain by first computing the sum \[ \sum\limits_{t = 0}^{\lfloor \frac{k-5}{2} \rfloor} \left( 2 \frac{k-2t-3}{3} -1 \right) , \] then add a correcting term.  
Assume that $k$ is odd (the sum will be the same with $k$ even). 
We get
\begin{align*}
\sum\limits_{t = 0}^{\frac{k-5}{2}} \left( 2 \left( \frac{k-2t-3}{3} \right) -1 \right) & = \frac{2}{3} \left( \sum\limits_{t = 0}^{\frac{k-5}{2} } (k-3) \right)  - \frac{4}{3} \left( \sum\limits_{t = 0}^{ \frac{k-5}{2} } t \right) -  \frac{k-3}{2} \\
& = \frac{k^2 - 6k + 9}{3} - \frac{k^2 -8k +15}{6} - \frac{k-3}{2} \\
& = \frac{k^2 - 7k +12}{6} .
\end{align*}
The correcting term is given as \[ 2 \left( \sum_{\substack{t = 0 \\ x_t \equiv 1 [3] }}^{\lfloor \frac{k-5}{2} \rfloor} \left( \frac{1}{3} \right) + \sum_{\substack{t = 0 \\ x_t \equiv 2 [3] }}^{\lfloor \frac{k-5}{2} \rfloor } \left( \frac{2}{3} \right) \right) = \frac{2}{3} N_1 + \frac{4}{3} N_2 , \]
for $N_i := \left| \left\{ t \in \left\{ 0, \ldots , \left\lfloor \frac{k-5}{2} \right\rfloor \right\} ~ : ~ x_t \equiv i \mod 3 \right\} \right|$.

We compute the correcting term $\frac{2}{3} N_1 + \frac{4}{3} N_2$ depending of the value of $k \mod 6$ as follows.

\begin{itemize}
\item Assume $k = 0 \mod 6$.
	Then $x_0 = 0 \mod 3$, hence the sequence $(x_t \mod 3)_{t\geq 0}$ has the form $ 0, 2, 1, 0, 2, 1, \cdots$. 
	Now $\left\lfloor \frac{k-5}{2} \right\rfloor = \frac{k-6}{2} = 0 \mod 3$, hence $x_{\left\lfloor \frac{k-5}{2} \right\rfloor} = 0 \mod 3$.
	Therefore $N_1 = N_2 = \frac{k-6}{6}$.
\item Assume $k = 1 \mod 6$. 
	Then $x_0 = 2 \mod 3$, hence the sequence $(x_t \mod 3)_{t\geq 0}$ has the form $ 2, 1, 0, 2, 1, 0, \cdots$. 
	Now $\left\lfloor \frac{k-5}{2} \right\rfloor = \frac{k-5}{2} = 1 \mod 3$, hence $x_{\left\lfloor \frac{k-5}{2} \right\rfloor} = 1 \mod 3$.
	Therefore $N_1 = N_2 = \frac{k-1}{6}$.
\item Assume $k = 2 \mod 6$. 
	Then $x_0 = 2 \mod 3$, hence the sequence $(x_t \mod 3)_{t\geq 0}$ has the form $ 2, 1, 0, 2, 1, 0, \cdots$. 
	Now $\left\lfloor \frac{k-5}{2} \right\rfloor = \frac{k-6}{2} = 1 \mod 3$, hence $x_{\left\lfloor \frac{k-5}{2} \right\rfloor} = 1 \mod 3$. 
	Therefore $N_1 = N_2 = \frac{k-2}{6}$.
\item Assume $k = 3 \mod 6$. 
	Then $x_0 = 0 \mod 3$, hence the sequence $(x_t \mod 3)_{t\geq 0}$ has the form $ 0, 2, 1, 0, 2, 1, \cdots$. 
	Now $\left\lfloor \frac{k-5}{2} \right\rfloor = \frac{k-5}{2} = 2 \mod 3$, hence $x_{\left\lfloor \frac{k-5}{2} \right\rfloor} = 1 \mod 3$.
	Therefore $N_1 = N_2 = \frac{k-3}{6}$.
\item Assume $k = 4 \mod 6$. 
	Then $x_0 = 1 \mod 3$, hence the sequence $(x_t \mod 3)_{t\geq 0}$ has the form $1, 0, 2, 1, 0, 2, \cdots$. 
	Now $\left\lfloor \frac{k-5}{2} \right\rfloor = \frac{k-6}{2} = 2 \mod 3$, hence $x_{\left\lfloor \frac{k-5}{2} \right\rfloor} = 2 \mod 3$.
	Therefore $N_1 = N_2 = \frac{k-4}{6}$.
\item Assume $k = 5 \mod 6$. 
	Then $x_0 = 1 \mod 3$, hence the sequence $(x_t \mod 3)_{t\geq 0}$ has the form $1, 0, 2, 1, 0, 2,  \cdots$. 
	Now $\left\lfloor \frac{k-5}{2} \right\rfloor = \frac{k-5}{2} = 0 \mod 3$, hence $x_{\left\lfloor \frac{k-5}{2} \right\rfloor} = 1 \mod 3$.
	Therefore $N_1 = \frac{k+1}{6}$ and $N_2 = \frac{k-5}{6}$.
\end{itemize}  
\end{proof}

\begin{figure}
\centering
\begin{subfigure}[t]{0.45\textwidth}
\centering
\includegraphics[width=\textwidth]{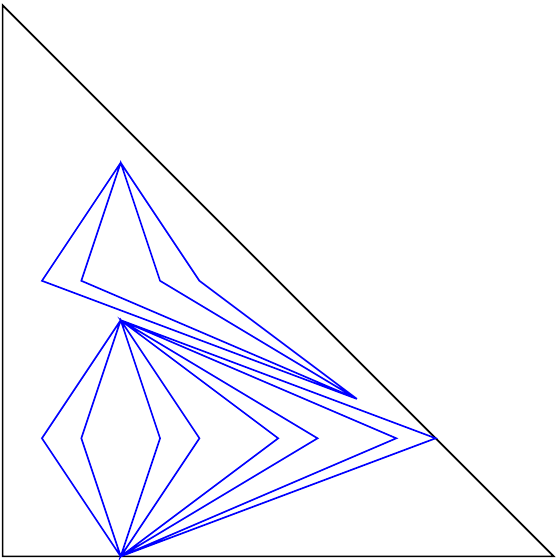}
\caption{Construction in Example \ref{ExDeg14}.}
\label{FigRagsdale14}
\end{subfigure}\hfill
\begin{subfigure}[t]{0.45\textwidth}
\centering
\includegraphics[width=\textwidth]{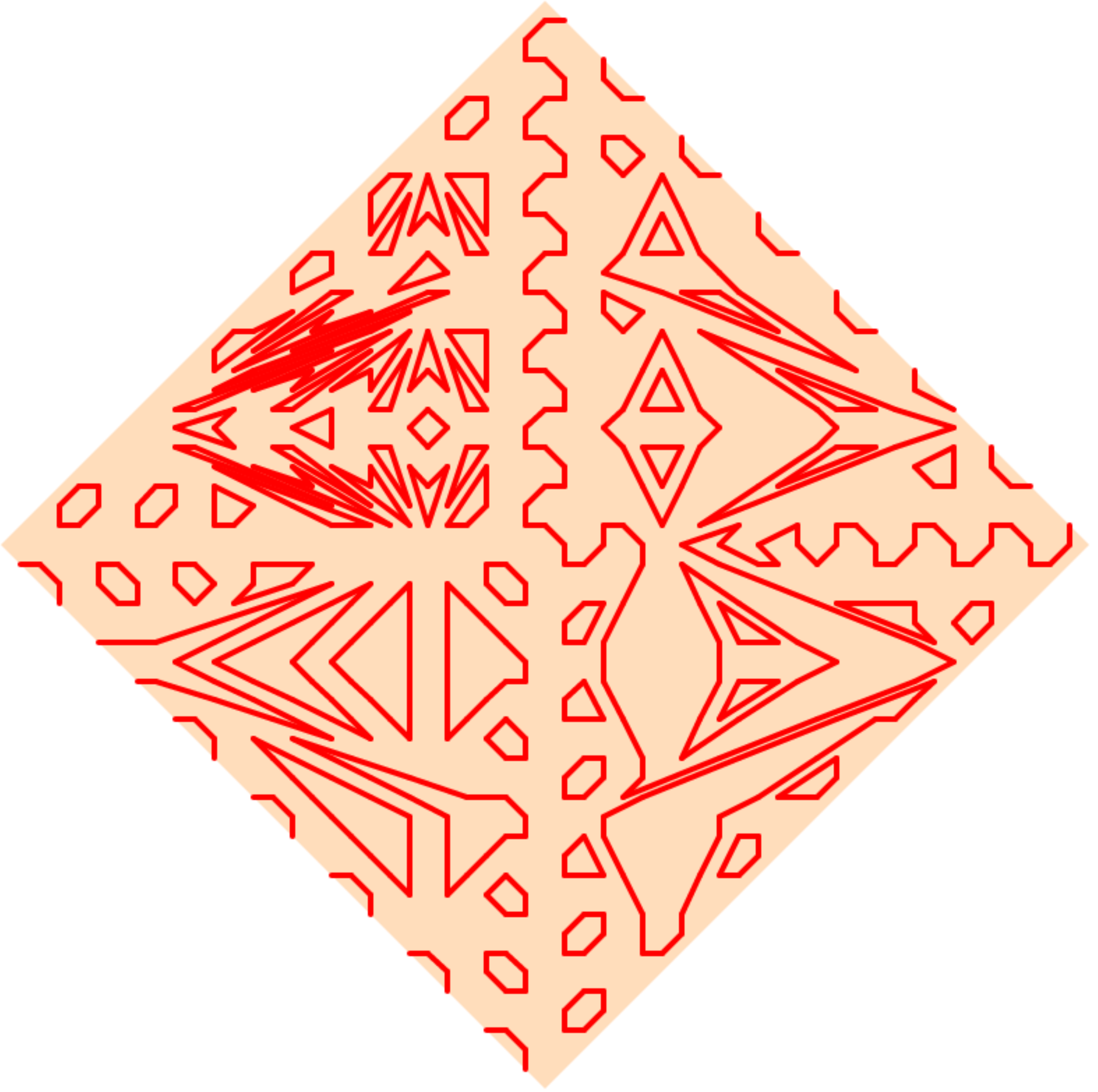}
\caption{Isotopy type of the real part.}
\label{RealPart142Article3}
\end{subfigure}
\caption{Construction in \Cref{ExDeg14}.}
\label{FigExDeg14}
\end{figure}

\begin{example}
\label{ExDeg14}
In degree 14, the construction is the one of \Cref{FigRagsdale14}. 
The number of even ovals of the real part of a non-singular real tropical curve $(C,\E)$ with zone decomposition $\Delta_T$ is $R(7)+5 = 68$, since we are in the case $k=1 \mod 6$.
The isotopy type of the real part is pictured in \Cref{RealPart142Article3}, and has been constructed via the \href{https://math.uniandes.edu.co/~j.rau/patchworking_english/patchworking.html}{Combinatorial Patchworking Tool} of El-Hilany, Rau and Renaudineau (this could also be constructed with the Sage package \href{https://cdoneill.sdsu.edu/viro/}{Viro.sage} from De Wolff, O'Neill and Owusu Kwaakwah).
\end{example}

\bigskip

 \bibliographystyle{alpha}
 \bibliography{bibliomoi1}

\end{document}